\documentclass[a4paper,11pt]{article}

\usepackage{amsthm}
\usepackage{empheq}
\usepackage{mathtools}
\usepackage{bbm}
\usepackage{rotating}
\usepackage[thicklines]{cancel}
\usepackage{epic}
\usepackage[mathcal]{euscript}
\usepackage{amsmath,amssymb}
\usepackage{amsfonts}
\usepackage{fancyhdr}
\usepackage{comment}
\usepackage{times}
\usepackage{amsmath}
\usepackage{framed,lipsum}
\usepackage{changepage}
\usepackage{amssymb}
\usepackage{graphicx}%
\usepackage{xcolor}
\usepackage{multicol}
\usepackage[toc]{appendix}
\usepackage{tikz}
\usetikzlibrary{patterns}
\usepackage{enumitem}
\usepackage{hyperref}
\usepackage{booktabs}
\usepackage{transparent}
\usepackage{standalone}
\usepackage{float}
\usepackage{bbold}
\usepackage{amsthm}
\usepackage{empheq}
\usepackage{mathtools}
\usepackage{rotating}
\usepackage[thicklines]{cancel}
\usepackage{epic}
\usepackage{scalerel}
\usepackage{caption}
\usepackage{geometry} 
\usepackage[toc]{appendix}

 \newcommand{\bfC}{{\boldsymbol C}}
 
 \newcommand{\bfH}{{\boldsymbol H}}
 \newcommand{\bff}{{\boldsymbol f}}
 \newcommand{\bfu}{{\boldsymbol u}}
 \newcommand{\bfuini}{{\bfu_0}}
 \newcommand{\bfv}{{\boldsymbol v}}
 \newcommand{\bfw}{{\boldsymbol w}}
 \newcommand{\bfe}{{\boldsymbol e}}
 \newcommand{\bfn}{\boldsymbol n}
 \newcommand{\bfx}{\boldsymbol x}
 \newcommand{\bfX}{\boldsymbol X}
 \newcommand{\bfvarphi}{{\boldsymbol \varphi}}
 
 \newcommand{\bfnabla}{{\boldsymbol \nabla}}

 \newcommand{\dgammax}{\ \mathrm{d}\gamma(\bfx)}
 \newcommand{\dx}{\ \mathrm{d}\bfx}
 \newcommand{\dt}{\ \mathrm{d} t}

 \newcommand{\ui}{u_i} 
 \newcommand{\vi}{v_i}

 \newcommand{\disc}{{\mathcal D}}
 \newcommand{\mesh}{{\mathcal M}}
 \newcommand{\edge}{{\sigma}}
 \newcommand{\edgeperp}{{\tau}}
 \newcommand{\edges}{{\mathcal E}}
 \newcommand{\edgesK}{\edges(K)}
 
 \newcommand{\edgesint}{{\mathcal E}_{\mathrm{int}}}
 \newcommand{\edgesext}{{\mathcal E}_{\mathrm{ext}}}
 \newcommand{\edgesinti}{{\mathcal E}_{\mathrm{int}}^{(i)}}
 
 \newcommand{\edgesexti}{{\mathcal E}_{\mathrm{ext}}^{(i)}}
 \newcommand{\edgesi}{{\edges\ei}}
 \newcommand{\edgesj}{{\edges\ej}}
 \newcommand{\edged}{\epsilon}
 \newcommand{\edgesd}{{\widetilde {\edges}}}
 \newcommand{\edgesdi}{{\edgesd^{(i)}}}
 \newcommand{\edgesdij}{{\edgesd^{(i,j)}}} 
 \newcommand{\edgesdinti}{{\edgesd^{(i)}_{{\rm int}}}}
 \newcommand{\edgesdexti}{{\edgesd^{(i)}_{{\rm ext}}}}
 \newcommand{\ei}{^{(i)}}
 \newcommand{\ej}{^{(j)}}
 \newcommand{\medge}{\vert \edge \vert}
 
 \newcommand{\deltat}{\delta t}

 \newcommand{\Hmesh}{\bfH_\edges}
 \newcommand{\Hmeshzero}{\bfH_{\edges,0}}

 \newcommand{\Hmeshi}{H_{\edges^{(i)}}}
 \newcommand{\Hmeshizero}{H_{\edges^{(i)},0}}
 \newcommand{\Hmeshzerom}{{\mathbf{H}_{\edges_m,0}}}
 \newcommand{\Hmeshzeromi}{{H_{\edges_m^{(i)},0}}}
 
 \newcommand{\Hmeshdij}{H_{\edgesd^{(i,j)}}}

 \newcommand{\dive}{{\mathrm{div}}}
 \newcommand{\characteristic}{\mathbb{1}}
 \newcommand{\diam}{{\mathrm{diam}}}

 \newcommand{\llbracket}{\bigl[ \hspace{-0.55ex} |}
 \newcommand{\rrbracket}{| \hspace{-0.55ex} \bigr]}
 \DeclareSymbolFont{slenderlargesymbols}{OMX}{ccex}{m}{n}

 \newcommand{\xN}{\mathbb{N}}
 \newcommand{\R}{\mathbb{R}}

 \newcommand{\mnn}{{m\in\xN}}

 \newcommand{\Fcvedge}{F_{K,\sigma}}
 \newcommand{\ucvedge}{u_{K,\sigma}}

 \newcommand{\ds}{\displaystyle}
 \newcommand{\bs}{\boldsymbol}

 \newcommand{\bop}{\vspace{3mm}}

 \renewenvironment{proof}{\paragraph{Proof:\hspace{-5pt}}}{\hfill$\square$}

\newtheorem{theorem}{Theorem}[section]
\newtheorem{lemma}[theorem]{Lemma}
\theoremstyle{definition}
\newtheorem{definition}[theorem]{Definition}
\theoremstyle{remark}
\newtheorem{remark}[theorem]{Remark}
\DeclareMathOperator{\sign}{sign}

\begin{document}

\title{Convergence of the implicit MAC-discretized Navier--Stokes equations with variable density and viscosity on non-uniform grids.}

\title{Convergence of the implicit {MAC} scheme of the incompressible {N}avier-{S}tokes equations with variable density and viscosity}
%
\author{
{\sc
L\'ea Batteux\thanks{Corresponding author. Email: lea.batteux@univ-antilles.fr}
and
Pascal Poullet\thanks{Email: pascal.poullet@univ-antilles.fr}} \\[2pt]
LAMIA, Universit\'e des Antilles, Campus de Fouillole \\[2pt]
BP 250 F-97159 Pointe-\`a-Pitre Guadeloupe\\[6pt]
{\sc and}\\[6pt]
{\sc Thierry Gallou\"et\thanks{Email: thierry.gallouet@univ-amu.fr},
Raphaele Herbin\thanks{Email: raphaele.herbin@univ-amu.fr}}\\[2pt]
I2M UMR 7373, Aix-Marseille Universit\'e, CNRS, Ecole Centrale de Marseille.\\[2pt]
39 rue Joliot Curie. 13453 Marseille, France\\[6pt]
{\sc and}\\[6pt]
{\sc Jean-Claude Latch\'e\thanks{Email: jean-claude.latche@irsn.fr}}\\[2pt]
IRSN, BP 13115, St-Paul-lez-Durance Cedex, France
}
%


%

%
\maketitle

\begin{abstract}
{
The present paper is focused on the proof of the convergence of the
discrete implicit Marker-and-Cell (MAC) scheme for time-dependent
Navier--Stokes equations with variable density and variable viscosity.
The problem is completed with homogeneous Dirichlet boundary
conditions and is discretized according to a non-uniform Cartesian
grid. A priori-estimates on the unknowns are obtained, and along with
a topological degree argument they lead to the existence of a solution
of the discrete scheme at each time step. We conclude with the proof
of the convergence of the scheme toward the continuous problem as mesh
size and time step tend toward zero with the limit of the sequence of
discrete solutions being a solution to the weak formulation of the
problem.
}
{Finite volume methods; MAC scheme; incompressible Navier--Stokes equations;
variable density and viscosity; transport equations.}
\end{abstract}


\section{Introduction}
\[\mathfrak{u} \mathcal{u}\]
We here consider the numerical approximation of the incompressible Navier--Stokes with variable density and viscosity,
\begin{subequations} \label{pb:cont}
\begin{align}\label{Mass} &
\partial_{t} \bar\rho+\dive(\bar\rho\bar\bfu)=0,
\\ \label{qdm} &
 \bar\rho \partial_{t}\bar{ \bfu}+(\bar{\bfu} \cdot \bfnabla)\bar{\bfu} - \dive(\mu(\bar\rho) D(\bar{\bfu})) +\bfnabla \bar p=\bff,
\\ &
\dive \ \bar\bfu=0, \label{inc}
\end{align}
\end{subequations}
in $\Omega\times (0, T)$ where $T\in\R^{+}$ and $\Omega$ being an open bounded
connected subset of $\R^{d}$, with $d\in \{2, 3\}$. For the sake of convenience we suppose that $\Omega$ may be covered
by a structured grid, so  that $\Omega$ is a finite union of rectangles if $d = 2$ and of rectangular parallelepipeds if $d=3$.
We assume that the source term $\bff$ belongs to $L^2(0,T;L^2(\Omega)^d)$.
The variables $\bar\rho$, $\bar\bfu$ and $\bar p$ are respectively the density, the velocity and the pressure of the flow.
The three above equations respectively express the mass conservation, the momentum balance and the
incompressibility of the flow.
This system is supplemented with initial and boundary conditions:
\begin{equation}
 \bar\bfu|_{\partial \Omega}=0, \  \quad \bar\bfu|_{t=0}=\bfu_{0}, \ \quad \bar\rho|_{t=0}=\rho_{0}.
\label{condition}
\end{equation}
These equations model the motion of mixtures of immiscible fluids having different densities and viscosities.
This system is mainly what we obtain when we want to describe the fluid-structure interaction by volume penalization method
\cite{liu-walkington-07, BatteuxTI2018}; 
One can use this approach to represent the interaction of several phases, for example droplet impact onto a solid or a
surface liquid, accurate tracking of interface surfaces between fluids of different density and viscosity in multiphase flows.
Due to the numerous applications that require faithful model of fluid-structure interaction (in computer vision, image processing,
special effects,...) this research area is in demand, both at the theoretical and numerical level.

The existence of weak solutions has been established by Lions \cite{Lions96},
following results concerning the renormalized solution concept of transport equations \cite{DiPL}. Since these first works,
only few results are available on variable viscosity fluids, contrary to the abondant literature on constant viscosity problems.
Certainly, to treat this mixed PDE system entangling hyperbolic, parabolic and elliptic features, is far for obvious;
an overview of theoretical advances can be found in chapter 6 of \cite{BoyerFabrie-book}.
In order to explain our motivation, let us summarize the theoretical results by requiring the following assumptions
on the regularity of the initial data:
\begin{equation}
\rho_{0} \in L^{\infty}(\Omega) \ \mbox{and}\ \bfu_{0}\in L^{2}(\Omega)^{d}. \label{propinit}
\end{equation}
Let us denote $\rho_{ \min}=\mbox{ess}   \inf_{\bfx\in\Omega} \ \rho_{0}(\bfx), \ \rho_{\max}=
\mbox{ess} \sup_{\bfx\in\Omega} \ \rho_{0}(\bfx)$.
A well-known consequence (see for instance \cite{DiPL}) of equations \eqref{Mass} and \eqref{inc}, is the following maximum principle:
\begin{equation}
 \rho_{\min}\leq\bar\rho(\bfx,t)\leq\rho_{\max},\qquad\quad \mbox{for} \ a.e. \ (\bfx,t)\in\Omega\times(0,T),
 \label{maximum}
\end{equation}
which shows that the natural regularity for $\bar\rho$ is $L^{\infty}(\Omega \times (0, T ))$.
For the velocity $\bar\bfu$, a classical calculation allows to derive natural estimates for the solutions.
Let us take for simplicity the right hand side of \eqref{qdm} $\bff=0$ to derive the kinetic energy equation.
One thus takes the inner product of \eqref{qdm} by $\bar\bfu$ and uses twice the mass conservation principle \eqref{Mass} to obtain:
\begin{equation}
 \partial_{t}(\frac 1 2 \bar\rho \vert \bar\bfu \vert^{2})+\dive(\frac 1 2 \bar\rho \vert \bar\bfu\vert^{2}\bar\bfu)
- \dive(\mu(\bar\rho) D(\bar{\bfu})) \cdot\bar\bfu + \nabla \bar p\cdot\bar\bfu=0.
\end{equation}
Integrating over $\Omega$, one gets, since $\dive \ \bar\bfu = 0$ and $\bar\bfu_{|\partial\Omega} = 0$, that,
for all $t\in (0, T )$:
\begin{equation*}
 \frac{d} {dt} \int_{\Omega} \frac 1 2 \bar\rho(\bfx,t) \vert \bar\bfu(\bfx,t)\vert^{2}\dx+\int_{\Omega}
 \mu(\bar\rho) D(\bar\bfu(\bfx,t)):
 D(\bar\bfu(\bfx,t))\dx=0.
\end{equation*}
Integrating over time yields, once again for all $t\in (0, T )$:
\begin{equation*}
 \int_{\Omega}\frac 1 2 \bar\rho(\bfx,t) \vert \bar\bfu(\bfx,t)\vert^{2}\dx
+\int_{0}^{t}\int_{\Omega} \mu(\bar\rho) \vert D(\bar\bfu(\bfx,t))\vert^{2}\dx \dt
=\int_{\Omega}\frac 1 2 \rho_{0}(\bfx) \vert \bfu_{0}(\bfx)\vert^{2}\dx, \ \forall t\in(0,T).
\end{equation*}
The kinetic energy identity above ensures that the natural regularity for $\bar\bfu$ is to lie in
$L^{\infty}((0, T ); L^{2}(\Omega) )\cap L^{2} ((0, T ); H^{1}_{0}(\Omega)^{d} )$.
From this and the regularity on the density we may define the
weak solutions to problem \eqref{pb:cont} as follows:
\begin{definition}\label{def:pbweak}
Let $\rho_{0}\in L^{\infty}(\Omega)$ such that $\rho_{0}>0$ for a.e. $\bfx\in\Omega$, and let
$\bfu_{0}\in L^{2}(\Omega)^{d}$. A pair $(\bar \rho, \bar \bfu)$ is a weak solution of problem \eqref{pb:cont} if it
satisfies the following properties:
\begin{itemize}
 \item $\bar \rho\in\{ \rho\in L^{\infty}(\Omega\times(0,T)), \ \rho>0 \ a.e. \ \mbox{in} \ \Omega\times(0,T)\}$.
 \item $\bar \bfu\in\{ \bfu \in L^{\infty}(0,T;L^{2}(\Omega)^{d})\cap L^{2}(0,T;H^{1}_{0}(\Omega)^{d}),
 \ \dive \ \bfu=0 \ a.e.
 \ \mbox{in} \ \Omega\times(0,T)\}$.
 \item For all $\varphi$ in $C_{c}^{\infty}(\Omega\times[0,T))$,
\begin{equation}
 -\int_{0}^{T}\int_{\Omega} \bar \rho\partial_{t} \varphi+\bar \bfu\cdot\nabla
  \varphi)\dx\dt
  =\int_{\Omega}\rho_{0}(\bfx)\varphi(\bfx,0)\dx.
  \label{massweak}
\end{equation}
 \item For all $\bfv$ in $\displaystyle \{\bfw\in C_{c}^{\infty}(\Omega\times[0,T))^{d}, \dive \ \bfw=0\}$,
\begin{multline}
\displaystyle
  \int_{0}^{T}\int_{\Omega} \bigr[ -\bar \rho\bar \bfu\cdot\partial_{t}\bfv-
  (\bar \rho \bar \bfu\otimes
  \bar \bfu):\nabla\bfv +\mu(\bar\rho)D(\bar\bfu):D(\bfv)\bigl]\dx\dt\\
  =\int_{\Omega}\rho_{0}(\bfx)\bfu_{0}(\bfx)\cdot\bfv(\bfx,0)\dx+\int_{0}^{T}\int_{\Omega}\bff\cdot
  \bfv \dx\dt
  \label{qdmweak}.
\end{multline}
\end{itemize}
\end{definition}

\noindent
Then, the existence a weak solution to the problem \eqref{pb:cont} as given in Definition
\ref{def:pbweak} was proven in \cite{sim-90-non} with a newer version in
\cite{BoyerFabrie-book}. For a weak formulation featuring the pressure, we have the
existence of a solution with $p\in W^{-1,\infty}(0,T;L^2_0(\Omega))$ ($L^2_0(\Omega)$
being the quotient space $L^2(\Omega)/\R$ which gathers the square integrable
functions that differ from a constant). \\\\
While the subject of numerical approximations of the Navier--Stokes equations (NSE) for
incompressible flows with constant density and viscosity has been widely discussed,
very few results are available in the context of our study. Some of them, based on front
tracking techniques, are level set or phase fields methods \cite{MeBenceOsherJCP94,
AxHeNeytMMA2015, SY15}. 
Some other use discontinuous Galerkin method to compute numerical approximations
\cite{liu-walkington-07}. One can also find some attempts based on fractional time step
approach/projection method \cite{GuermondQuartapelle00, GS09, LMGS13, AxHeNeytMMA2015,
ChaChuDubois2018}.

In what follows, the continuous problem is approximated using a time implicit scheme
combined with the MAC scheme, providing a sequence of discrete solutions
will be shown to converge toward the solution of \eqref{def:pbweak} without
any regularity assumption on the solution. This MAC scheme is one of the most
well-known methods used for the approximation of both the compressible and
incompressible Navier--Stokes equations. It was introduced in \cite{HarlowWelch65}
for the approximation of free-surface problems, however
through numerous papers it appeared to possess remarkable mathematical and physical
properties. Among them is the inf--sup stability 
as first demonstrated in a finite difference context with staggered grids in
\cite{shinstrikwerda97}, and more recently for the generalized class of DDFV methods
\cite{boyerddfv15}. In finite volume context, this spatial approximation scheme also
allows a local conservation of the mass fluxes. Let us note also the recent work of
\cite{MN20} who have shown unconditional and uniform asymptotic stabilities in the
low Mach number regime of the implicit MAC scheme to solve the compressible NSE. \\
Studies of the MAC method applied to linear problems are quite present in the
litterature; in the case of convergence results for the Stokes equations we may
cite \cite{blanc05} where the source term is in $H^{-1}(\Omega)$ and \cite{lisun15}
for a superconvergence proof on non-uniform grids given regularity assumptions on the
velocity and pressure. Some of the studies of the MAC scheme approximation of the
Navier--Stokes equations include the pioneering error analysis in \cite{nicolaides96},
as well as convergence proofs on locally refined grids in \cite{chenier2015},
and non-uniform grids in \cite{conv-mac-NS-18} for the steady and time dependent
Navier--Stokes equations. \\
The goal of this paper is to extend the results from \cite{conv-mac-NS-18} to
the case of a incompressible flow with variable density and variable viscosity
as defined in \eqref{pb:cont}. This problem was partially adressed in
\cite{latche-saleh-17} in the finite-element context with Rannacher-Turek elements,
as the viscosity was considered as constant. It is known that for discontinuous Galerkin
or general finite-element context some stabilization terms occur in the diffusion term
with the density-dependent viscosity \cite{liu-walkington-07,latche-saleh-17}.
One shows that is not the case with the MAC scheme by proving a discrete Korn
inequality in order to control the $L^{2}(0,T;H^{1}_{0}(\Omega)$-norm of the
velocity. Moreover, the staggered MAC scheme ensures, by construction, the control
of the discrete kinetic energy and thus the stability of the scheme
\cite{latche-saleh-17}.

Following what has been introduced before, only a weak solution of mass conservation
exists, then building piecewise constant approximate density that has
low-regularity and has to be strictly positive generates a lack of classical
consistency property. The key point in this work is to prove that
the implicit MAC scheme is weakly consistent in the Lax--Wendroff sense for
incompressible flows with variable density and viscosity. As it has previously
done for constant viscosity flows with staggered finite element context
\cite{latche-saleh-17}, we build a sequence of solutions based on a sequence
of discretizations in time and space. Estimates of the solution are derived
allowing us to use compactness arguments (by generalized Aubin-Lions Theorem to
dicrete derivative and sequence of subspaces \cite{conv-mac-NS-18})
to prove the existence of converging sequences of solutions.

According to the Lax-Wendroff argument, the limit of
the discrete solutions converges (up to a subsequence) toward
the weak solution of problem \eqref{pb:cont}.

In the Section 2, one defines the non-uniform staggered MAC grids.
Section 3 gathers the discrete functional spaces, the discrete operators and
unknowns leading to the definition of the discrete scheme. It also recalls
results and tools from \cite{conv-mac-NS-18} needed in the next
sections. In Section 4 we derive a $L^2(H^1_0)\cap L^{\infty}(L^2)$
estimate of the velocity, a non--uniform $L^2$ bound on the pressure and
the classic $L^{\infty}(L^{\infty})$ bound on the density. Those
estimates are required for the topological degree argument used to prove
the existence of a solution of the discrete problem at every time
step. Finally we resort to compactness results in Section 5 to pass to
the limit in the momentum equation and mass balance equation as the
mesh size and time step tend to $0$. We can conclude that the sequence
of discrete solutions converges (up to a subsequence) toward the weak
solution of problem \eqref{pb:cont}.

\section{MAC discretization}
\label{sec;MAC-def}

Following the introduction, the MAC grid is recalled in the present section
 using so as to allow the self-readability of this paper.
Let us assume that $\overline{\Omega}$ is given by the union of mutually
disjoint rectangles if $d=2$, and rectangular parallelepipeds if $d=3$.
Additionally, let the edges (respectively the faces) of those rectangles
(respectively parallelepipeds) be orthogonal to $(\bfe_1, \ldots, \bfe_d)$,
the canonical basis vectors of $\R^d$.

\begin{definition}[MAC grid]
\label{def:MACgrid}
A mesh $\mathcal{D}$ associated to the MAC discretization of $\Omega$, is
defined by $\mathcal{D} = (\mesh, \edges)$, where:\\

\begin{list}{-}{\itemsep=0.ex \topsep=0.5ex \leftmargin=1.cm \labelwidth=0.7cm \labelsep=0.3cm \itemindent=0.cm}

  \item $\mesh$ : refers to the primal (or pressure) grid, and consists in a
  conforming structured partition of $\Omega$ made of rectangles if $d=2$ or
  rectangular parallelepipeds if $d=3$. A generic element $K$ of $\mesh$ designates
  a primal cell and we note $\bfx_K$ its center of mass. Therefore, we have
  $\overline{\displaystyle{\cup_{K \in \mesh}} K} = \overline{\Omega}$.
  The set of edges (or faces) of $K$ is denoted $\edgesK \subset \R^{d-1}$.
  For an element $\edge \in \edgesK$ in the boundary of $K$, we note
  $\bfx_\edge$ its center of mass.
\item $\edges$ : is the $d$-uplet of dual grids associated to the $d$
components of the velocity. It is defined as the set of every edges (or faces)
of $\mesh$ : $\edges =\{ \edge \in \edges (K) |  K \in \mesh\}$.  We note
$\edges= \edgesint \cup \edgesext$, where $\edgesint$ (resp. $\edgesext$) are
the edges of $\edges$ that lie in the interior (resp. on the boundary) of the
domain. For any $i \in \llbracket 1,d\rrbracket$ the set $\edgesi$ contains the
edges  that are orthogonal to $\bfe_{i}$. Similarly, we define
$\edgesinti = \edgesi \cap \edgesint$ and $\edgesexti = \edgesi \cap \edgesext$
to obtain $\edgesi= \edgesinti \cup \edgesexti$.\\

Let $\edge \in \edgesint$, we note $\edge = K \vert L$ if $(K,L)\in \mesh^2$
are such that $\edge = \partial K \cap \partial L$. Additionally, let $D_{\edge}$
refer to the dual cell associated to $\edge$. We have $D_{\edge} = D_{K, \edge}\cup D_{L,\edge}$
where $D_{K,\edge}$ (resp. $D_{L,\edge}$) denotes the half of cell $K$ (resp. $L$)
adjacent to $\edge$. For an element $\edge \in \edgesext$ adjacent to $K \in \mesh$,
we note $D_\edge=D_{K,\edge}$. Thus, we end up with $d$ partitions of $\Omega$
: $\overline{\Omega} = \overline{\bigcup_{\edge \in \edges_i} D_{\edge}}$ for
$ i=1,\dots, d$.
\end{list}
\end{definition}~\\
In order to deal with the discrete momentum equation later on,
we introduce the faces of the $i-$th dual mesh, which we denote $\edgesdi$. Let
us distinguish once more the internal elements of $\edgesdi$ from the external elements
by writing $\edgesdi =\edgesdinti \cup \edgesdexti$. Thus, for any element
$\edged\in \edgesdinti$ we note $\edged = \edge \vert \edge '$ if  $(\edge, \edge ')\in \edges_i^2$
are such that $\partial D_{\edge} \cap \partial D_{\edge '} = \edged$.\\\\

\begin{figure}[!h]
  \begin{center}
  \begin{tikzpicture}[scale=1]
  \draw[-, red!20, pattern=north west lines, pattern color=red!30, opacity=0.8]
  (1.5,0)--(2.5,0)--(2.5,2)--(1.5,2) -- (1.5,0); 
  \draw[-, blue!20, pattern=north west lines, pattern color=blue!30, opacity=0.8]
  (0.5,1)--(2.5,1)--(2.5,3)--(0.5,3); \path(0.9,2.7) node[blue!80]{$D_\edge$};
  \path(2,0.4) node[red!90]{$D_{K,\edge'}$};
  \draw[-](0,0)--(4.5,0);
  \draw[-](0,2)--(4.5,2);
  \draw[-](0,4)--(4.5,4);
  \draw[-](0.5,-0.5)--(0.5,4.5);
  \draw[-](2.5,-0.5)--(2.5,4.5);
  \draw[-](4,-0.5)--(4,4.5); 
  \draw[-, very thick, blue!50] (0.5,2)--(2.5,2);
  \path(1.9,2.2) node[blue] {\scriptsize $\edge = K|L$}; 
  \path(0.75,3.75) node[black] {$L$};
  \path(0.75,0.35) node[black] {$K$};
  \path(1.5,1) node[black] {$+$};
  \path(1.5,1.35) node[black] {$\bfx_K$};
  \draw[-, very thick, red!50] (2.5,0)--(2.5,2);
  \path(2.75,0.25) node[red,rotate=90] {\footnotesize	 $\edge'$}; 
  \end{tikzpicture}\hspace{15pt}
  \begin{tikzpicture}[scale=0.75]
  \draw[-, blue!8, fill=blue!8] (0.5,2.5)--(5,2.5)--(5,0.0)--(0.5,0.0);
  \draw[-](0,0.5)--(5.5,0.5);
  \draw[-](0,4.5)--(5.5,4.5);
  \draw[-](0.5,0)--(0.4,5);
  \draw[-](5,0)--(5,5);
  \draw[-, very thick, blue!50] (0.45,0.5)--(5,0.5);
  \path(4.15,0.75) node[blue] {\small $\edge = K|L$}; 
  \draw[-, very thick, blue!50] (0.45,4.5)--(5,4.5);
  \path(4.5,4.2) node[blue] {\small $\edge'$}; 
  \path(2.75,4.75) node[blue] {\small $\bfu_{\edge'}\cdot \bfe_2$};
  \path(2.75,4.5) node[blue] {\footnotesize $\times$};
  \path(2.75,0.25) node[blue] {\small $\bfu_{\edge}\cdot \bfe_2$};
  \path(2.75,0.5) node[blue] {\footnotesize $\times$};
  \draw[-, very thick, gray!50] (0.45,2.5)--(5,2.5);
  \path(4.15,2.75) node[gray] {\small $\edged = \edge|\edge'$}; 
  \path(1.5,1.75) node[blue] {\large $D_{K,\edge}$}; 
  \path(2.75,2.25) node[black] {$\rho_K,p_{K}$};
  \path(2.75,2.5) node[black] {\footnotesize $\times$};
  \end{tikzpicture}
  \caption{Representation of $(\mesh,\edges)$ for $d=2$}
  \end{center}
  \label{fig:mesh}
\end{figure}

\noindent Let us define $d\times d$ partitions of $\Omega$ by associating the
$(i,j)$--partition to the elements of $\edgesdi$ that are orthogonal $\bfe_j$ for
$i,j \in \llbracket 1,d\rrbracket$ :

\begin{equation}\label{eq:defcelldualdual}
  (D_\edged)_{\edged \in \edgesdi, \edged \perp \bfe_j} =
  \left\{ \begin{array}{l l }
    \edged \times [\bfx_{\edge}, \bfx_{\edge'}] &\qquad \mbox{ for }\ \edged = \edge | \edge ' \in \edgesdinti \\
    \edged \times [\bfx_{\edge}, \bfx_{\edge,\edged}] &\qquad \mbox{ for }\ \edged \in \edgesdexti \cap \edgesd( D_\edge)
  \end{array} \right.
\end{equation}~\\
where $\bfx_{\edge,\edged}$  refers to the orthogonal projection of $\bfx_{\edge}$
on $\edged$. Hence the $(i,j)$--partition
coincides with the $(j,i)$--partition, and the $(i,i)$--partition is none other
than $\mesh$. For the sake of convenience, we introduce the following subsets of
$\edgesdi$:
\begin{equation*}
\edgesdij = \{ \edged \in \edgesdi, \edged \perp \bfe_j \}
\end{equation*}
\noindent Finally, the constants defined below aim to characterize the non--uniformity
of the space discretization of the MAC grid
$\mathcal{D}$ :
\begin{equation}\label{regmesh}
  \begin{array}{l }
    \eta_\mesh = \max\left\{  \dfrac{\vert \edge \vert}{\vert \edge' \vert}, \ (\edge,\edge') \in \edgesi\times\edgesj,\
    i, j \in \llbracket 1, d\rrbracket, i\not = j \right\} \\\\
    h_\mesh=\max\{\diam(K),  K\in\mesh\}
  \end{array}
\end{equation}~\\
with $|\cdot|$ designating the Lebesgue measure of either $\R^d$ or $\R^{d-1}$.

\begin{figure}[htb]
  \centering
  \begin{center}
    \begin{tikzpicture}[scale=0.5]

      \draw [blue!5, fill=blue!5]  (0,0) rectangle (4,4);
      \path(0.5,0.5) node[blue] {\small $D_\edged $};

      \draw (0,0) rectangle (4,4); \draw (0,4) rectangle (4,8);
      \draw (4,0) rectangle (7,4); \draw (4,4) rectangle (7,8);

      \path(2,2) node[black] {\small  $+$}; \path(5.5,2) node[black] {\small  $+$};
      \path(2,6) node[black] {\small  $+$}; \path(5.5,6) node[black] {\small  $+$};

      \path(4,2) node[black] {\small  $\times$};
      \path(3.6,1.5) node[black] {\small  $u_\edge$};

      \path(0,2) node[black] {\small  $\times$};
      \path(-0.4,1.5) node[black] {\small  $u_\edge'$};

      \draw [blue, line width = 1.5] (2,0) -- (2,4); \path(2.4,3.6) node[blue] {\small $\edged $};

      \draw (0,8) -- (0,8.5); \draw (4,8) -- (4,8.5); \draw (7,8) -- (7,8.5);
      \draw (-0.5,8) -- (0,8); \draw (-0.5,4) -- (0,4); \draw (-0.5,0) -- (0,0);

    \end{tikzpicture} \hspace{10pt}
    \begin{tikzpicture}[scale=0.5]

      \draw [blue!5, fill=blue!5]  (2,2) rectangle (5.5,6);
      \path(2.5,2.4) node[blue] {\small $D_\edged $};

      \draw (0,0) rectangle (4,4); \draw (0,4) rectangle (4,8);
      \draw (4,0) rectangle (7,4); \draw (4,4) rectangle (7,8);

      \path(2,2) node[black] {\small  $+$}; \path(5.5,2) node[black] {\small  $+$};
      \path(2,6) node[black] {\small  $+$}; \path(5.5,6) node[black] {\small  $+$};

      \path(4,2) node[black] {\small  $\times$};
      \path(3.6,1.5) node[black] {\small  $u_\edge$};

      \path(4,6) node[black] {\small  $\times$};
      \path(3.6,6.8) node[black] {\small  $u_\edge'$};

      \draw [blue, line width = 1.5] (2,4) -- (5.5,4); \path(2.4,4.4) node[blue] {\small $\edged $};

      \draw (0,8) -- (0,8.5); \draw (4,8) -- (4,8.5); \draw (7,8) -- (7,8.5);
      \draw (-0.5,8) -- (0,8); \draw (-0.5,4) -- (0,4); \draw (-0.5,0) -- (0,0);

    \end{tikzpicture}
    \captionsetup{justification=centering,margin=1cm}
    \caption{Cells $D_\edged$ associated to elements of  $\edgesdij$ with $\edged = \edge | \edge' \in \edgesd_{int}^{(1)}$:\\
      for $\edged \perp \bfe_1$ (left) ; for $\edged \perp \bfe_2$ (right) ; ($d = 2$)}
    \label{fig:dualdualint}  \end{center} \end{figure}
\begin{figure}[htb]
  \centering

  \begin{tikzpicture}[scale=0.5]

    \draw [blue!5, fill=blue!5]  (2,0) rectangle (5.5,2);
    \path(5,0.6) node[blue] {\small $D_\edged $};

    \draw [blue!5, fill=blue!5]  (8.5,2) rectangle (10,6);
    \path(9.3,2.7) node[blue] {\small $D_{\edged'} $};

    \draw (0,0) rectangle (4,4); \draw (0,4) rectangle (4,8);
    \draw (4,0) rectangle (7,4); \draw (4,4) rectangle (7,8);
    \draw (7,0) rectangle (10,4); \draw (7,4) rectangle (10,8);

    \path(2,2) node[black] {\small  $+$}; \path(5.5,2) node[black] {\small  $+$} ; \path(8.5,2) node[black] {\small  $+$};
    \path(2,6) node[black] {\small  $+$}; \path(5.5,6) node[black] {\small  $+$} ; \path(8.5,6) node[black] {\small  $+$};

    \path(4,2) node[black] {\small  $\times$};
    \path(3.6,1.5) node[black] {\small  $u_\edge$};

    \path(8.5,4) node[black] {\small  $\times$};
    \path(7.8,4.6) node[black] {\small  $u_\edge'$};

    \draw [blue, line width = 1.5] (2,0) -- (5.5,0); \path(2.4,0.4) node[blue] {\small $\edged $};

    \draw [blue, line width = 1.5] (10,2) -- (10,6); \path(9.6,5.7) node[blue] {\small $\edged' $};

    \draw (0,8) -- (0,8.5); \draw (4,8) -- (4,8.5); \draw (7,8) -- (7,8.5); ; \draw (10,8) -- (10,8.5);
    \draw (-0.5,8) -- (0,8); \draw (-0.5,4) -- (0,4); \draw (-0.5,0) -- (0,0);

    \path(4,0) node[black] {\small  $\times$};
    \path(3.6,-0.5) node[black] {\small  $\bfx_{\edge,\edged}$};

    \path(10,4) node[black] {\small  $\times$};
    \path(11,4.5) node[black] {\small  $\bfx_{\edge',\edged'}$};

  \end{tikzpicture}
  \captionsetup{justification=centering,margin=1cm}
  \caption{Cells $D_\edged$ and $D_{\edged'}$ associated respectively  to\\
    $\edged = \edgesd_{ext}^{(1)} \cap \edgesd( D_\edge)$ and $\edged' = \edgesd_{ext}^{(2)} \cap \edgesd( D_\edge')$ ; ($d = 2$)}
  \label{fig:dualdualext}
\end{figure}
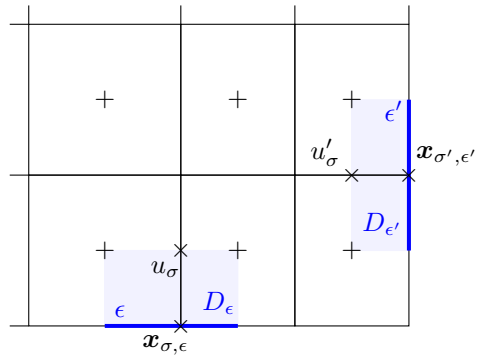

\section{Discrete scheme}

In the present section we introduce the discrete scheme associated with
the continuous problem \eqref{pb:cont} using a finite volumes formalism.
In order to do so let us define the discrete unknowns, the functional spaces
they lie in as well as the approximation of the operators in \eqref{pb:cont}.

\subsection{Discrete spaces and unknowns}

For the time discretization, we consider an uniform partition of the time interval
$[0,T]$ with $T > 0$. By denoting $\deltat$ the constant time step, the partition
of $[0,T]$ is given by $\{t^n = n \deltat, n = 0, .. , N\}$ with $N$ satisfying
$T = N\deltat$. For all $n \in \{0,..,N\}$, we proceed with a staggered space
discretization on $\mathcal{D} = (\mesh, \edges)$ in the following sense: the
degrees of freedom for the density and pressure are associated to the primal
cells $K \in \mesh$, whereas those associated to the $i$--th component of the
velocity are located on the faces $\edgesi$. Hence  we are led to handle piecewise
constant functions on the cells $K \in \mesh\,$ (respectively the cells $D_\edge$
with $\edge\in\edges$) and their associated spaces:

\begin{equation}\label{lmeshhmesh}
  L_{\mesh} =\Bigl\{ q = \displaystyle \sum_{K \in \mesh} q_K\characteristic_{K},\, q_K \in \R \Bigr\},\ \ \mbox{and}\ \
  \Hmeshi  =\Bigl\{ v = \displaystyle \sum_{\edge \in \edgesi} v_{\edge}\characteristic_{D_\edge},\, v_{\edge} \in \R \Bigr\},
  \ \  i\in \llbracket 1, d\rrbracket,
\end{equation}

with $\characteristic_K$ and $\characteristic_{D_\edge}$ referring to the characteristic
functions of the cells $K$  and $D_\edge, \edge \in \edges$ respectively:

 \[
  \characteristic_K(\bfx) = \begin{cases}
                  &1  \text{ if } \bfx \in K,\\ &0  \text{ if } \bfx \not \in K,\\
                 \end{cases}
 \quad  \quad \mbox{and }\quad \quad \characteristic_{D_\sigma}(\bfx) = \begin{cases}
                   &1  \text{ if } \bfx \in D_\edge,\\ &0  \text{ if } \bfx \not \in D_\edge.
                 \end{cases}
 \]\\
For any element $q$ in $L_{\mesh}$ we will frequently resort to the notation
$q = (q_K)_{K\in \mesh} \in \R^{N_\mesh}$ with $N_\mesh = \text{card}(\mesh)$.
By analogy, we note $v = (v_{\edge})_{\edge \in \edgesi}$ for any element $ v \in \Hmeshi$. \\\\
As a consequence, an approximation of the velocity field $\bfu = (u_1, .., u_d)^T$
will belong to  $\Hmesh =  \prod_{i=1}^d  \Hmeshi$.  Similarly to the continuous case,
the Dirichlet boundary conditions are incorporated in the definition of the
velocity spaces and, for this purpose, we introduce $\Hmeshizero \subset \Hmeshi, i=1,\ldots,d$,
defined as follows:
\begin{equation}\nonumber
  \Hmeshizero=\Bigl\{u\in\Hmeshi, \, \ u(\bfx)=0\ \,  \forall \bfx \in D_{\edge},
  \,  \ \edge \in \edgesdexti \displaystyle\Bigr\},\quad \quad  i\in \llbracket 1, d\rrbracket.
\end{equation}
along with,
\begin{equation}\nonumber
\Hmeshzero =  \prod_{i=1}^d  \Hmeshizero \subset \Hmesh
\end{equation}\\

The continuous pressure unknown fulfills a role as a Lagrange multiplier of the
incompressibilty constraint \eqref{inc} and needs
to be defined up to a constant. As such, an approximation of the pressure will lie
in the discrete counterpart of $L^2_0(\Omega)$, denoted by $L_{\mesh,0}$ and designating
the space of functions of  $L_{\mesh}$ with zero mean value,
\begin{equation}\label{lmeshzero}
  L_{\mesh,0}=\Bigl\{ q = (q_K)_{K\in \mesh} \in L_{\mesh} , \quad \displaystyle \sum_{K \in \mesh} |K| q_{K} = 0 \displaystyle\Bigr\}.
\end{equation}~

\noindent Finally, the staggered discrete scheme is given by,

\hspace{-10pt}\begin{subequations} \label{eq:scheme:ro}
\begin{align} \nonumber &
\mbox{Let }  \bfu^0 = \Pi_\edges \bfu_0\in \Hmeshzero \quad  \mbox{and} \quad \rho^0 = \Pi_\mesh \rho_0\in L_\mesh
\\ \nonumber &
\\ \nonumber & \mbox{For }  n\in \{ 0, \cdots, N-1\}
\\[1ex] \nonumber & \hspace{2ex}
 \mbox{ Find } (\rho^{n+1},\bfu^{n+1},p^{n+1})\in  L_{\mesh}\times \Hmeshzero \times  L_{\mesh,0} \quad \mbox{satisfying :}
\\[1ex] \label{eq:mass-rv} & \hspace{7ex}
\eth_{t}\rho^{n+1}+\dive_{\mesh}(\rho^{n+1}\bfu^{n+1})=0
\\[1ex] \label{schemevitesse} & \hspace{7ex}
\eth_{t}(\rho \bfu)^{n+1}+  {\bfC}_{\edges} (\rho^{n+1} \bfu^{n+1})\ \bfu^{n+1}
- \dive_\edges(\mu^{n+1}_\edgesd\mathbf{D}_\edgesd(\bfu^{n+1})) + \nabla_{\edges} \ p^{n+1}   =\bff_{\edges}^{n+1}
\\[1ex] \label{schemediv} & \hspace{7ex}
\dive_{\mesh}\bfu^{n+1} = 0,
\end{align}
\end{subequations}~\\
where \eqref{eq:mass-rv}, \eqref{schemevitesse}, \eqref{schemediv} are the
approximations at time $t^{n+1}$ of \eqref{Mass}, \eqref{qdm} and \eqref{inc} respectively for
which we give the details below. The discrete initial data $(\rho^0,\bfu^0)$ is
obtained by projection of $(\rho_0,\bfu_0)$ on $L_{\mesh}$ and $\Hmesh$ using
the following operators:

\begin{align} \label{eq:projdual}
  & \quad \quad \begin{array}{l|l} \displaystyle \Pi_\edges^{(i)}: \quad & \quad  L^1(\Omega) \longrightarrow \Hmeshizero \\ [1ex]
    & \displaystyle \quad v_i \mapsto \Pi_\edges^{(i)}v_i = \sum_{\edge\in\edgesi}
    \Big(\frac{1}{|D_\edge|}\int_{D_\edge} v_i(\bfx) \dx \Big) \characteristic_{D_\edge} \qquad i = 1, \dots, d
\end{array}
\end{align}
\begin{align} \label{eq:projprim}  &
\begin{array}{l|l} \displaystyle
  \Pi_\mesh: \quad & \quad  L^1(\Omega) \longrightarrow L_\mesh \\ [1ex]
  & \displaystyle \quad q \mapsto \Pi_\mesh q = \sum_{K\in\mesh} \Big(\frac{1}{|K|}\int_{K} q(\bfx) \dx)\Big)
  \characteristic_{K}\hphantom{ \qquad i = 1, \dots, d}
\end{array}
\end{align}\\~\\
Similarly, the discrete source term in \eqref{schemevitesse} is defined as $\bff_{\edges}^{n+1} = \Pi_\edges(\bff(t^{n+1},\cdot))$.
\begin{remark}
Operators $\Pi_\mesh$ and $(\Pi_\edges^{(i)})_{i=1,\dots,d}$ are linear and
continuous from $L^p(\Omega)$ into $L^p(\Omega)$. For instance, we have for
$\Pi_\mesh$:
\begin{equation*}
  \| \Pi_\mesh q \|^p_{L^p(\Omega)} \leq  \sum_{K\in\mesh} |K|^{1-p}\left| \int_{K} q(\bfx) \dx \right|^p
  \leq  \sum_{K\in\mesh} |K|^{1-p} |K|^{p-1} \int_{K} \left| q(\bfx) \right|^p \dx =  \| q \|^p_{L^p(\Omega)}
\end{equation*}
\end{remark}
By considering the complete partition of the time interval $[0,T]$, the discrete
functions approaching the solution  of \eqref{pb:cont} are given by:
\begin{align}\label{solupiecewise}
\begin{array}{l l l} \displaystyle
  u_{i}(t,\bfx) & = \ds \sum_{n= 0}^{N-1} \sum_{\edge \in \edgesinti} u_{\edge}^{n+1}\characteristic_{D_\edge}(\bfx)\characteristic_{]t_n, t_{n+1}]}(t),
   \quad  i\in \llbracket 1, d\rrbracket  \vspace{3pt}\\
p(t,\bfx) & = \ds \sum_{n= 0}^{N-1} \sum_{K \in \mesh}p_{K}^{n+1} \characteristic_K(\bfx)\characteristic_{]t_n, t_{n+1}]}(t)\vspace{3pt}\\
\rho(t,\bfx) & = \ds \sum_{n= 0}^{N-1} \sum_{K \in \mesh} \rho_{K}^{n+1} \characteristic_K (\bfx)\characteristic_{]t_n, t_{n+1}]}(t),
\end{array}
\end{align}~\\
where $\characteristic_{]t_n, t_{n+1}]}$ designates the characteristic function
of $]t_n, t_{n+1}]$.\\

\noindent As we aim to study the convergence of the weak discrete problem towards its continuous analogue \eqref{def:pbweak},
let us introduce the weak formulation of \eqref{eq:scheme:ro} :

\begin{subequations}
  \begin{align} \nonumber \\
    \nonumber & \mbox{For }  n\in \{ 0, \cdots, N-1\} \\[1ex]
    \nonumber & \hspace{2ex}
    \mbox{Find} (\rho^{n+1},\bfu^{n+1}) \in L_{\mesh}\times\boldsymbol{E}_\edges \,\,\,
    \mbox{such that for any} \,\,\, (q,\bfv) \in L_{\mesh}\times\boldsymbol{E}_{\edges}\, \mbox{,} \\
    \nonumber\\
    \label{mass:weak} & \hspace{5ex} \int_\Omega \Bigl(\eth_{t} \rho^{n+1} + \dive_{\mesh}(\rho^{n+1}\bfu^{n+1}) \Bigr) q \dx=0,\\
    \label{qdm:weak} & \hspace{5ex}
 \int_\Omega \eth_t(\rho\bfu)^{n+1} \cdot \bfv \dx + b_\edges(\rho^{n+1}\bfu^{n+1}, \bfu^{n+1}, \bfv)
 + \int_\Omega \mu^{n+1}_\edgesd \mathbf{D}_\edgesd(\bfu^{n+1}) : \mathbf{D}_\edgesd(\bfv) \dx \nonumber\\
 &\hspace{5ex}\qquad\qquad\qquad\qquad\qquad\qquad\qquad\qquad\qquad\qquad\qquad\qquad\qquad\qquad
 = \int_\Omega \bff_{\edges}^{n+1} \cdot \bfv \dx ,
\end{align}
\end{subequations}
where $b_\edges$ designates the discrete trilinear form associated to the convection term in \eqref{schemevitesse}
and $\boldsymbol{E}_\edges$ stands for piecewise constant functions of $\Hmeshzero$ that are discrete divergence--free
(see below). Although we will eventually come back to the formulation above for the convergence result,
let us introduce the weak formulation of the problem featuring the pressure:
\hspace{-10pt}\begin{subequations}
  \begin{align} \nonumber \\
    \nonumber & \mbox{For }  n\in \{ 0, \cdots, N-1\}\\[1ex]
    \nonumber & \hspace{2ex}
    \mbox{ Find } (\rho^{n+1},\bfu^{n+1},p^{n+1}) \in L_{\mesh}\times\boldsymbol{E}_\edges\times L_{\mesh,0}  \,\,\,
    \mbox{such that for any} \,\,\, (q,\bfv) \in L_{\mesh}\times\Hmesh\, \mbox{,}\\[1ex]
    \nonumber & \hspace{2ex} \mbox{ Equation \eqref{mass:weak} holds and,}\\
    \label{qdm:weakp} & \hspace{5ex}
 \int_\Omega \eth_t(\rho\bfu)^{n+1} \cdot \bfv \dx + b_\edges(\rho^{n+1}\bfu^{n+1}, \bfu^{n+1}, \bfv)
 + \int_\Omega \mu^{n+1}_\edgesd \mathbf{D}_\edgesd(\bfu^{n+1}) : \mathbf{D}_\edgesd(\bfv) \dx\\
 \nonumber & \hspace{5ex} \displaystyle
 \quad \quad  \quad  \quad  \quad  \quad  \quad  \quad  \quad  \quad  \quad  \quad  \quad
 + \int_\Omega \nabla_\edges p^{n+1} \cdot \bfv \dx
 = \int_\Omega \bff_{\edges}^{n+1} \cdot \bfv \dx.
\end{align}
\end{subequations}

Unlike the former weak formulation of the scheme \eqref{mass:weak}-\eqref{qdm:weak},
this formulation yields the same algebraic equations as the "strong" form of the scheme
\eqref{eq:scheme:ro}.

\subsection{Discrete operators}

Along with the divergence formula, we resort to the following local conservativity
property for the approximation of \eqref{pb:cont}: for any given edge
$\edge = K|L \in \edgesi$, $i\in \llbracket 1,d \rrbracket$, let $\bfn_{K,\edge}$
stand for the unit normal vector to $\edge$ outward $K$, and $\ucvedge$ be defined as
$\ucvedge = u_\edge\, \bfn_{K,\edge} \cdot \bfe_i$. With this definition, we then
have the usual finite volume property of numerical flux local conservativity, through
any primal face:
\begin{equation}\label{eq:conserv}
  \ucvedge =  - u_{L,\sigma}\quad  \text{for}\ \edge = K|L\in \edgesint.
\end{equation}
and a similar property holds for the fluxes through the elements $\edged \in \edgesd$. \\

We mainly focus on the discretization of differential operators for the space
variable. Regarding the time derivatives, special consideration will be taken for
the approximation of $\eth_{t}(\rho \bfu)^{n+1}$ in \eqref{schemevitesse} to yield
the consistency with the discrete mass balance on the dual cells. The latter is
primordial in order to obtain a kinetic energy balance that is analogous to the
continuous case \cite{conv-mac-NS-18}.\\

The continuous mass balance equation \eqref{Mass} is first discretized on $\mesh$
by homogeneity with the discrete density. Hence, for any $K \in \mesh$, we have:

\begin{equation}\label{eq:divdiv}
(\dive_\mesh (\rho^{n+1}\bfu^{n+1}))_K =\ds \frac 1 {|K|} \sum_{\edge\in\edges(K)} \Fcvedge^{n+1},
\end{equation}
where $\Fcvedge^{n+1} = \medge\, \rho^{n+1}_\edge \ucvedge^{n+1}$ for
$K \in \mesh,\ \edge \in \edgesK$ stands for the mass flux across $\edge$ outward
$K$ at time $t^{n+1}$. The value of the density at the face $\edge$ is approximated
by $\rho_{\edge}$ using an upwind scheme,
\begin{align} \label{eq:divflux}  &
\rho^{n+1}_{\edge}=
\begin{cases}
 \rho^{n+1}_{K}   \text{ if } \ \ucvedge^{n+1}\geq 0,\\
 \rho^{n+1}_{L}   \text{ otherwise}.
\end{cases}
\end{align}
\paragraph{Discrete divergence operator $\dive_\mesh$ :}~\\
Additionally, the discretization \eqref{schemediv} of the incompressibility
condition \eqref{inc} is straightforward,

\begin{align} \label{discdiv-u}  &
\begin{array}{l|l} \displaystyle
  \dive_\mesh: \quad & \Hmesh \longrightarrow L_{\mesh}\\ [1ex]
  & \displaystyle \quad \bfu \mapsto \dive_\mesh \bfu = \sum_{K\in\mesh}
  \frac 1 {|K|} \sum_{\edge\in\edges(K)}  \medge\, \ucvedge \ \characteristic_K,
\end{array}
\end{align}
and we may introduce the space of divergence--free functions from $\Hmeshzero$,
 \begin{equation}\label{def:Edivnulle}
   \boldsymbol{E}_{\edges (\Omega)}= \{\bfu\in\Hmeshzero~;~\dive_{\mesh} \ \bfu=0 \}
 \end{equation}

We proceed with the approximation of the momentum balance equation \eqref{qdm}.
Let us define the discrete gradient of a function $\bfu \in \Hmesh$ (in compliance
with the $\bfH^1$ norm and seminorm). For $(i,j) \in \llbracket 1,d \rrbracket^2$
we define $(\eth_j u_i)$, the approximation of $(\partial_j u_i)$, as:
\begin{equation}\label{eq:defduidxj}
  (\eth_j u_i)_{D_\edged} = \left\{ \begin{array}{l l }
    \dfrac{u_{\edge'} - u_{\edge}}{d(\bfx_{\edge},\bfx_{\edge'})} (\bfn_{\edge, \edged} \cdot \bfe_j) &
    \qquad \mbox{ for } \edged = \edge | \edge ' \in \edgesdinti \\\\
    \dfrac{- u_{\edge}}{d(\bfx_{\edge},\bfx_{\edge,\edged})} (\bfn_{\edge, \edged} \cdot \bfe_j) &
    \qquad \mbox{ for } \edged \in \edgesdexti \cap \edgesd( D_\edge) \end{array} \right.
\end{equation}\\
where $d(\cdot,\cdot)$ refers to the euclidian distance of $\R^d$ and
$\bfn_{\edge, \edged}$ designates the unit normal vector to $\edged$ outward $D_\edge$.

As we now have to deal with functions which are piecewise constant on the cells
$D_\edged$ for  $\edged = \edge | \edge \in \edgesdij$, let us define $\Hmeshdij$ as:
\begin{equation}\label{hmeshdij}
  \Hmeshdij =\Bigl\{ \varphi =  \displaystyle \sum_{\edged \in \edgesdij} \varphi_{\edged}
  \characteristic_{D_\edged}, \quad \varphi_{\edged} \in \R \Bigr\}, \quad  i,j \in \llbracket 1, d\rrbracket
\end{equation}
\paragraph{Discrete gradient operator $\nabla_{\edgesdi}$:}~\\
Let $i \in \llbracket 1, d \rrbracket$, we consider the discrete gradient of the $i^{\text{th}}$
velocity component $u_i$ as follows,

\begin{equation}\label{eq:gradientv}
\begin{array}{l|l}
  \nabla_{\edgesdi}: \quad & \quad \Hmeshi\longrightarrow \prod_{j=1}^d \Hmeshdij \\[1ex]
  & \displaystyle \quad u_i \mapsto \nabla_{\edges^{(i)}} u_i = (\eth_1 u_i, \ldots, \eth_d u_i)
\end{array}\\
\end{equation}\\
with $\eth_j u_i = \ds \sum_{\edged \in \edgesdij }  (\eth_j u_i)_{ D_\edged} \ \characteristic_{D_\edged}$
and $(\eth_j u_i)_{ D_\edged}$ satisfying \eqref{eq:defduidxj}. \\

\begin{remark}
  Considering that the $(i,i)$--partitions, $i\in \llbracket1,d\rrbracket$ coincide with
  $\mesh$, we note $(\eth_i u_i)_K$ in lieu of $(\eth_i u_i)_{D_\edged}$ when
  $\edged \subset K$. Additionally, the definition of $(\eth_i u_i)$ given by
  \eqref{eq:defduidxj} is consistent with \eqref{discdiv-u} by noting that
  $\frac{|\edge|}{|K|} = d(\bfx_\edge, \bfx_\edge ')$ for $(\edge, \edge ')\in \edges(K) \cap \edgesi$.
\end{remark}

\begin{remark}
By analogy with the continuous case, the discrete space $\Hmeshzero$ is endowed with norm $[\cdot,\cdot ]_{1,\edges,0}$ where,
\end{remark}

\begin{equation}
  \left|\begin{array}{l}
 \Hmeshizero \times \Hmeshizero \to \R \\[0.5ex]
 (u,v) \mapsto  [\ui,\vi]_{1,\edgesi,0}
 \end{array}\right. \quad \quad \quad \quad \mbox{ and } \quad \quad \quad \quad
 \left|\begin{array}{l}
 \Hmeshzero \times \Hmeshzero \to \R \\[0.5ex]
 (\bfu,\bfv) \mapsto  [\bfu,\bfv]_{1,\edges,0}
 \end{array}\right.
\end{equation}\\\\
verifying,
\begin{subequations} \label{norm-rv}
\begin{align} \label{normi-rv} &
  \|\ui\|^2_{1,\edgesi,0} = [\ui,\ui]_{1,\edgesi,0} = \sum_{\substack{\edged \in \edgesdinti \\\edged=\edge\!\vert\edge'}}
  \frac{|\edged|}{d_\edged}\ (u_{\edge}-u_{\edge'})^{2}
  + \sum_{\substack{\edged \in \edgesdexti\\ \edged\subset \partial( D_{\edge})}}
  \frac{|\edged|}{d_\edged}\ u_{\edge}^{2}, \quad \mbox{for } i = 1, \ldots, d,\\
  \label{normfull-rv} & \|\bfu\|^2_{1,\edges,0} = [\bfu,\bfu]_{1,\edges,0} = \sum_{i=1}^d \|\ui\|^2_{1,\edgesi,0}.
\end{align}
\end{subequations}\\

\begin{equation}\label{gradient-and-innerproduct-rv}
\int_\Omega \nabla_{\edgesd} \bfu : \nabla_{\edgesd} \bfv  \dx = [\bfu,\bfv]_{1,\edges,0} \quad \forall u, v \in \Hmeshzero.
\end{equation}\\\\
\begin{lemma}[Discrete Poincar\'e inequality (\cite{NS-temam})]\label{poincare}
  Let $\disc=(\mesh,\edges)$ be an admissible MAC mesh of $\Omega$.  For any element $\bfu\in \Hmeshzero$ the discrete inequality holds,
\begin{equation}\label{eq:poincare}
\|\bfu\|_{L^2(\Omega)^d} \leq \diam(\Omega)  \|\bfu\|_{1,\edges,0}
\end{equation}
\end{lemma}

\paragraph{Discrete gradient operator $\bs{\nabla_\edges}$ :}~\\
As the discrete pressure belongs to $L_\mesh$, its gradient is homogeneous to an element of $\Hmesh$:
\begin{equation}\label{discpressuregradient}
\begin{array}{l|l}
\nabla_\edges:\quad & \quad L_{\mesh} \longrightarrow \Hmeshzero \\[1ex]
& \displaystyle \quad p \longmapsto \nabla_\edges p(\bfx)=(\eth_{1} p (\bfx) , \ldots,  \eth_{d} p (\bfx) )^T
\end{array}
\end{equation}\\
where $\eth_i p  =  \displaystyle \sum_{\edge \in \edgesi} (\eth p)_{\edge} \ \characteristic_{D_\edge} \in \Hmeshizero$
is the $i$--th discrete partial derivative of $p$ given by,
\begin{equation}
\ (\eth p)_{\edge} = \frac{|\edge|}{|D_\edge|}\ (p_L - p_K) (\bfn_{K,\edge}\cdot \bfe_i) \ \mbox { for } \ \edge = K | L
\end{equation}
and we have the following duality property connecting $\dive_\mesh$ and $\nabla_\edges$:
\begin{lemma}\label{duality} Let $(q,\bfv) \in L_{\mesh}\times\Hmeshzero$. Then,
\begin{equation}\label{eq:duality}
\int_{\Omega} q  \ \dive_{\mesh}\bfv \dx +\int_{\Omega} \nabla_{\edges} q\cdot \bfv \dx =0.
\end{equation}
\end{lemma}
\begin{proof}
From the definition of $\dive_{\mesh}\bfv \in L_{\mesh}$ and $\nabla_{\edges} q \in \Hmeshzero$,
we reformulate the left hand side in \eqref{eq:duality} as follows,
\begin{equation}\nonumber
  \ds \sum_{K \in \mesh} \ds  | K | q_K (\dive_{\mesh}\bfv)_K + \ds \sum_{i = 1}^{d} \ds
  \sum_{\substack{\edge \in \edgesinti\\\edge = K | L}}|D_{\edge}| (\nabla_{\edges} q)_\edge v_\edge =
  \ds \sum_{K \in \mesh} \ds \sum_{\edge \in \edges(K)} |\edge| v_{K,\edge} q_K  + \ds \sum_{i = 1}^{d}
  \ds \sum_{\substack{\edge \in \edgesinti\\\edge = K | L}}|\edge| (q_L - q_K) v_{K,\edge}
\end{equation}
Using the local conservativity property \eqref{eq:conserv} as well as the boundary
conditions for $\bfu$, the first term may be rewritten so as to sum over the interior
primal faces $\edges \in \edgesinti$  to yield $\ds \sum_{\sigma \in \edgesint}  |\edge|(q_K - q_L) v_{K,\edge} $.
\end{proof}

\paragraph{Approximating the convection term in the momentum equation.}~\\
We carry out the discretization of the term $\partial_t (\rho \bfu) +   \dive(\rho \bfu \otimes \bfu)$
in \eqref{qdm} at time $t\in ]t_n, t_{n+1}]$ with a few precautions for the purpose
of obtaining a discrete kinetic energy balance. For other instances of this approximation
we refer to \cite{degtopo08,ans-11-anl} in the context of finite element schemes,
to \cite{her-10-kin} for a MAC discretized scheme adapted to DDFV Finite Volumes
in \cite{gou-14-rhov}. \\

Let $i = 1,\dots,d$  and $\edge = K|L \in \edgesinti$.  Owing to the fact that the
discrete density lies in $L_\mesh$, we resort to a convex combination of $\rho_K$
and $\rho_L$ for the computation of $(\rho\bfu)_{D_\edge}^{n+1}$ :
\[
(\rho u_i)^{n+1} = \sum_{\edge\in\edgesi} \rho^{n+1}_{D_{\edge}}u_{\edge}^{n+1}\characteristic_{D_\edge}
\in  \Hmeshizero,
\]
where $(\rho^{n+1}_{D_{\edge}})_{\edge \in \edgesinti} \in \Hmeshi$ satisfies:
\begin{equation} \label{rhovitesse}
\vert D_{\edge}\vert\, \rho_{D_{\edge}}^{n+1}=\vert D_{K,\edge}\vert\, \rho_{K}^{n+1} +\vert D_{L,\edge}\vert\,\rho_{L}^{n+1}.
\end{equation}
As a consequence, the discrete derivative $\eth_{t}(\rho u_i)^{n+1}$ approximating
$\partial_{t}(\rho\bfu_i)$
is given by, \[\eth_{t}(\rho u_i)^{n+1} = \sum_{\edge\in\edgesi}\frac{1}{\deltat}(\rho^{n+1}_{D_{\edge}}u_{\edge}^{n+1}-
      \rho^{n}_{D_{\edge}}u_{\edge}^{n}) \characteristic_{D_\edge} \in  \Hmeshizero.\]\\

\noindent Let $i = 1,\dots,d$  and $\edge = K|L \in \edgesinti$. Focusing on the
nonlinear term  $\dive(\rho \bfu \otimes \bfu)$, we have:
\[
\int_{D_\edge} \dive(\rho u_{i}  \bfu) \dx = \sum_{\edged \in \edges(D_\edge)} \int_{\edged} \rho u_{i}
\bfu \cdot \bfn_{\edge,\edged}\dgammax = \sum_{\edged \in \edges(D_\edge)} F_{\edge,\edged} u_\edged.
\]
where the values $F_{\edge,\edged}$ and $u_\edged$ are to be determined through
convex combinations of $(u_\edge)_\edge$ and $\{F_{K,\edge} , K \in \mesh, \edge \in \edges(K)\}$
respectively. For $u_\edged \simeq (u_i)_{|\edged}$ we adopt an upwind scheme :

\begin{equation}
\label{dual:unkown-rv}
u_\edged=\left\{ \begin{array}{l l } u_\edge & \mbox{ if } F_{\edge,\edged} \geq 0 \\
                                     u_\edge ' & \mbox{ otherwise } \end{array} \right.
\end{equation}\\\\
The dual mass flux through $\edged$ outward $D_{\edge}$, denoted $F_{\edge,\edged}$,
is computed following \cite{her-10-kin,Latche14,conv-mac-NS-18} for the mass balance
equation to hold on $D_\edge$. Let $\edged \in \edgesdij$, then the value of $F_{\edge,\edged}$
is determined by the orientation between $\edged$ and $\bfe_i$ (Fig. \ref{fig:fluxconvint}):

\begin{equation}\label{eq:flux_dual}
F_{\edge,\edged}=\left\{ \begin{array}{l l } \dfrac 1 2 \ ( -F_{K,\edge} + F_{K,\edge'} ) & \mbox{ if }  j=i \vspace{10pt} \\
\dfrac 1 2 \ (F_{K,\edgeperp}+ F_{L,\edgeperp'} ) & \mbox{ otherwise. } \end{array} \right.
\end{equation}
In the first case there exists $K\in \mesh$ and $\edge'\in \edgesdi$ such that
$\edged  = \edge | \edge'  \subset K$. In the second case, given that $\edge = K|L$
with $(K,L) \in \mesh^2$ we may find $\edgeperp \in \edges(K), \edgeperp' \in \edges(L)$
with $\edgeperp,\edgeperp'\in \edges^{(j)}$ so that $\edged$ can be defined as the
union of half of each of the two faces $\edgeperp$ and $\edgeperp'$.
\begin{remark}
  Regardless of the situation, there exists $(\rho_\edged,\hat u_\edged)$ such
  satisfying $F_{\edge, \edged}=   | \edged | \rho_{\edged} \hat u_\edged (\bfn_{\edge,\edged} \cdot \bfe_j)$
  where:
\begin{equation}\label{eq:rhoi-ui}
  (\rho_\edged, \hat u_\edged) = \left\{ \begin{array}{l l }
    \left(\dfrac{\rho_\edge + \rho_{\edge'}}{2} , \dfrac{\rho_\edge u_\edge+\rho_{\edge'} u_{\edge'}}{\rho_\edge +\rho_{\edge'} }\right)
    & \mbox{ if }  j=i \\\\
    \left(\dfrac{| \edgeperp| \rho_\edgeperp +| \edgeperp'|\rho_{\edgeperp'} }{|\edgeperp | + | \edgeperp'|}  ,
    \dfrac{|\edgeperp | \rho_\edgeperp u_\edgeperp +|\edgeperp '| \rho_{\edgeperp'} u_{\edgeperp'}}{|\edgeperp |\rho_\edgeperp +|\edgeperp'|\rho_{\edgeperp'}}
    \right)& \mbox{ otherwise. } \end{array} \right.
\end{equation}
\end{remark}

\begin{figure}[htb]
        \centering
        \begin{tikzpicture}[scale=0.5]

          \draw (0,0) rectangle (4,4); \draw (0,4) rectangle (4,8);
          \draw (4,0) rectangle (7,4); \draw (4,4) rectangle (7,8);

          \path(0.5,0.5) node[black] {\small  $K$};

          \path(2,2) node[black] {\small  $+$}; \path(5.5,2) node[black] {\small  $+$};
          \path(2,6) node[black] {\small  $+$}; \path(5.5,6) node[black] {\small  $+$};

          \path(4,2) node[black] {\small  $\times$};
          \path(3.7,0.3) node[black, rotate=90] {\small  $\edge$};

          \path(0,2) node[black] {\small  $\times$};
          \path(-0.4,0.4) node[black, rotate=90] {\small  $\edge'$};

          \draw [blue, line width = 1.5] (2,0) -- (2,4); \path(1.7,0.3) node[blue, rotate = 90] {\small $\edged$};

          \draw [blue, ->, line width = 1] (0,2) -- (-2,2) ; \path(-0.8,2.6) node[blue] {\small $F_{K,\edge'}$};
          \draw [blue, ->, line width = 1] (4,2) -- (3,2) ; \path(3.1,2.5) node[blue] {\small $- F_{K,\edge}$};

          \draw [red, ->, line width = 1] (2,2) -- (0.5,2) ; \path(0.8,2.7) node[red] {$F_{\edge,\edged}$};

          \draw (0,8) -- (0,8.5); \draw (4,8) -- (4,8.5); \draw (7,8) -- (7,8.5);
          \draw (-0.5,8) -- (0,8); \draw (-0.5,4) -- (0,4); \draw (-0.5,0) -- (0,0);

        \end{tikzpicture} \hspace{10pt}
      \begin{tikzpicture}[scale=0.5]

        \path(0.5,0.5) node[black] {\small  $K$};
        \path(6.5,0.5) node[black] {\small  $L$};

        \draw (0,0) rectangle (4,4); \draw (0,4) rectangle (4,8);
        \draw (4,0) rectangle (7,4); \draw (4,4) rectangle (7,8);

        \path(2,2) node[black] {\small  $+$}; \path(5.5,2) node[black] {\small  $+$};
        \path(2,6) node[black] {\small  $+$}; \path(5.5,6) node[black] {\small  $+$};

        \path(4,2) node[black] {\small  $\times$};
        \path(3.7,0.3) node[black, rotate=90] {\small  $\edge$};

        \path(4,6) node[black] {\small  $\times$};

        \draw [blue, ->, line width = 1] (2,4) -- (2,5.5) ; \path(6.1,5) node[blue,rotate= 90] {\small $F_{K,\edgeperp'}$};
        \draw [blue, ->, line width = 1] (5.5,4) -- (5.5,5.5) ; \path(1.3,4.8) node[blue,rotate= 90] {\small $ F_{K,\edgeperp}$};

        \draw [red, ->, line width = 1] (4,4) -- (4,5.5) ; \path(3.3,4.5) node[red] {$F_{\edge,\edged}$};

        \draw [red, line width = 1.5] (0.05,3.9) -- (3.95,3.9); \path(0.5,3.6) node[red] {\small $\edgeperp $};
        \draw [red, line width = 1.5] (4.05,3.9) -- (6.95,3.9); \path(6.5,3.6) node[red] {\small $\edgeperp' $};

        \draw [blue, line width = 2] (2,4) -- (5.5,4); \path(2.4,3.6) node[blue] { $\edged $};

        \draw (0,8) -- (0,8.5); \draw (4,8) -- (4,8.5); \draw (7,8) -- (7,8.5);
        \draw (-0.5,8) -- (0,8); \draw (-0.5,4) -- (0,4); \draw (-0.5,0) -- (0,0);

      \end{tikzpicture}
      \captionsetup{justification=centering,margin=1cm}
      \caption{Dual mass flux associated to $\edged = \edge | \edge' \in \edgesd_{int}^{(1)}$:\\
        for $\edged \perp \bfe_1$ (left) ; for $\edged \perp \bfe_2$ (right) ; ($d = 2$)}
      \label{fig:fluxconvint}
\end{figure}
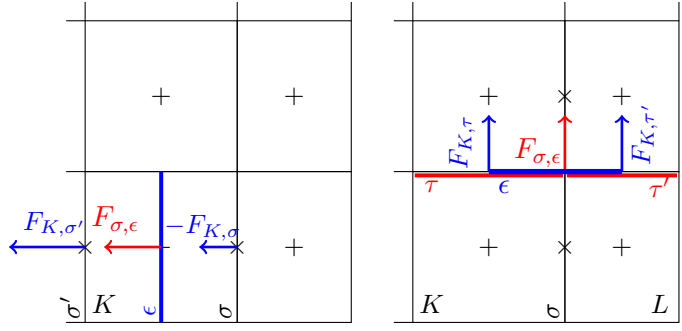
It follows that the nonlinear convection operator ${C}^{(i)}_\edges$ associated
to the $i$--th component of the momentum equation \eqref{schemevitesse} is defined
as,
\begin{equation*}
  \begin{array}{l|l}
   {C}^{(i)}_\edges(\rho\bfu): \quad
    & \quad\Hmeshizero  \longrightarrow  \Hmeshizero \\
     & \displaystyle \quad v \longmapsto {C}^{(i)}_\edges(\rho\bfu) v = \sum_{\edge \in \edgesinti}
     \frac 1{|D_\edge|} \sum_{\substack{\edged \in \tilde{\edges} (D_\edge)\\  \edged=\edge|\edge'}} \!
    F_{\edge,\edged}\  v_{\edged} \ \characteristic_{D_\edge}.
  \end{array}
   \end{equation*}
where $v_{\edged}$ and $F_{\edge,\edged}$ are given by (\ref{dual:unkown-rv}--\ref{eq:rhoi-ui}).
The full convection operator is given by $\bfv \in \Hmeshzero \mapsto {\bfC}_\edges (\rho\bfu) \bfv
= (C^{(1)}_\edges(\rho\bfu) v_1, \ldots, C^{(d)}_\edges(\rho\bfu) v_d)^T$, and for
$i=1,\dots,d$ the following duality property holds on $\edgesdi$:

 \begin{equation}  \label{dd-conv}
  \begin{array}{l l}
    \ds \int_\Omega {C}^{(i)}_\edges(\rho\bfu) v\ w \dx & =  \ds \sum_{\edge \in \edgesinti}
    \ds \sum_{\substack{\edged \in \tilde{\edges} (D_\edge)\\ \edged=\edge|\edge'}}
    | \edged | \rho_{\edged}\hat u_{\edged}v_{\edged} (\bfn_{\edge,\edged} \cdot \bfe_j) w_\edge \\
    & = - \ds \sum_{\substack{\edged \in  \edgesdij\\ \edged=\edge|\edge'}}  | \edged |
    \rho_{\edged} \hat u_{\edged}v_{\edged} (w_\edge' - w_\edge  )(\bfn_{\edge,\edged} \cdot \bfe_j) \\
    & = - \ds \sum_{\edged \in  \edgesdij}  | D_\edged | (\rho \hat u_j)_{D_\edged} (v_i)_{D_\edged}
    (\eth_j w_i)_{D_\edged}  \\\vspace{5pt}
    & = - \ds \int_\Omega v_\edgesdi (\rho \hat \bfu)_\edgesdi \cdot \nabla_\edgesdi w \dx
  \end{array}
   \end{equation}
  where $v_\edgesdi = \sum_{\edged \in \edgesdij}  v_\edged \characteristic_{D_\edged}$
  and $(\rho \hat \bfu)_\edgesdi = \sum_{\edged \in \edgesdij} \rho_{\edged}\hat u_{\edged} \characteristic_{D_\edged}$,
  with $v_\edged,\rho_{\edged}, \hat u_{\edged}$ given
  by \eqref{dual:unkown-rv}--\eqref{eq:rhoi-ui}. By analogy with the continuous
  case, we associate a discrete trilinear form to ${\bfC}_\edges(\rho\bfu)$:
\begin{equation}\label{def:weak-conv-op-rv}
  \begin{array}{r l }
    \forall (\rho,\bfu, \bfv, \bfw) \in L_{\mesh}\times\Hmeshzero^3,  &  \\
    b_\edges(\rho\bfu, \bfv,\bfw)  &  =  \ds \sum_{i=1}^d  b^{(i)}_\edges(\rho\bfu,v_i, w_i)  =
    \ds \sum_{i=1}^d  \int_\Omega \Bigl( {C}^{(i)}_\edges(\rho\bfu)\, v_i\Bigr)\ w_i \dx
  \end{array}
\end{equation}~\\

\begin{remark}
  Owing to the definition \eqref{eq:flux_dual} the mass balance holds over the
  dual cells: let $(\rho,\bfu) \in L_\mesh \times \Hmeshzero$ satisfy \eqref{eq:mass-rv}
  and $\edge = K|L\in \edgesdinti$.  Using the approximation of the mass balance
  equation \eqref{Mass} over $K$ and $L$ combined with $|K| = 2 |D_{K,\edge}|$
  and $|L| = 2 |D_{L,\edge}|$ yields:
\begin{equation}
  \frac{1}{\deltat} (\rho_{D_{\edge}}^{n+1} -\rho_{D_{\edge}}^{n}) + \dfrac{1}{|D_\edge|} \sum_{\edged \in \edgesd(D_\edge)} F_{\edge,\edged}=0.
  \label{mass-dual}
\end{equation}
\end{remark}

\paragraph{Discrete diffusion operator} \ \\
\noindent\\

Following the introduction, the momentum balance equation for variable viscosity and variable
density incompressible flows involves the diffusion term ${\dive(\mu \mathbf{D}(\cdot))}$,
for which $\mu(t,\bfx) \in L^\infty([0,T]\times \Omega)$ is the dynamic fluid viscosity (see \eqref{qdm}).
As in the continuous case, the diffusive term discretization, denoted  $ \dive_\edges(\alpha \mathbf{D})(\cdot)$,
should allow the discrete counterpart of the following integration by parts:
\[
\forall (\bfu,\bfv)\in \Hmeshzero^2, \qquad -\int_\Omega \dive(\mu \mathbf{D}(\bfu)) \cdot \bfv \dx
=  \int_\Omega \mu \mathbf{D}(\bfu) : \mathbf{D}(\bfv) \dx,
\]
to hold, but also match with the discrete analog $\mathbf{D}_\edgesd(\cdot)$ of the strain rate tensor. Hence, let
\begin{equation}\label{eq:deformationtensorv}
\begin{array}{l|l}
\mathbf{D}_{\edgesd}: \quad
& \quad
\Hmesh \longrightarrow \bs \prod_{i,j=1}^d \Hmeshdij \\[1ex]
& \displaystyle \quad \bfu \mapsto \mathbf{D}_{\bs \edgesd} \bfu = (\frac{1}{2} (\eth_j u_i + \eth_i u_j ))_{i,j = 1,\dots, d}
\end{array}\\
\end{equation}\\
where $\eth_j u_i = \ds \sum_{\edged \in \edgesdij }  (\eth_j u_i)_{ D_\edged} \ \characteristic_{D_\edged}$
with $(\eth_j u_i)_{ D_\edged}$ given by \eqref{eq:defduidxj}.\\\\

\noindent Thus, for each $\edge \in \edgesi \cap \edgesint$, $i \in \llbracket1,d\rrbracket$,
let us consider the average onto $D_\edge$ of the $i^{th}$ component of the momentum balance
\begin{equation} \label{viscous-discr-ith}
\ds \int_{D_\edge} \sum_{j=1}^d \frac{\partial}{\partial x_j}(\frac{\mu}{2}
(\frac{\partial u_i}{\partial x_j} + \frac{\partial u_j}{\partial x_i})) \dx
= \sum_{\substack{\edged \in \edgesd(D_\edge) \\ \edged \perp \bfe_j}} \ds
\int_{\edged} \frac{\mu}{2}(\frac{\partial u_i}{\partial x_j} + \frac{\partial u_j}{\partial x_i}) (
\bfn_{\edge,\edged}\cdot \bfe_j) \dgammax.
\end{equation}

In order to approximate \eqref{viscous-discr-ith}, viscosity values that are computed
onto the dual faces $\edged \in \edgesd$ are required. As we assume that the viscosity
depends continuously on the density, $(\mu_\edged)_{\edged \in \edgesd}$ will be
defined from $(\rho_K)_{K \in \mesh}$ in a way to be determined. Moreover, recalling
that  the sets $\{D_\edged, \edged \in \edgesdij\}$ and $\{D_\edged, \edged \in \edgesd^{(j,i)}\}$
coincide ; the approximate values of $(\mu_\edged)_{\edged \in \edgesd}$
are chosen to insure that the $N_\edges$-matrix relative to $\dive_\edges(\mu \mathbf{D})(\cdot)$
is symmetric (where we identified $\Hmeshzero$ to $\R^{N_\edges}$ with $\vert\edgesint\vert = N_\edges$).\\

More precisely, if one considers a dual face $\edged \in \edgesd(D_\edge) \cap \edgesdi$
with $\edged \perp \bfe_j$, the approximation of the surface integrals in  \eqref{viscous-discr-ith}
are dependent on whether $\edged$ is perpendicular or tangent to $\bfe_j$,

\begin{align}\label{eq:flux-visc-ij}
  \ds \int_{\edged} \frac{\mu}{2}(\frac{\partial u_i}{\partial x_j}
  + \frac{\partial u_j}{\partial x_i}) (\bfn_{\edge,\edged}\cdot \bfe_j) \dgammax \approx
  \left\{ \begin{array}{l l }|\edged|\mu_\edged \dfrac{u_{\edge '}-u_\edge}{d(\bfx_\edge, \bfx_{\edge'})} & \mbox{ if }  j=i \vspace{10pt} \\
    |\edged|\dfrac{\mu_\edged}{2} \Big(\Big(\dfrac{u_{\edge '}-u_\edge}{d(\bfx_\edge, \bfx_{\edge'})}\Big)(\bfn_{\edge,\edged}\cdot \bfe_j)+  \\
    \quad \Big(\dfrac{u_{\tau'}- u_{\tau}}{d(\bfx_\tau, \bfx_{\tau'})}\Big) (\bfn_{\tau,\edged'}\cdot \bfe_i) \Big)
    (\bfn_{\edge,\edged}\cdot \bfe_j) & \mbox{ otherwise. } \end{array} \right.
\end{align}
In the first case we use the definition of $(\eth_i u_i)_{D_\edged}$ from \eqref{eq:defduidxj}.
In the second case, setting $\edge = K|L$ with $(K,L) \in \mesh^2$ we may find
$\edgeperp \in \edges(K), \edgeperp' \in \edges(L)$ with $\edgeperp,\edgeperp'\in \edges^{(j)}$ so that $\edged$ can
be defined as the of half of each of the two faces $\edgeperp$ and $\edgeperp'$.
Additionally, there exists $\edged ' \in \edgesd^{(j)}$, $\edged ' \perp \bfe_i$
verifying $D_\edged = \edged ' \times [\bfx_\tau, \bfx_{\tau'}]$.\\\\

 \noindent From the approximate value of
$\ds \int_{\edged} \frac{\mu}{2}(\frac{\partial u_i}{\partial x_j} + \frac{\partial u_j}{\partial x_i}) (\bfn_{\edge,\edged}\cdot \bfe_j) \dgammax$,
one can denote by $\dive_\edges(\mu \mathbf{D})(\cdot)$ the discrete diffusion operator by setting
 \begin{equation}\label{def:eq:Du}
  \begin{array}{l|l}
  \dive_\edges(\mu \mathbf{D})(\cdot): \quad
    & \quad\Hmeshzero  \longrightarrow  \Hmeshzero \\
  & \displaystyle \quad \bfu \longmapsto \dive_\edges(\mu \mathbf{D}(\bfu)) =
  \left(\dive^{(1)}_\edges(\mu \mathbf{D}(\bfu)), \dots, \dive^{(d)}_\edges(\mu \mathbf{D}(\bfu))\right)^T
  \end{array}
   \end{equation}
with for $i \in \llbracket 1,d \rrbracket$,
 \begin{equation}\label{def:eq:Dui}
\dive^{(i)}_\edges(\mu \mathbf{D}(\bfu)) = \sum_{\edge \in \edgesinti}
\sum_{\substack{\edged \in \tilde{\edges} (D_\edge)\\  \edged \perp \bfe_j}} \!
\frac {|\edged| }{|D_\edge|} \frac{\mu_\edged}{2} \left(\eth_j u_i + \eth_i u_j\right)_{D_\edged}
(\bfn_{\edge,\edged}\cdot \bfe_j)\ \characteristic_{D_\edge}
\end{equation}
 computed by the previous values \eqref{eq:flux-visc-ij}.\\

\begin{lemma}[\textbf{Duality} $\bs{\dive_\edges}$ -- $\bs{D_\edgesd}$]\label{duality-dual} Let $\disc=(\mesh,\edges)$ a MAC mesh of $\Omega$.\\
Then for all $\bfu, \bfv\in\Hmeshzero$,
\begin{equation}\label{eq:duality-dual}
\int_\Omega \dive_\edges(\mu \mathbf{D}(\bfu)) \cdot \bfv \dx = - \int_\Omega \mu \bs{D_\edgesd}(\bfu) : \bs{D_\edgesd}(\bfv) \dx.
\end{equation}
\end{lemma}
\begin{proof}
From \eqref{def:eq:Du}--\eqref{def:eq:Dui}, the left hand side of equation \eqref{eq:duality-dual} can be written as,
\begin{equation*}
 \begin{array}{l l l}
   \ds \sum_{i=1}^d\int_\Omega\dive^{(i)}_\edges(\mu \mathbf{D}(\bfu))\cdot v_i\dx &=&
   \ds \sum_{i,j=1}^d \sum_{\edge \in \edgesinti}\sum_{\substack{\edged \in \tilde{\edges} (D_\edge)\\  \edged \perp \bfe_j}} \!
   |\edged| \frac{\mu_\edged}{2} \left(\eth_j u_i + \eth_i u_j\right)_{D_\edged} (\bfn_{\edge,\edged}\cdot \bfe_j) v_\edge \\
    & =&  \overbrace{-  \ds \sum_{i,j=1}^d \sum_{\substack{\edged \in \edgesdij\\  \edged  = \edge | \edge' \in \edgesdinti}} \!
   |\edged| \frac{\mu_\edged}{2} \left(\eth_j u_i + \eth_i u_j\right)_{D_\edged} (\bfn_{\edge,\edged}\cdot \bfe_j)(v_{\edge'} - v_\edge)}^{S_{int}}\\
    && \underbrace{+ \ds \sum_{i,j=1}^d \sum_{\substack{\edged \in \edgesdij\\  \edged \in \edgesd(D_\edge) \cap \edgesdexti}} \!
  |\edged| \frac{\mu_\edged}{2} \left(\eth_j u_i + \eth_i u_j\right)_{D_\edged} (\bfn_{\edge,\edged}\cdot \bfe_j) v_\edge}_{S_{ext}}.
 \end{array}
\end{equation*}
The velocity fields $\bfu$ and $\bfv$ vanishing on the boundary, and owing to property of symmetry
$\{D_\edged,~\edged~\in~\edgesdij\} = \{D_\edged, \edged \in \edgesd^{(j,i)}\}$, one can reorganise the sums $S_{int}$ and $S_{ext}$,
\begin{equation*}
 \begin{array}{l l l}
   \displaystyle S_{int} & = - \dfrac{1}{2}\ds \sum_{i, j=1}^d \Bigg( & \displaystyle
   \sum_{\substack{\edged \in \edgesdij\\\edged  = \edge | \edge' \in \edgesdinti}} \! |\edged| \frac{\mu_\edged}{2}
   \Big(\eth_j u_i + \eth_i u_j\Big)_{D_\edged} (\bfn_{\edge,\edged}\cdot \bfe_j)(v_{\edge'} - v_\edge) \\
   && +   \displaystyle \sum_{\substack{\edged' \in \edgesd^{(j,i)}\\ \edged'  = \tau | \tau' \in \edgesd^{(j)}_{int}}} \! |\edged'| \frac{\mu_{\edged'}}{2}
   \Big(\eth_j u_i + \eth_i u_j\Big)_{D_\edged'} (\bfn_{\tau,\edged'}\cdot \bfe_j)(v_{\tau'} - v_\tau)\Bigg).
 \end{array}
\end{equation*}
Using $|\edged| = |D_\edged|/d(\bfx_\edge, \bfx_{\edge'})$ and $|\edged'| = |D_{\edged'}|/d(\bfx_\tau, \bfx_{\tau'})$, we get
\begin{equation*}
 \begin{array}{l l l}
   \displaystyle S_{int} & = - \dfrac{1}{2}\ds \sum_{i, j=1}^d \Bigg( &
   \displaystyle \sum_{\substack{\edged \in \edgesdij\\ \edged  = \edge | \edge' \in \edgesdinti}} \!
    |D_\edged| \frac{\mu_\edged}{2} \Big(\eth_j u_i + \eth_i u_j\Big)_{D_\edged} \Big(\eth_j v_i\Big)_{D_\edged}  \\
    && +   \displaystyle \sum_{\substack{\edged' \in \edgesd^{(j,i)}\\\edged'  = \tau | \tau' \in \edgesd^{(j)}_{int}}} \!
   |D_{\edged'}| \frac{\mu_{\edged'}}{2} \Big(\eth_j u_i + \eth_i u_j\Big)_{D_{\edged'}} \Big(\eth_i v_j\Big)_{D_{\edged'}} \Bigg).
 \end{array}
\end{equation*}
We proceed similarly for $S_{ext}$ and obtain after summing over the dual faces $\edged \in \edgesdexti$, the expected outcome,
\[
S_{int} + S_{ext}  =  - \sum_{\edged \in \edgesd} |D_\edged| \frac{\mu_\edged}{4}
\Big(\eth_j u_i + \eth_i u_j\Big)_{D_\edged} \Big(\eth_j v_i + \eth_i v_j\Big)_{D_\edged}
= - \int_\Omega \mu \bs{D_\edgesd}(\bfu) : \bs{D_\edgesd}(\bfv) \dx.
\]
\end{proof}\\
In order to derive apriori estimates later on, we resort to a discrete Korn inequality
for the purpose of recovering the $\Hmeshzero$ norm of the velocity.

\begin{lemma}[Discrete Korn inequality]\label{korn} For any $\bfu\in \Hmeshzero$
we have,
\begin{equation}\label{eq:korn}
\|\bfu\|_{1,\edges,0} \leq  \sqrt{2} \|D_\edgesd (\bfu) \|_{L^2(\Omega)}
\end{equation}
\end{lemma}

\begin{proof}
We resort to the discrete counterparts of the arguments used in \cite{BoyerFabrie-book}.
First and foremost Lemma \ref{duality-dual} is adapted to the discrete operator
$\nabla^T_\edgesd$ to yield,
\begin{equation*}
 \begin{array}{l l }
   \ds \sum_{i=1}^d \int_\Omega u_i \cdot \dive^{(i)}_\edges( \nabla_\edgesd^T \bfu)  \dx   & =
   \ds \sum_{i,j=1}^d \sum_{\edge \in \edgesinti} \sum_{\substack{\edged \in \tilde{\edges} (D_\edge)\\  \edged \perp \bfe_j}} \!
   |\edged| (\bfn_{\edge,\edged}\cdot \bfe_j)u_\edge \left(\frac{\eth u_j}{\eth \bfx_i}\right)_{D_\edged}  \\
   & =  -  \ds \sum_{i,j=1}^d \sum_{\edged \in \edgesdij} \!
  |D_\edged| \left(\frac{\eth u_i}{\eth \bfx_j} \frac{\eth u_j}{\eth \bfx_i}\right)_{D_\edged} = - \int_\Omega \nabla_\edgesd \bfu : \nabla^T_\edgesd \bfu  \dx\\
 \end{array}
\end{equation*}
Moreover, for $\edge = K|L \in \edgesdinti$
\begin{equation*}
 \begin{array}{l l }
   \displaystyle
   \left(\dive^{(i)}_\edges( \nabla_\edgesd^T \bfu)\right)_{D_\edge}   & = \displaystyle \dfrac{1}{|D_\edge|}
   \displaystyle \sum_{j=1}^d \displaystyle \sum_{\substack{\edged \in \tilde{\edges} (D_\edge)\\  \edged \perp \bfe_j}} \!
   |\edged| \left(\frac{\eth u_j}{\eth \bfx_i}\right)_{D_\edged} (\bfn_{\edge,\edged}\cdot \bfe_j) \\
   & = \displaystyle \dfrac{1}{|D_\edge|}  \displaystyle \displaystyle \sum_{j=1}^d |\edge|
   \left(\left(\frac{\eth u_j}{\eth \bfx_j}\right)_{L} - \left(\frac{\eth u_j}{\eth \bfx_j}\right)_{K}\right) (\bfn_{K,\edge}\cdot \bfe_i) \\\\
   & =  \left(\eth^{(i)}_\edges( \dive_\mesh \bfu)\right)_{D_\edge}
    \end{array}
\end{equation*}
The result above is demonstrated for the terms when $j = i$ by noting that for
$(\edged_-, \edged_+)\in (\edgesd(D_\edge)\cap\edgesdij)$ we have $|\edge| = |\edged_+| = |\edged_-|$ et
$\bfn_{\edge,\edged} = - \bfn_{K,\edge}$ si $D_\edged = K$.  \\\\
Otherwise when $j \neq i$, then if $\edge = K|L$ there exists
$(\edgeperp_+,\edgeperp_-) \in (\edges(K)\cap \edges^{(j)})^2$, $(\edgeperp'_+,\edgeperp'_-) \in (\edges(L)\cap \edges^{(j)})^2$
et $(\edged_-, \edged_+)\in (\edgesd(D_\edge)\cap\edgesdij)$ verifying:
\begin{list}{$\ast$}{\itemsep=0.ex \topsep=0.5ex \leftmargin=1.cm \labelwidth=0.7cm \labelsep=0.3cm \itemindent=0.cm}
\item $\edged_+$ the union of half of each of the two faces $\edgeperp_+$ et $\edgeperp_+'$ :
  $\bfn_{\edge,\edged_+} = \bfn_{K,\edgeperp_+} = \bfn_{L,\edgeperp'_+}$.
\item $\edged_-$ the union of half of each of the two faces $\edgeperp_-$ et $\edgeperp_-'$ :
  $\bfn_{\edge,\edged_-} = \bfn_{K,\edgeperp_-} = \bfn_{L,\edgeperp'_-}$.
\end{list}
Then,
\[
|\edged_+|/d(\bfx_{\edgeperp_+}, \bfx_{\edgeperp'_+}) = |\edged_-|/d(\bfx_{\edgeperp_-}, \bfx_{\edgeperp'_-})
= | \edge |/d(\bfx_{\edgeperp_+}, \bfx_{\edgeperp_-}) = | \edge |/d(\bfx_{\edgeperp'_+}, \bfx_{\edgeperp'_-})
\]
to end up with $\displaystyle \sum_{\substack{\edged \in \tilde{\edges} (D_\edge)\\  \edged \perp \bfe_j}} \!
|\edged| (\eth_i u_j)_{D_\edged} (\bfn_{\edge,\edged}\cdot \bfe_j) = |\edge|
\left((\eth_j u_j)_{L} - (\eth_j u_j)_{K}\right) (\bfn_{K,\edge}\cdot \bfe_i)$. \\\\\\
Using those results combined with \eqref{eq:duality} we obtain,
\begin{equation*}
 \begin{array}{l  }
   \displaystyle
   \int_\Omega \nabla_\edgesd \bfu : \nabla^T_\edgesd \bfu  \dx  =  - \displaystyle \int_\Omega
   \bfu \cdot \nabla_\edges (\dive_\mesh \bfu)  \dx   = \displaystyle \int_\Omega  (\dive_\mesh \bfu)^2 \dx \\
    \end{array}
    \end{equation*}
which finally yields the discrete Korn inequality,
\begin{equation}
 \begin{array}{l l }
   \displaystyle
   4  \int_\Omega D_\edgesd (\bfu) : D_\edgesd (\bfu)  \dx & =  \|\bfu\|^2_{1,\edges,0}
   + \displaystyle \int_\Omega \nabla_\edgesd^T \bfu : \nabla_\edgesd^T \bfu  \dx
   + 2 \int_\Omega (\dive_\mesh \bfu)^2   \dx \label{eq:bound-D} \\\\
   & \geq  2 \|\bfu\|^2_{1,\edges,0}
    \end{array}
\end{equation}
thanks to the symmetry property,
\begin{equation}\label{eq:symm-grad}
 \begin{array}{l  }
   \displaystyle  \int_\Omega \nabla_\edgesd^T \bfu : \nabla_\edgesd^T \bfu  \dx = \sum_{i,j = 1}^d
   \sum_{\edged \in \edgesdij} |D_\edged| (\eth_i u_j)^2_{D_\edged} = \int_\Omega \nabla_\edgesd \bfu : \nabla_\edgesd \bfu  \dx
    \end{array}
    \end{equation}

\end{proof}\\\\

Finally, special attention is given for the computation of the discrete viscosity
tensor in \eqref{schemevitesse} in order to gain a symmetry property for the diffusive operator:\\

\begin{definition}[Computation of the viscosity "tensor"]\label{def:dis:visc}
  \noindent Let $\mu : \rho \mapsto \mu(\rho)$ be the $C^0$ viscosity from the continuous
  formulation of the problem \eqref{pb:cont}. Then, the discrete viscosity is function
  of the discrete density in the following manner:\\\\
  For $0\leq n\leq N-1$ and $\rho^{n+1} \in L_{\mesh}$ satisfying the discrete mass balance equation
  \eqref{eq:mass-rv}, we define $\mu^{n+1}\in  L_{\mesh}$ as,\\
\begin{equation}\label{def:mumesh}
  \mu^{n+1} = \mu(\rho^{n+1})  = \ds \sum_{K \in \mesh} \mu_K^{n+1} \characteristic_{K} \quad
  \mbox{ with }  \quad \mu_K^{n+1} = \mu(\rho_K^{n+1}) \quad \forall K \in \mesh
\end{equation}
Additionally, for any $i,j$ in $\llbracket1,d\rrbracket$ we compute $\mu^{n+1}_{ij}$ -- the
discrete viscosity tensor associated to $\edgesdij$ -- from the values $(\mu^{n+1}_K)_{K\in \mesh}$ and we have:
\begin{equation}\label{def:muedgesd}
  \mu^{n+1}_{ij}(\bfx) = \sum_{\substack{\edged \in \edgesdi\\\edged \perp \bfe_j}} \mu^{n+1}_\edged \characteristic_{D_\edged}(\bfx)
\end{equation}
Therefore $\mu^{n+1}_{ij}$  belongs to the space $\Hmeshdij$ defined by \eqref{hmeshdij}. \\\\
Let $\edged \in \edgesdij\cap \edgesdinti = \{ \edged \in \edgesdinti, \edged \perp \bfe_j \}$
be the dual face separating $D_\edge$ and $D_{\edge'}$. If $i =j$ then we set
$\mu^{n+1}_\edged = \mu^{n+1}_K$ if $\edged$ is included in $K\in\mesh$ (see Figure \ref{fig:muedgedii}).
Otherwise, we may find $(K,K',L, L') \in \mesh^4$ verifying $\edge = K|L$ and $\edge' = K' | L'$
(see Figure \ref{fig:muedgedij}) and we define $\mu^{n+1}_\edged$ associated to the dual cell $D_\edged$ as,\\
\begin{equation}\label{def:muedgedint}
  |D_\edged| \mu^{n+1}_\edged = \frac{1}{4} \left( |K| \mu^{n+1}_{K} + |L| \mu^{n+1}_{L} +  |K'| \mu^{n+1}_{K'} + |L'| \mu^{n+1}_{L'} \right)
\end{equation}
If $\edged \in \edgesdij\cap \edgesdexti = \{ \edged \in \edgesdexti, \edged \perp \bfe_j \}$
then $\edged \in \edgesd(D_\edge)\cap \edgesdexti$ for some $\edge = K|L$
and we are in the configuration illustrated in Figure \ref{fig:muedgedii}.
Hence $|D_\edged| = \frac{1}{2}|D_\edge| = \frac{1}{4}(|K|+|L|) $ and we set:
\begin{equation}\label{def:muedgedext}
  |D_\edged| \mu^{n+1}_\edged = \frac{1}{4} \left( |K| \mu^{n+1}_{K} + |L| \mu^{n+1}_{L}\right).
\end{equation}
Finally, so as to avoid an unnecessary excess of notations in some parts of the chapter,
we note $\mu^{n+1}_\edgesd \bs{D_\edgesd}(\bfu) = \frac{1}{2} (\mu^{n+1}_{ij}(\eth_j u_i + \eth_i u_j))_{i,j =1,..,d}$.
\end{definition}
\begin{remark}
  Indeed, when $i=j$ we recall that the partition given by $\{ D_\edged, \edged \in \edgesd^{(i,i)}\}$
  coincides with the set of primal cells $K\in \mesh$. Moreover, thanks to the equivalence of
  the partitions $\{ D_\edged, \edged \in \edgesdij \} = \{ D_\edged, \edged \in \edgesd^{(j,i)} \}$
  there exists $\edged' \in \edgesd^{(j,i)}$ such that $\mu_\edged = \mu_{\edged'}$.
\end{remark}
\begin{figure}[htb]
  \centering
\begin{tikzpicture}[scale=0.7]

  \draw [green!10, fill=green!10]  (0,0) rectangle (4,4);

  \draw (0,4) -- (0,4.5); \draw (4,4) -- (4,4.5);
  \draw (4,4) -- (4.5,4); \draw (4,0) -- (4.5,0);

  \draw (0,0) -- (0,-0.5); \draw (4,0) -- (4,-0.5);
  \draw (0,4) -- (-0.5,4); \draw (0,0) -- (-0.5,0);

  \path(1.3,0.5) node[blue] {\small $ = D_\edged $};

  \path(0.5,0.5) node[black] {\small  $K$};

  \draw (0,0) rectangle (4,4);

  \path(2,2) node[black] {\small  $+$};

  \path(2.5,1.5) node[black] {\small  $\mu_K$};
  \path(3.3,1.5) node[blue] {\small $ = \mu_\edged $};

  \draw [blue, line width = 1.5] (2,0) -- (2,4); \path(2.4,3.7) node[blue] { $\edged $};

    \draw [red, line width = 1.5] (0,0) -- (0,4); \path(-0.5,2) node[red] {\small $\edge ' $};
    \draw [red, line width = 1.5] (4.0,0) -- (4,4); \path(4.5,2) node[red] {\small $\edge $};

\end{tikzpicture} \hspace{20pt}\begin{tikzpicture}[scale=0.4][htb]

  \path(0.5,0.5) node[black] {\small  $K$};
  \path(6.5,0.5) node[black] {\small  $L$};
  \path(0.5,7.5) node[black] {\small  $K'$};
  \path(6.5,7.5) node[black] {\small  $L'$};

  \draw [green!10, fill=green!10]  (2,0) rectangle (5.5,2);

  \path(2.5,1.4) node[blue] {\small $D_\edged $};

  \draw (0,0) rectangle (4,4); \draw (0,4) rectangle (4,8);
  \draw (4,0) rectangle (7,4); \draw (4,4) rectangle (7,8);

  \path(2,2) node[black] {\small  $+$}; \path(5.5,2) node[black] {\small  $+$};
  \path(2,6) node[black] {\small  $+$}; \path(5.5,6) node[black] {\small  $+$};

  \path(2,2.5) node[black] {\small  $\mu_K$};   \path(2,6.5) node[black] {\small  $\mu_{K'}$};
  \path(5.5,2.5) node[black] {\small  $\mu_L$};   \path(5.5,6.5) node[black] {\small  $\mu_{L'}$};

  \draw [blue, line width = 1.5] (2,0) -- (5.5,0); \path(2.4,-0.4) node[blue] {\small $\edged $};

  \draw (0,8) -- (0,8.5); \draw (4,8) -- (4,8.5); \draw (7,8) -- (7,8.5);
  \draw (-0.5,8) -- (0,8); \draw (-0.5,4) -- (0,4); \draw (-0.5,0) -- (0,0);

    \draw [red, line width = 1.5] (4,0) -- (4,3.95); \path(3.7,0.5) node[red] {\small $\edge $};

    \path(4,4) node[black] {\small  $\times$};
    \path(4.5,4.5) node[black] {\small  $\mu_\edged$};

\end{tikzpicture}
\centering
\captionsetup{justification=centering,margin=1cm}
\caption{Computation of \eqref{def:muedgedint} for $\edged \in  \edgesd^{(1,1)}$ (left) ; for $\edged \in  \edgesd^{(1,2)} \cap \edgesext^{(1)}$ (right)}
\label{fig:muedgedii} \end{figure}

\begin{figure}[htb]
  \centering
\begin{tikzpicture}[scale=0.5]

  \path(0.5,0.5) node[black] {\small  $K$};
  \path(6.5,0.5) node[black] {\small  $L$};
  \path(0.5,7.5) node[black] {\small  $K'$};
  \path(6.5,7.5) node[black] {\small  $L'$};

  \draw [green!10, fill=green!10]  (2,2) rectangle (5.5,6);

  \path(2.5,2.4) node[blue] {\small $D_\edged $};

  \draw (0,0) rectangle (4,4); \draw (0,4) rectangle (4,8);
  \draw (4,0) rectangle (7,4); \draw (4,4) rectangle (7,8);

  \path(2,2) node[black] {\small  $+$}; \path(5.5,2) node[black] {\small  $+$};
  \path(2,6) node[black] {\small  $+$}; \path(5.5,6) node[black] {\small  $+$};

  \path(2,1.5) node[black] {\small  $\mu_K$};   \path(2,6.5) node[black] {\small  $\mu_{K'}$};
  \path(5.5,1.5) node[black] {\small  $\mu_L$};   \path(5.5,6.5) node[black] {\small  $\mu_{L'}$};

  \draw [blue, line width = 1.5] (2,4) -- (5.5,4); \path(2.4,4.4) node[blue] {\small $\edged $};

  \draw (0,8) -- (0,8.5); \draw (4,8) -- (4,8.5); \draw (7,8) -- (7,8.5);
  \draw (-0.5,8) -- (0,8); \draw (-0.5,4) -- (0,4); \draw (-0.5,0) -- (0,0);

    \draw [red, line width = 1.5] (4,0) -- (4,3.95); \path(3.7,0.5) node[red] {\small $\edge $};
    \draw [red, line width = 1.5] (4.0,4.05) -- (4,8); \path(3.7,7.5) node[red] {\small $\edge' $};

    \path(4,4) node[black] {\small  $\times$};
    \path(4.5,4.5) node[black] {\small  $\mu_\edged$};

\end{tikzpicture} \hspace{10pt}
\begin{tikzpicture}[scale=0.5]

  \path(0.5,0.5) node[black] {\small  $K$};
  \path(6.5,0.5) node[black] {\small  $L$};
  \path(0.5,7.5) node[black] {\small  $K'$};
  \path(6.5,7.5) node[black] {\small  $L'$};

  \draw [green!10, fill=green!10]  (2,2) rectangle (5.5,6);

  \path(2.5,2.4) node[blue] {\small $D_\edged $};

  \draw (0,0) rectangle (4,4); \draw (0,4) rectangle (4,8);
  \draw (4,0) rectangle (7,4); \draw (4,4) rectangle (7,8);

  \path(2,2) node[black] {\small  $+$}; \path(5.5,2) node[black] {\small  $+$};
  \path(2,6) node[black] {\small  $+$}; \path(5.5,6) node[black] {\small  $+$};

  \path(2,1.5) node[black] {\small  $\mu_K$};   \path(2,6.5) node[black] {\small  $\mu_{K'}$};
  \path(5.5,1.5) node[black] {\small  $\mu_L$};   \path(5.5,6.5) node[black] {\small  $\mu_{L'}$};

  \draw [blue, line width = 1.5] (4,2) -- (4,6); \path(3.6,5.6) node[blue] {\small $\edged '$};

  \draw (0,8) -- (0,8.5); \draw (4,8) -- (4,8.5); \draw (7,8) -- (7,8.5);
  \draw (-0.5,8) -- (0,8); \draw (-0.5,4) -- (0,4); \draw (-0.5,0) -- (0,0);

    \draw [red, line width = 1.5] (0,4) -- (4,4); \path(0.4,4.4) node[red] {\small $\tau $};
    \draw [red, line width = 1.5] (4.0,4) -- (7,4); \path(6.6,4.4) node[red] {\small $\tau' $};

    \path(4,4) node[black] {\small  $\times$};
    \path(4.5,4.5) node[black] {\small  $\mu_\edged$};

\end{tikzpicture}
\centering
\captionsetup{justification=centering,margin=1cm}
\caption{Computation of \eqref{def:muedgedint} for $\edged \in  \edgesd^{(1,2)}$ and $\edged' \in  \edgesd^{(2,1)}$ ($d = 2$)}
\label{fig:muedgedij} \end{figure}
\noindent Thus, the viscosity tensor given by the definition above yields the symmetry of
the diffusive operator as needed. We give below a formulation equivalent to \eqref{eq:scheme:ro}
using the approximations defined so far,
\hspace{-10pt}\begin{subequations} \label{eq:scheme:dis}
\begin{align} \nonumber &
  \mbox{Let }  \bfu^0 = \Pi_\edges \bfu_0\in \Hmeshzero \quad  \mbox{and} \quad \rho^0 = \Pi_\mesh \rho_0\in L_\mesh\\ \nonumber
  & \\ \nonumber
  & \mbox{For }  n\in \{ 0, \cdots, N-1\} \\[1ex] \nonumber
  & \hspace{1ex}
  \mbox{ Find } (\rho^{n+1},\bfu^{n+1},p^{n+1})\in  L_{\mesh}\times \Hmeshzero \times  L_{\mesh} \quad \mbox{satisfying :} \\[1ex]
  \label{eq:mass-rv:dis} & \quad \dfrac{1}{\deltat}(\rho_K^{n+1} - \rho_K^{n}) + \ds \dfrac{1}{|K|}
  \sum_{\edge \in \edges(K)} F^{n+1}_{K,\edge} =0& \forall K \in \mesh \\[1ex] \label{schemevitesse:dis}
  & \quad \dfrac{1}{\deltat}(\rho_{D_\edge}^{n+1}  u_{\edge}^{n+1} - \rho_{D_\edge}^{n}u_{\edge}^{n})
  + \dfrac{1}{|D_\edge|} \ds \sum_{\edged \in \edgesd(D_\edge)} F^{n+1}_{\edge,\edged} u_{\edged}^{n+1}\\[0ex]\nonumber
  &   \quad \quad   \quad \quad  \quad \quad \quad \quad  - (\dive_\edges(\mu_\edgesd\mathbf{D}_\edgesd(\bfu^{n+1})))_{D_\edge}
  +  (\nabla_{\edges} \ p^{n+1})_{D_\edge} = f_{\edge}^{n+1}& \forall i, \forall \sigma \in \edgesinti \\[0ex]\nonumber
  & \\[1ex]\label{schemediv:dis}
  &\,\,\,\, \ds \sum_{\edge \in \edges(K)} |\edge| \bfu^{n+1}_{\edge}\cdot \bfn_{K,\edge} =0& \forall K \in \mesh \\[1ex] \label{schemediv:pre-moy}
  & \quad \ds \sum_{K\in \mesh} |K| p^n_K =0&
\end{align}
\end{subequations}~\\

\section{Preliminary results and tools.}
\label{sec;MAC-tools}

\noindent Due to the staggered nature of the MAC grids, we are led to deal with convex combinations
of the discrete solutions $(\rho^{n+1},\bfu^{n+1})\in  L_{\mesh}\times \Hmeshzero$.
Furthermore, at some point in the proof -- mainly during the passage to the limit in the discrete scheme --
we are required to control those quantities to assure their convergence.
Hence, we define  in this section several "reconstruction" operators so as to cover all our needs.

\begin{definition}[Reconstructions from $\mesh$ to $\edgesi$]
  Let $\disc=(\mesh, \edges)$ be an admissible MAC mesh, and $i\in \llbracket1,d\rrbracket$.
  We define the reconstruction operator from $L_\mesh$ to $\Hmeshi$ as follows:
\begin{equation}
\begin{array}{l|ccl}\label{def:R-mesh-to-edgesi}
\mathcal R_{\mesh}^{\edgesi}: \quad & L_\mesh & \to & \Hmeshi \\[2ex]
& q & \mapsto & \displaystyle \mathcal R_{\mesh}^{\edgesi} q =
\sum_{\substack{\edge \in \edgesinti\\\edge = K | L}} (\mathcal R_{\mesh}^{\edgesi} q)_{D_\edge}\ \characteristic_{D_\edge}
+ \sum_{\substack{\edge \in \edgesexti\\\edge \in \edges(K)}} q_K\ \characteristic_{D_{K,\edge}}
\end{array}
\end{equation}
where,
\[
(\mathcal R_{\mesh}^{\edgesi} q)_{D_\edge} = \alpha_\edge q_K + (1-\alpha_\edge)q_L \quad
\mbox{with}\quad  \alpha_\edge = |D_{K,\edge}|/|D_\edge|  \quad \mbox{if}\quad  \edge = K|L \in \edgesinti
\]
and we note that the value $ (\mathcal R_{\mesh}^{\edgesi} q)_{D_\edge}$ for $\edge \in \edges(K)\cap \edgesexti$ is coherent since $|D_{K,\edge}| = |D_\edge|$.
Additionally, we introduce the following reconstruction operator from $\Hmeshi$ to $L_\mesh$ :
\begin{equation}
\begin{array}{l|ccl}\label{def:R-mesh-to-edgesi}
\mathcal R^{\mesh}_{\edgesi}: \quad & \Hmeshi & \to & L_\mesh \\[2ex]
& v & \mapsto & \displaystyle \mathcal R^{\mesh}_{\edgesi} v =
\sum_{\substack{K\in\mesh}} (\mathcal R^{\mesh}_{\edgesi} v)_K\ \characteristic_{K}
\end{array}
\end{equation}
where,
\[
(\mathcal R^{\mesh}_{\edgesi} v)_K =  \frac{1}{2}(v_\edge + v_{\edge '}) \quad \mbox{with}\quad
(\edge,\edge') \in (\edges(K)\cap\edgesi)^2
\]~\\
which again is coherent as we have $|D_{K,\edge}| = |D_{K,\edge'}| = |K|/2$.
\end{definition}

\begin{remark}
  Both operators are clearly linear. Additionally, for $q \in L_\mesh$ and $v \in \Hmeshi$,
  we compute $\mathcal R_{\mesh}^{\edgesi} q$ and $\mathcal R^{\mesh}_{\edgesi} v$ through convex combinations of the values $(q_K)_{K \in \mesh}$ and $(v_\edge)_{\edge\in\edgesi}$ respectively.
\end{remark}~\\
\noindent Let us prove the following stability result,

\begin{lemma}[\textbf{Stability in $\bs{L^p, p\in [1,\infty]}$}] \label{lem-stab-recons}
Let $q\in L_\mesh$ and  $v \in \Hmeshi$ for some $i \in \llbracket1,d\rrbracket$. Then,
\[
\|\mathcal R_{\mesh}^{\edgesi} q\|_{L^p(\Omega)} \leq \|q\|_{L^p(\Omega)} \quad \mbox{and}
\quad\|\mathcal R^{\mesh}_{\edgesi} v\|_{L^p(\Omega)} \leq \|v\|_{L^p(\Omega)}
\]
\end{lemma}

\begin{proof}
  We prove the first estimate only since both proofs are almost identical.
  The case with $p=\infty$ is straightforward since we are dealing with convex combinations.
  Thus, let $q \in L_\mesh$ and $p\in[1,\infty)$. Using Jensen's inequality and
  $|D_{\edge}|=|D_{K,\edge}|+|D_{L,\edge}|$ for $\edge = K|L$ yields:

\begin{align*}
  \|\mathcal R_{\mesh}^{\edgesi} q\|^p_{L^p(\Omega)} & = \sum_{\substack{\edge \in \edgesinti\\\edge = K | L}}
  |D_\edge| |\alpha_\edge q_K + (1-\alpha_\edge)q_L|^p + \sum_{\substack{\edge \in \edgesexti\\\edge \in \edges(K)}}
  |D_{K,\edge}| |q_K|^p\ \\
  & \leq \sum_{\substack{\edge \in \edgesinti\\\edge = K | L}} \Big[|D_{K,\edge}| |q_K|^p + |D_{L,\edge}| |q_L|^p \Big]
  + \sum_{\substack{\edge \in \edgesexti\\\edge \in \edges(K)}} |D_{K,\edge}| |q_K|^p\ \\
  & \leq \sum_{K\in\mesh} |K| |q_K|^p = \| q\|^p_{L^p(\Omega)}
\end{align*}

\end{proof}

\noindent In the same manner we define, following \cite{conv-mac-NS-18}, a reconstruction operator
from $\Hmeshizero$ to $\Hmeshdij$ for some $(i,j) \in \llbracket1, d\small\rrbracket^2$.
The latter will allow us to deal with the quantities $(\rho_\edged, u_\edged)_{\edged \in \edgesd}$
appearing in the nonlinear velocity convection term in \eqref{schemevitesse:dis}.
Then, we may reformulate the trilinear form $b_\edges$ defined by \eqref{def:weak-conv-op-rv}
so as to obtain estimates necessary to derive the a priori estimates for the problem,
as well as the convergence of the scheme.

\begin{definition}[Reconstructions from $\Hmeshizero$ to $\Hmeshdij$]\label{def:vel-reco}
Let $\disc=(\mesh, \edges)$ be an admissible MAC mesh, and $i,j \in \llbracket1,d\rrbracket$.
We define the reconstruction operator from $\Hmeshizero$ to $\Hmeshdij$ as follows:
\begin{equation}
  \begin{array}{l|ccl}\label{eq:vel-reco}
  \mathcal R_\edgesd^{(i,j)}: \quad & \Hmeshizero & \to & \Hmeshdij \\[2ex]
  & v & \mapsto & \displaystyle \mathcal R_\edgesd^{(i,j)} v =
  \sum_{\substack{\edged \in \edgesdi\\\edged \perp \bfe\ej}} (\mathcal R_\edgesd^{(i,j)} v)_{D_\edged}\ \characteristic_{D_\edged},
  \end{array}
\end{equation}
where,
\[
(\mathcal R_\edgesd^{(i,j)} v)_{D_\edged} =\ \left\{ \begin{array}{l l } \alpha_\edged v_\edge
  + (1-\alpha_\edged) v_{\edge'} \quad  & \mbox{if}\quad  \edged = \edge|\edge' \in \edgesdinti \vspace{10pt} \\
  \alpha_\edged v_\edge \quad & \mbox{if}\quad  \edged \in \edgesd(D_\edge)\cap\edgesdexti \end{array} \right.
\]
and $\alpha_\edged \in [0,1]$ to which we give a meaning below (Lemma \ref{lem-new-b}).
\end{definition}

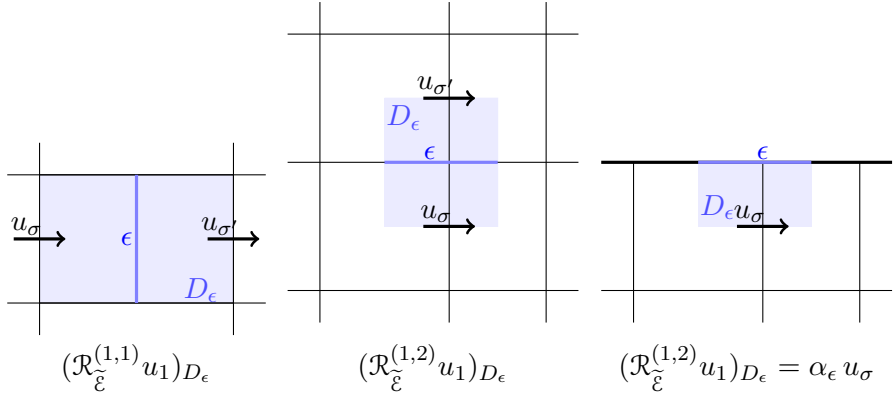
\begin{figure}[htb]
  \begin{center}
\begin{tikzpicture}[scale=0.85]
\centering
\draw[-,fill=blue!8] (0.5,1)--(3.5,1)--(3.5,3)--(0.5,3); \path(3.,1.2) node[blue!70]{$D_\edged$}; 
\draw[-](0,1)--(4,1); \draw[-](0,3)--(4,3); \draw[-](0.5,0.5)--(0.5,3.5); \draw[-](3.5,0.5)--(3.5,3.5); 
\draw[->, very thick](0.1,2)--(0.9,2); \draw[->, very thick](3.1,2)--(3.9,2);
\path(0.3,2.2) node{$u_\edge$}; \path(3.3,2.2) node{$u_{\edge'}$};
\draw[-, very thick, blue!50] (2,1)--(2,3); \path(1.85,2) node[blue] {$\edged$}; 
\path(2,0) node{$(\mathcal R_\edgesd^{(1,1)} u_1)_{D_\edged} $};
\end{tikzpicture}\hspace{5pt}
 \begin{tikzpicture}[scale=0.85]
\draw[-, blue!8, fill=blue!8] (1.5,1)--(3.25,1)--(3.25,3)--(1.5,3); \path(1.8,2.7) node[blue!70]{$D_\edged$}; 
\draw[-](0,0)--(4.5,0); \draw[-](0,2)--(4.5,2); \draw[-](0,4)--(4.5,4);
\draw[-](0.5,-0.5)--(0.5,4.5); \draw[-](2.5,-0.5)--(2.5,4.5); \draw[-](4,-0.5)--(4,4.5); 
\draw[->, very thick](2.1,1)--(2.9,1); \draw[->, very thick](2.1,3)--(2.9,3);
\path(2.3,1.2) node{$u_\edge$}; \path(2.3,3.2) node{$u_{\edge'}$};
\draw[-, very thick, blue!50] (1.5,2)--(3.25,2); \path(2.2,2.15) node[blue] {$\edged$}; 
\path(2.25,-1.2) node{$(\mathcal R_\edgesd^{(1,2)} u_1)_{D_\edged}$};
 \end{tikzpicture}\hspace{5pt}
 \begin{tikzpicture}[scale=0.85]
\draw[-, blue!8, fill=blue!8] (9.5,1)--(11.25,1)--(11.25,2)--(9.5,2); \path(9.8,1.3) node[blue!70]{$D_\edged$}; 
\draw[-](8,0)--(12.5,0); \draw[-, very thick](8,2)--(12.5,2);
\draw[-](8.5,-0.5)--(8.5,2); \draw[-](10.5,-0.5)--(10.5,2); \draw[-](12,-0.5)--(12,2); 
\draw[->, very thick](10.1,1)--(10.9,1); \path(10.3,1.2) node{$u_\edge$};
\draw[-, very thick, blue!50] (9.5,2)--(11.25,2); \path(10.5,2.15) node[blue] {$\edged$}; 
\path(10.25,-1.2) node{$(\mathcal R_\edgesd^{(1,2)} u_1)_{D_\edged} = \alpha_\edged\, u_\edge$};
\end{tikzpicture}
\caption{Reconstruction of the first component of the velocity ($i=1$ and $d=2$).}
\label{fig:recons}    \end{center} \end{figure}

\begin{lemma}[\textbf{Stability of $\bs{\mathcal R_\edgesd^{(i,j)}}$ in $\bs{L^p, p\in [1,\infty)}$}] \label{lem-stab-reconsij}
Let $p\in [1,\infty)$. Then, there exists $C_{\eta_\mesh} \geq 0$ depending only on $\eta_\mesh$
such that for any $v \in \Hmeshizero$ we have,
\[
\|\mathcal R_\edgesd^{(i,j)} v \|_{L^p(\Omega)} \leq C_{\eta_\mesh} \|v\|_{L^p(\Omega)}.
\]
\end{lemma}

\begin{proof}
Using $\alpha_\edged \in [0,1]$ and $(a+b)^p \leq 2^{p-1}(a^p+b^p)$ for $a,b >0$ yields:
\begin{align*}
  \| \mathcal R_\edgesd^{(i,j)}\|^p_{L^p(\Omega)} & \leq  2^{p-1} \sum_{\substack{\edged \in \edgesdinti\\\edged = \edge|\edge'}}
  |D_\edged| (|v_\edge|^p + |v_{\edge'}|^p) + \sum_{\substack{\edged \in \edgesdexti\\\edged \in \edgesd(D_\edge)}} |D_{\edged}| |v_\edge|^p\
\end{align*}
We reorder the sum then use the regularity of the mesh $\eta_\mesh$, to bound the measure $|D_{\edged}|+|D_{\edged'}|$
by $ C_{\eta_\mesh} |D_\edge|$, with $\edged$ and $\edged'$ being the two dual faces of $D_\edge$ orthogonal to $\bfe_j$,
\[
\| \mathcal R_\edgesd^{(i,j)}\|^p_{L^p(\Omega)} \leq  2^{p-1} \sum_{\substack{\edge \in \edgesi\\\edged, \edged' \in \edgesd(D_\edge)}}
\left[|D_\edged| + |D_{\edged'}| \right] |v_\edge|^p \leq  2^{p-1} C_{\eta_\mesh}\|v\|^p_{L^p(\Omega)}.
\]

\end{proof}

\noindent Using the reconstructions by $\mathcal R_\edgesd^{(i,j)}$, we give another useful
formulation of the trilinear form $b_\edges$ as follows:
\begin{lemma}[\textbf{Reformulation of $\bs {b_\edges}$}] \label{lem-new-b}
  Let $i \in \{1,..,d\}$ and $ (\rho,\bfu, \bfv,\bfw) \in L_{\mesh}\times\Hmeshizero^3$.
  Let $b^{(i)}_\edges(\rho\bfu,v_i, w_i)$ be defined by \eqref{def:weak-conv-op-rv}.
  Then there exists three reconstruction operators : $(\mathcal R_\edgesd^{(i,j)})^\rho$, $(\mathcal R_\edgesd^{(i,j)})^u$
  and $(\mathcal R_\edgesd^{(i,j)})^v$ in the sense of Definition \ref{def:vel-reco} such that,
\begin{equation}
 \begin{array}{r l }\label{eq:new-b}
   b^{(i)}_\edges(\rho\bfu,v_i, w_i) = - \ds \sum_{j=1}^d \int_\Omega (\mathcal R_\edgesd^{(j,i)})^\rho \rho_{\edges^{(j)}}
   (\mathcal R_\edgesd^{(j,i)})^u u_j (\mathcal R_\edgesd^{(i,j)})^v v_i\eth_j w_i \dx
 \end{array}
  \end{equation}
where $(\rho_{\edges^{(j)}}) = (\rho_\edge)_{\edge \in {\edges^{(j)}}}$ are the values defined on $\edges^{(j)}$
computed from $\rho \in L_{\mesh}$ using the upwind scheme (see \eqref{eq:divflux}).
\end{lemma}
\begin{proof}
  By the definition of $b^{(i)}_\edges(\rho\bfu,v_i, w_i)$ in \eqref{def:weak-conv-op-rv}
  and the duality property of the convective term \eqref{dd-conv} we have,

\begin{equation}
  \begin{array}{r l l}\nonumber
    b^{(i)}_\edges(\rho\bfu,v_i, w_i)  & = - & \ds \sum_{\edged \in  \edgesdij}  | D_\edged |
    (\rho u_j)_{D_\edged} (v_i)_{D_\edged} (\eth_j w_i)_{D_\edged} \vspace{5pt} \\
    & = - & \ds \sum_{j=1}^d \sum_{\substack{\edged \in  \edgesdij\\ \edged=\edge|\edge'}}  | D_\edged |
    \rho_{\edged} u_{\edged}v_{\edged} (\eth_j w_i)_{D_\edged}
  \end{array}
\end{equation}
From their definitions in \eqref{eq:rhoi-ui}, we have that $\rho_{\edged}$ and $u_{\edged}$ are convex
combinations of elements $\in \edges$ lying on the faces of $D_\edged$.
Therefore there exists two reconstruction operators $(\mathcal R_\edgesd^{(j,i)})^\rho$ and
$(\mathcal R_\edgesd^{(j,i)})^u$ such that
$((\mathcal R_\edgesd^{(j,i)})^\rho \rho_{\edges^{(j)}})_{D_\edged} = \rho_\edged$ and
$((\mathcal R_\edgesd^{(j,i)})^u u)_{D_\edged} = u_\edged$. Finally, $v_{\edged} $ is given either
by an interpolation or the upwind scheme $\eqref{dual:unkown-rv}$, hence we also have the
existence of an operator $(\mathcal R_\edgesd^{(i,j)})^v$ verifying
$((\mathcal R_\edgesd^{(i,j)})^v v)_{D_\edged} = v_\edged$ which concludes the reformulation of $b^{(i)}_\edges$.
\end{proof}~\\

Thanks to the reconstruction operators we may derivate the discrete equivalent of the classical
estimates on the trilinear form \cite{NS-temam} and we have,

 \begin{lemma}[\textbf{Estimates on $\bs {b_\edges}$}] \label{lem-estimb}
   There exists a constant $C_{\eta_{\mesh}} >0$, depending only on $\eta_\mesh$ -- as defined
   in \eqref{regmesh} -- such that for any
   $(\rho,\bfu, \bfv,\bfw) \in L_{\mesh}\times\boldsymbol{ E}_{\edges}\times\Hmeshzero^2$ we have,
 \begin{equation}
  \begin{array}{r l }\label{estimb-rho}
    |b_\edges(\rho\bfu, \bfv,\bfw)| & \leq  C_{\eta_{\mesh}}\, \|\rho\|_{L^{\infty}(\Omega)}\,
    \|\bfu\|_{L^4(\Omega)^d}\, \|\bfv\|_{L^4(\Omega)^d} \, \|\bfw\|_{1,\edges,0}  \\\\
    |b_\edges(\rho\bfu, \bfv,\bfw)| & \leq C_{\eta_{\mesh}}\, \|\rho\|_{L^{\infty}(\Omega)}\,
    \|\bfu\|_{1,\edges,0}\, \|\bfv\|_{1,\edges,0} \, \|\bfw\|_{1,\edges,0}.
  \end{array}
   \end{equation}
 \end{lemma}
 \begin{proof}
   The result is a direct adaptation of \cite{conv-mac-NS-18}. Let us consider
   $(\rho,\bfu, \bfv,\bfw) \in L_{\mesh}\times\boldsymbol{ E}_{\edges}\times\Hmeshzero^2$.
   For any $i \in \{1,..,d\}$, we use H\"older's inequality on the reformulation of
   $b^{(i)}_\edges(\rho\bfu,v_i, w_i)$ to yield,
\[
\left| b^{(i)}_\edges(\rho\bfu,v_i, w_i)  \right| \le \sum_{j=1}^d\Vert
(\mathcal R_\edgesd^{(j,i)})^\rho \rho_{\edges^{(j)}} (\mathcal R_\edgesd^{(j,i)})^u u
\Vert_{L^4(\Omega)} \Vert (\mathcal R_\edgesd^{(i,j)})^v v \Vert_{L^4(\Omega)^d} \Vert \eth_j w_i \Vert_{L^2(\Omega)}
\]
Using Lemma \ref{lem-stab-reconsij} and the fact that $(\mathcal R_\edgesd^{(j,i)})^\rho \rho_{\edges^{(j)}}$
is obtained by convex interpolations of values of $\rho$ yields the first estimate of \eqref{estimb-rho}.
The second estimate holds thanks to the discrete Sobolev inequality \cite[Lemma 3.5]{FV-book}:
\[
\|\bfv \|_{L^q(\Omega)} \le C(\eta_\mesh,q) \|\bfv\|_{1,\edges,0}   \quad \forall \bfv \in \Hmeshzero \quad \forall q\in [2,6]
\]
which allows us to control the $L^4$ norm by the discrete $\bfH^1_0$ norm.
\end{proof}\\
\noindent Finally,  following \cite{Fortin2012}, let us define the following Fortin
operator:\begin{align} \label{eq:projdual}
& \quad \quad \begin{array}{l|l} \displaystyle
    \widetilde{\Pi}_\edges^{(i)}: \quad & \quad  H^1_0(\Omega) \longrightarrow \Hmeshizero \\ [1ex]
    & \displaystyle \quad v_i \mapsto \widetilde{\Pi}_\edges^{(i)}v_i = \sum_{\edge\in\edgesi}
    \Big(\frac{1}{|\edge|}\int_{\sigma} v_i(\bfx) d\gamma(\bfx)\Big) \characteristic_{D_\edge}
\end{array}
\end{align}
satisfying the following criterias \cite{Fortin2012}: a continuity from $\bfH^1_0(\Omega)$
to  $\Hmeshzero$ and the preservation of the divergence in a sense explained below.
Hence for $\bfv \in \bfH^1_0(\Omega)$, we have,\\
\begin{equation}\label{fortin-prop1}
  \| \widetilde{\Pi}_\edges \bfv  \|_{1,\edges,0} \leq C_{\eta_\mesh} \| \nabla \bfv \|_{L^2(\Omega)^{d\times d}}
\end{equation}\\
where $C_{\eta_\mesh}$ is a constant depending on $\Omega$ and $\eta_\mesh$ (in a non-decreasing way) only.
Additionally $\widetilde{\Pi}_\edges$ verifies a divergence preservation property:
\begin{equation}\label{fortin-prop2}
  \quad \int_\Omega q \dive \bfv \dx = \int_\Omega q \dive_\mesh (\widetilde{\Pi}_\edges \bfv) \dx \qquad \forall q \in L_\mesh
\end{equation}\\
Setting $q = \characteristic_{K}$ for any $K\in \mesh$, an important consequence is:\\
\begin{equation}\label{fortin-prop3}
  \dive_\mesh(\widetilde{\Pi}_\edges \bfv) = \Pi_\mesh( \dive \bfv)
\end{equation}

\section{Estimates and Existence results}
In the present section we aim to prove that a solution to the discrete scheme
(\ref{eq:scheme:ro}) exists at any given time $t^n$. In order to deal with the
nonlinear terms in equations (\ref{eq:mass-rv},\ref{schemevitesse}) we use the
topological degree theory for which we recall some relevant advances.
These results can be found in \cite{nonlinfunctanal} and their  applications
to numerical schemes in \cite{degtopo97,latche-saleh-17,conv-mac-NS-18,degtopo08}\\

First we demonstrate some preliminary results and estimates regarding the solutions
of (\ref{eq:scheme:ro}). The following uniform estimate is a classical consequence
of the upwind choice in the mass equation and of the fact that the velocity is
divergence-free. Then we derive bounds on the velocity and pressure which are
needed for the existence result.\\

\subsection{Estimates on the discrete solutions}

\begin{lemma}[Estimate on the density - discrete maximum principle] \label{lem:estim-density}
Let $\mathcal{D}$ be an admissible MAC discretization of $\Omega$.
For $n\in\{1,\cdots, N-1 \}$, assume $\rho^{n} \in L_\mesh$ is such that
$0<\rho_{\min}\leq\rho^{n}\leq \rho_{\max}$.
If $\rho^{n+1} \in L_\mesh$ and $\bfu^{n+1} \in \Hmesh$ satisfy the discrete
mass balance \eqref{eq:mass-rv} and the divergence constraint \eqref{schemediv},
then
\begin{equation} \label{romin}
\rho_{\min}\leq \rho^{n+1}\leq \rho_{\max}.
\end{equation}\\
Moreover, provided that the discrete viscosity $\mu^{n}\in  L_\mesh$ defined as
in Definition \ref{def:dis:visc} satisfies $0 < \mu_{\min}\leq \mu^{n}\leq \mu_{\max}$,
we have
\begin{equation}\label{mumin}
  \mu_{\min}\leq \mu^{n+1}\leq \mu_{\max}
\end{equation}
\end{lemma}

\begin{proof}
Starting with the right side of the inequality \eqref{romin},
the discrete mass conservation equation \eqref{eq:mass-rv} given on $K \in \mesh$
is multiplied by $|K|(\rho_K^{n+1}-\rho_{\max})^{+}$. Summing over cells $K$ yields,
\begin{equation}\label{proof:romin:1}
  \underbrace{\displaystyle\sum_{K\in\mesh} \vphantom{\sum_{\edge\in\edges(K)}} \displaystyle\frac{\vert K\vert}{\deltat}
  (\rho^{n+1}_{K}-\rho_{K}^{n})(\rho_K^{n+1}-\rho_{\max})^{+}}_{\textstyle A}+ \underbrace{\displaystyle\sum_{K\in\mesh}
  \sum_{\edge\in\edges(K)}F_{K,\edge}(\rho_K^{n+1}-\rho_{\max})^{+}}_{\textstyle B} = 0
\end{equation}
The desired inequality is obtained by the first term in the equation. We
focus on the second term, denoted by $B$, which we reformulate using the
divergence constraint \eqref{schemediv:dis} as,

\begin{equation}\nonumber
  B = - \displaystyle\sum_{K\in\mesh}  \sum_{\edge\in\edges(K)}\vert \edge \vert u_{K,\edge}
  ( \rho_{K}^{n+1} -  \rho_{\edge}^{n+1}) (\rho_K^{n+1}-\rho_{\max})^{+}
\end{equation}
Let us denote $\sign(x) = 1$ if $x \geq 0$, and $0$ otherwise. Therefore,
\begin{equation}\nonumber
  B = - \displaystyle\sum_{K\in\mesh}  \sum_{\edge\in\edges(K)} \sign(\rho_K^{n+1}-\rho_{\max})
  \vert \edge \vert u_{K,\edge} (\rho_K^{n+1}-\rho_{\max})
  (( \rho_{K}^{n+1} -\rho_{\max}) -  (\rho_{\edge}^{n+1} -\rho_{\max}) )
\end{equation}
Using $2a(a-b) = a^2 + (a-b)^2 - b^2$ and (\ref{schemediv:dis}) we end up with,
\begin{equation}\nonumber
  B = - \dfrac{1}{2} \displaystyle\sum_{K\in\mesh}  \sum_{\edge\in\edges(K)}
  \sign(\rho_K^{n+1}-\rho_{\max})\vert \edge \vert u_{K,\edge}
  (( \rho_{K}^{n+1} -\rho_{\edge}^{n+1})^2 -  (\rho_{\edge}^{n+1} -\rho_{\max})^2 )
\end{equation}
We perform a discrete integration by parts by rearranging the sum over the edges,
\begin{align*}
  B = & - \dfrac{1}{2} \displaystyle\sum_{\substack{\edge\in\edges\\\edge = K | L}}
  \vert \edge \vert u_{K,\edge} \Big( \sign(\rho_K^{n+1}-\rho_{\max})
  (( \rho_{K}^{n+1} -\rho_{\edge}^{n+1})^2 -  (\rho_{\edge}^{n+1} -\rho_{\max})^2 ) \\ \vspace{-15pt}
& \qquad \qquad \qquad  \qquad\qquad- \sign(\rho_L^{n+1}-\rho_{\max})
  (( \rho_{L}^{n+1} -\rho_{\edge}^{n+1})^2 -  (\rho_{\edge}^{n+1} -\rho_{\max})^2 )  \Big)
\end{align*}
This term vanishes if $\rho_{K}^{n+1},\rho_{L}^{n+1} \leq\rho_{\max}$ and positive
if $\rho_{K}^{n+1},\rho_{L}^{n+1} \geq\rho_{\max}$ (owing to the assumption that
the upwind scheme is used for the density). In the remaining cases, using the
upwind scheme assumption and the order between the values
$\rho_{K}^{n+1},\rho_{L}^{n+1}$ and $\rho_{\max}$ we show that the term is positive.
Thus we conclude with the negativity of $A$ in \eqref{proof:romin:1} which in turn
yields that,

\begin{equation}\nonumber
  \displaystyle\sum_{K\in\mesh} \displaystyle\frac{\vert K\vert}{\deltat}
  (\rho^{n+1}_{K}-\rho_{\max})(\rho_K^{n+1}-\rho_{\max})^{+} \leq \displaystyle\sum_{K\in\mesh}
  \displaystyle\frac{\vert K\vert}{\deltat}
  (\rho_{K}^{n}-\rho_{\max})(\rho_K^{n+1}-\rho_{\max})^{+} \leq 0
\end{equation}
Therefore
$(\rho_{K}^{n+1}-\rho_{\max})(\rho_K^{n+1}-\rho_{\max})^{+} \leq 0 \quad \forall K \in \mesh$,
and  $\rho_K^{n+1} \leq \rho_{\max}$. \\\\ Using similar arguments, the mass
conservation equation is multiplied by $|K|(\rho_K^{n+1}-\rho_{\min})^{-}$ to
conclude with $\rho_K^{n+1} \geq \rho_{\min}$. Finally, the extreme value
theorem combined with $\rho^{n},\rho^{n+1} \in [\rho_{\min}, \rho_{\max}]$
and $\mu^{n} \in [\mu_{\min}, \mu_{\max}]$ yield the bound  $\mu_{\min}\leq \mu^{n+1}\leq \mu_{\max}$. ~\\
\textcolor{white}{?}  \end{proof}

\begin{remark}
  As a consequence, the estimate \eqref{mumin} holds for the discrete
  viscosity $\mu^{n+1}_{ij}$ by being a reconstruction of $\mu^{n+1}$ \eqref{def:muedgesd}--\eqref{def:muedgedext}.
\end{remark}

\begin{remark}
  The result still holds for non-homogeneous boundary conditions under the assumption
  that $\rho_{\edge}^{n+1}$ is given by $\rho_{in}^{n+1}$ on $\Gamma_{in} \neq \emptyset$
  and verifies $\rho_{\min} \leq \rho_{in}^{n+1} \leq \rho_{\max}$.
\end{remark}

From the discrete mass balance equation we also derive the following estimates,
including a weak $BV$ inequality. Those bounds will be required to deal with the
density variations in the convergence proof. In the case of nonlinear flux functions,
we may still be able to bound the total variation of the density using Lemma 5.5 from \cite{FV-book}.

\begin{lemma}[Weak \texorpdfstring{$\bs{BV}$}{d} estimate for the density]
  Any solution $(\rho,\bfu,p)$ to the discrete scheme \eqref{eq:scheme:dis}
  satisfies the following equality, for all $K\in\mesh$ and $0 \leq n \leq N-1$:
\begin{equation} \label{eq:roka}
\frac{\vert K\vert}{2\deltat} \bigl[(\rho_{K}^{n+1})^{2}-(\rho_{K}^{n})^{2}\bigr]
+\frac 1 2 \sum_{\edge\in\edges(K)}\vert \edge\vert\, (\rho_{\edge}^{n+1})^2\, u_{K,\edge}^{n+1}
+\mathcal{R}_K^{n+1}=0,
\end{equation}
with the remainder term $\mathcal{R}_K^{n+1}$ given by,
\begin{equation} \label{eq:rk}
\mathcal{R}_{K}^{n+1}=\frac{\vert K\vert}{2\deltat} \bigl(\rho_{K}^{n+1}-\rho_{K}^{n}\bigr)^2
- \frac 1 2 \sum_{\edge\in\edges(K)} \vert \edge\vert\, \bigl(\rho_{\edge}^{n+1}-\rho_{K}^{n+1}\bigr)^2
\ u_{K,\edge}^{n+1}.
\end{equation}
As a consequence, we get that
\begin{equation} \label{eq:rhobv}
\frac 1 2 \sum_{K\in\mesh} \vert K\vert (\rho_{K}^{n+1})^{2}
+\frac {\deltat }{2} \sum_{\substack{\edge\in\edgesint \\ \edge=K|L} }
      \vert \edge\vert\, (\rho_{L}^{n+1}-\rho_{K}^{n+1})^{2}\ \vert u_{K,\edge}^{n+1}\vert
+ \mathcal{R}_{\rho}^{n+1}=\frac 1 2 \sum_{K\in\mesh}\vert K \vert (\rho_{K}^{n})^2,
\end{equation}
where $\displaystyle \mathcal{R}_{\rho}^{n+1}=\frac{1}{2}
\sum_{K\in\mesh}\vert K \vert (\rho_{K}^{n+1}- \rho_{K}^{n})^2 \ge 0$.
Thus, the following weak $BV$ estimate holds:
\begin{equation} \label{eq:bvweak}
\sum_{n=1}^N \delta t  \sum_{\substack{\edge \in \edgesint \\ \edge=K|L}}
|\edge|\, \bigl(\rho_L^{n} -\rho_K^{n}\bigr)^2\  |u_{K,\edge}^{n}| \leq C,
\end{equation}
where  $C\ge 0$ depends only on the initial data $\|\rho_0 \|_{L^2}$.
\end{lemma}

\begin{proof}
Let us multiply the discrete mass balance equation \eqref{eq:mass-rv}
by $\vert K\vert \rho_{K}^{n+1}$.
We deal with the discrete time derivative term by using the identity
$2(a^{2}-ab) = (a^{2}-b^{2})+(a-b)^{2}$ with $a = \rho_K^{n+1}$
and $b =\rho_K^{n}$. We then apply the identity $2 ab = a^{2} + b^{2} - (a-b)^{2}$
with $a =\rho_K^{n+1}$ and $b =\rho_\edge^{n+1}$ for the primal mass fluxes in
the second term. We obtain \eqref{eq:roka}--\eqref{eq:rk} by noting that,
\[
\sum_{\edge\in\edges(K)}\vert \edge\vert (\rho_{K}^{n+1})^{2}u_{K,\edge}^{n+1} = (\rho_{K}^{n+1})^{2}
(\dive_\mesh \bfu)_K = 0.
\]
We then sum \eqref{eq:roka} over the cells $K\in \mesh$. This yields all
the terms of \eqref{eq:rhobv} but the second term. The latter is obtained from
reformulating the sum in the second term from $\sum_{M\in\mesh} \mathcal{R}_{K}^{n+1}$
to obtain,
\[
- \frac 1 2 \sum_{\substack{\edge \in \edgesint \\ \edge=K|L}} \vert \edge\vert\,
\left(\bigl(\rho_{\edge}^{n+1}-\rho_{K}^{n+1}\bigr)^2 - \bigl(\rho_{\edge}^{n+1}-\rho_{L}^{n+1}\bigr)^2 \right)
\ u_{K,\edge}^{n+1}
\]
and using the definition of the upwind approximation of $\rho_{\edge}$. Thanks
to the boundary conditions on $\bfu$, the terms on $\edge \in \edgesext$ vanish.
Finally, the second term in \eqref{eq:roka} vanishes from the local conservativity.
The last estimate -- the weak $BV$ estimate  -- is obtained by summing
\eqref{eq:rhobv} over $n$.
\end{proof}\\

\begin{lemma}[Discrete \texorpdfstring{$\bs{L^2(H_0^1) / L^{\infty}(L^{2})}$}{d} velocity estimates] \label{lem:est-vit-ro}
There exists $C > 0$ depending only on $\bfuini$, $\rho_0$, $\mu_0$ and $\bff$
such that, for any function $\bfu\in \bfX_{\edges,\deltat}$ satisfying \eqref{eq:scheme:ro},
the following estimates hold:
\begin{align} \label{estiLdeux:ro} &
\|\bfu\|^2_{L^{2}(0,T;\Hmeshzero)} = \sum_{n=0}^{N-1} \deltat\ \Vert \bfu^{n+1} \Vert_{1,\edges,0}^2 \leq C,
\\ \label{estiLinfiny:ro} &
\|\bfu\|_{L^{\infty}(0,T;L^{2}(\Omega)^{d})} = \max_{0\leq n \leq N-1} \Vert \bfu^{n+1} \Vert_{L^{2}(\Omega)^{d}} \leq C.
\end{align}
\end{lemma}
\begin{proof}
  We start from \eqref{schemevitesse:dis} to derivate a discrete kinetic energy
  balance. First, we multiply \eqref{schemevitesse:dis} by $(u^{n+1}_{\edge})$,
  for $i = 1,\dots d$, $\forall \edge \in \edgesinti$ and rearrange the resulting
  relation to obtain,\\

  \begin{multline*}
  \Big(\dfrac{1}{\deltat } (\rho_{D_{\edge}}^{n+1} - \rho_{D_{\edge}}^{n} ) + \dfrac{1}{|D_\edge|}
  \sum_{\edged\in\edges(D_\sigma)} F^{n+1}_{\edge,\edged} \Big) (u_{\edge}^{n+1})^{2} \\
  + \dfrac{1}{\deltat } \rho_{D_{\edge}}^{n} u_{\edge}^{n+1} (u_{\edge}^{n+1} - u_{\edge}^{n})
  - \dfrac{1}{|D_\edge|} \sum_{\edged\in\edges(D_\sigma)}  F^{n+1}_{\edge,\edged} u_{\edge}^{n+1} (u_{\edge}^{n+1} - u_{\edged}^{n+1}) \\
  - (\dive_\edges(\mu^{n+1}_\edgesd\mathbf{D}_\edgesd(\bfu^{n+1})))_{D_\edge}u_{\edge}^{n+1}
  + (\nabla \ p^{n+1})_{D_\edge}u_{\edge}^{n+1} = f_{\edge}^{n+1}u_{\edge}^{n+1}
\end{multline*}

The first term vanishes since it can be expressed as \eqref{eq:mass-rv:dis}
averaged over $K$ and $L$ for $\sigma = K | L$, assuming a suitable choice for
$\rho_{D_{\edge}}^{k}$. For the second and third terms we use $2ab = a^2 + (a-b)^2 -b^2$.
We use \eqref{eq:mass-rv:dis} once again to deal with the term
$ - \frac{1}{2}\sum_{\edged\in\edges(D_\sigma)} F^{n+1}_{\edge,\edged}  (u_{\edge}^{n+1})^{2}$.
Thanks to the approximations
$u^{n+1}_{\edged} = (u^{n+1}_\edge + u^{n+1}_{\edge'})/2$ when $\edged = \edge | \edge'$,
the term $ (u_{\edge}^{n+1} - u_{\edged}^{n+1})^2 - (u_{\edged}^{n+1})^2$ yields
$-u_{\edge}^{n+1}u_{\edge'}^{n+1}$. We are left with the discrete kinetic energy identity,

\begin{multline}\label{discrete-kinetic-nrg}
  \dfrac{1 }{2 \deltat | D_{\edge}|} \Big( \rho_{D_{\edge}}^{n+1} (u_{\edge}^{n+1})^2
  + \rho_{D_{\edge}}^{n}(u_{\edge}^{n+1}-u_{\edge}^{n})^2 - \rho_{D_{\edge}}^{n} (u_{\edge}^{n})^2 \Big)
  - \dfrac{1}{2 |D_\edge|} \sum_{\edged\in\edges(D_\sigma)}  F^{n+1}_{\edge,\edged} u_{\edge}^{n+1} u_{\edge '}^{n+1}\\
  - (\dive_\edges(\mu^{n+1}_\edgesd\mathbf{D}_\edgesd(\bfu^{n+1})))_{D_\edge}u_{\edge}^{n+1}
  + (\nabla \ p^{n+1})_{D_\edge}u_{\edge}^{n+1} = f_{\edge}^{n+1}u_{\edge}^{n+1}
\end{multline}

Let us sum \eqref{discrete-kinetic-nrg} over the faces $\edge\in\edgesinti$ for
$i=1,\ldots,d$ and over time for $n = 0,\ldots,M$ with $M \leq N-1$. The convection
term can be rearranged as,
\[
\sum_{n=0}^{M} \sum_{i=1}^{d} \sum_{\edged\in\edgesdinti} (F^{n+1}_{\edge,\edged}+F^{n+1}_{\edge ',\edged}) u_{\edge}^{n+1} u_{\edge '}^{n+1}
\]
and thus vanishes by conservativity of the numerical flux along $\edged = \edge | \edge'$.
The duality property \eqref{eq:duality-dual} combined with the Korn inequality \eqref{eq:korn}
allows us to retrieve the $\Hmeshzero$ norm of $\bfu^{n+1}$ in the third term.
The pressure term vanishes thanks to \eqref{eq:duality} and \eqref{schemediv}.
Finally, from the positivity of the density we have $\rho_{D_{\edge}}^{n}(u_{\edge}^{n+1}-u_{\edge}^{n})^2 \geq 0$.
Hence we get,
\begin{multline*}
 \frac{1}{2} \sum_{i=1}^{d} \sum_{\edge\in\edgesinti} \vert D_{\edge}\vert \rho_{D_{\edge}}^{(M+1)}(u_{\edge}^{(M+1)})^{2}
 -  \frac{1}{2} \sum_{i=1}^{d} \sum_{\edge\in\edgesinti} \vert D_{\edge}\vert  \rho_{D_{\edge}}^{(0)}(u_{\edge}^{(0)})^{2}\\
 + \mu_{min}\frac{\deltat}{2} \sum_{n=0}^{M}\|\bfu^{n+1}\|_{1,\edges,0}^{2}
\le \sum_{n=0}^{M}\sum_{i=1}^{d} \sum_{\edge\in\edgesi}\deltat \vert D_{\edge}\vert f_{\edge}^{n+1} u_{\edge}^{n+1}.
\end{multline*}
By Lemma \ref{lem:estim-density} and thanks to the Cauchy-Schwarz, discrete
Poincar\'e and Young inequalities, we then get the existence of $C>0$ depending
only on $\Omega$ and $\mu_{min}$ such that
\[
\frac{\mu_{min}}{4} \sum_{n=0}^{M} \deltat\ \|\bfu^{n+1}\|_{1,\edges,0}^{2}+  \rho_{\min}\|\bfu^{(M+1)}\|_{L^{2}(\Omega)}^{2}
\leq  \rho_{\max}\|\bfu_0\|_{L^{2}(\Omega)}^{2}  + C \, \|\bff\|_{L^{2} (0,T;L^{2}(\Omega)^{d})}^{2}.
\]
On one hand, we can conclude with the $L^{\infty} (L^{2} )$ estimate
\eqref{estiLinfiny:ro} ; on the other hand, taking $M=N-1$  yields the
$L^{2} (\Hmeshzero)$ estimate \eqref{estiLdeux:ro}.
 \end{proof}\\

\medskip
\begin{lemma}[Non-uniform bound on the pressure]
\label{lem:est:pressure:ro}
Let $(\rho^{n},\bfu^{n},p^{n})\in L_{\mesh}\times\Hmeshzero\times L_{\mesh,0}$
be given and assume that $(\rho^{n+1},\,\bfu^{n+1},\, p^{n+1})\in L_{\mesh}\times\Hmeshzero\times L_{\mesh,0}$ satisfies
  \eqref{eq:scheme:ro}.
  Then there exists a constant $C > 0$ depending only on $\rho_{0}$, $\mu_{0}$,
  $\deltat$, $\eta_{\mesh}$, $\bff$, and $\Omega$ such that:
  \begin{equation}
    \label{es:pres}
    \|p^{n+1}\|_{L^{2}(\Omega)}\leq C
  \end{equation}
\end{lemma}

\begin{proof}
Let $(\rho^{n+1},\bfu^{n+1},p^{n+1})$ verify $n+1-$th step of the discrete scheme
\eqref{eq:scheme:ro}. Since $p^{n+1}\in L^2_0(\Omega)$, then there exists a unique
weak solution in $\bfH^1(\Omega)$ to the Poisson problem $-\Delta \bfw = p^{n+1}$
on $\Omega$ with homogeneous neumann boundary conditions on $\partial \Omega$
\cite[Theorem III.4.3]{BoyerFabrie-book}. From this solution we may construct
$\bfvarphi \in H^1_0(\Omega)^d$ so that,
\[
\dive \bfvarphi = p^{n+1} \quad \mbox{and} \quad \|\nabla \bfvarphi\|_{L^{2}(\Omega)^{d\times d}}
\leq C \|p^{n+1}\|_{L^{2}(\Omega)}\] with $C$ depending only on $\Omega$
(see \cite[Lemma 5.4.2]{babuskaaziz} for the construction process).
We then choose $\bfv = \widetilde{\Pi}_{\edges}\bfvarphi$ as a test function in
\eqref{qdm:weakp}.\\\\ Using the preservation of the divergence property of the
Fortin operator \eqref{fortin-prop3}, we obtain $\| p^{n+1} \|_{L^{2}(\Omega)}^2  = T_1 + T_2 + T_3 +T_4$
with
\begin{align*}
T_1 &= \int_{\Omega} \eth_{t}(\rho\bfu)^{n+1} \cdot   \bfv \dx,  & T_2 &= b_\edges ((\rho\bfu)^{n+1}, \bfu^{n+1} ,\bfv) \\
T_3 &=  \int_\Omega \mu^{n+1}_\edgesd  \bs{D_\edgesd}(\bfu^{n+1}) : \bs{D_\edgesd}(\bfv) \dx & T_4 &= -
   \int_{\Omega}\bff_{\edges}^{n+1} \cdot   \bfv \dx
\end{align*}
Let us focus on $T_1$; by the Cauchy-Schwarz inequality and thanks to the
$L^\infty(0,T;L^\infty(\Omega))$ bound \eqref{romin} on $\rho$ we deal with
$(\rho\bfu)^{n+1}$ to get
\[
\vert T_1 \vert \le  \frac {\rho_{\max}}{\deltat} \sum_{i=1}^d \int_\Omega
(\vert u_i^{n+1} \vert + \vert u_i^{n} \vert) \vert v_i \vert \dx \le
\frac{2 \rho_{\max}}{\deltat} \Vert \bfu \Vert_{L^\infty(0,T,(L^2(\Omega)^d)} \Vert \bfv  \Vert_{L^2(\Omega)^d}.
\]
We combine the fact that $\widetilde{\Pi}_{\edges}$ is continuous in $H^1_0(\Omega)$
with the Poincar\'e inequality to obtain $ \Vert \bfv \Vert_{L^2(\Omega)^d} \leq C \|\nabla \bfvarphi\|_{L^{2}(\Omega)^{d\times d}}$.
By the $L^\infty(L^2)$ estimate \eqref{estiLinfiny:ro}  we find $C_1 \ge 0$ depending only on $\rho_{0}$,
$\deltat$, $\eta_{\mesh}$, $\bff$, and $\Omega$ such that $ \vert T_1 \vert \le C_1 \| p^{n+1} \|_{L^{2}(\Omega)}$.\\
From estimate \eqref{estimb-rho} on the trilinear form we obtain a bound on $T_2$:
\begin{align*}
  &\displaystyle \vert T_2 \vert \leq C_{\eta_{\mesh}} \|\rho^{n+1}\|_{L^{\infty}} \|\bfu^{n+1} \|_{1,\edges,0}^{2}
  \|  \bfv\|_{1,\edges,0} \le C_2(\rho_0,\eta_{\mesh},\bff)\| p^{n+1} \|_{L^{2}(\Omega)}.
\end{align*}

again by using the $L^\infty(L^\infty)$ bound on the density \eqref{estiLdeux:ro}
and the uniform $L^2(0,T;H^1_0(\Omega))$ bound on $\bfu$. We deal with the remaining
terms similarly: using the symmetry property \eqref{eq:symm-grad} of the discrete
operator $\bs{D_\edgesd}$ on $T_3$ yields,
\begin{align*}
\displaystyle T_3 & =  \int_\Omega \mu^{n+1}_\edgesd \bs{D_\edgesd}(\bfu^{n+1}) : \bs{D_\edgesd}(\bfv) \dx\vspace{5pt}\\
& =  \frac{1}{2} \int_\Omega \mu^{n+1}_\edgesd \nabla_\edgesd\bfu^{n+1} : \nabla_\edgesd\bfv \dx +
\frac{1}{2} \int_\Omega \mu^{n+1}_\edgesd \nabla_\edgesd\bfu^{n+1} : \nabla^T_\edgesd\bfv \dx
\end{align*}

It follows by Cauchy-Schwarz : $\vert T_3 \vert \leq \mu_{max}\|\bfu^{n+1}\|_{1,\edges,0} \|\bfv\|_{1,\edges,0}
\le C_3(\mu_0,\eta_{\mesh},\bff)\| p^{n+1} \|_{L^{2}(\Omega)}$. We have for the last bound :
\[
T_4 \le \|\bff_{\edges}^{n+1}\|_{L^{2}(\Omega)^{d}} \|\bfv\|_{L^{2}(\Omega)^{d}} \dx \leq  C_4(\eta_{\mesh},\bff) \|p^{n+1}\|_{L^{2}(\Omega)}
\]
which concludes the proof.
\end{proof}
\medskip
\medskip

\subsection{Existence result at any given time}

In order to prove the existence of a solution to the scheme
\eqref{eq:scheme:ro} at any given time, we resort to a topological
degree argument.  In a finite-dimension context, we make use of the
Brouwer topological degree $d:\mathcal{A} \rightarrow \mathbb{Z}$,
where $\mathcal{A}$ refers to the set of triplets $(F,\mathcal{O},y)$
with $\mathcal{O}$ being an open bounded set of $\R^N$,
$F:\overline{\mathcal{O}} \rightarrow \R^N$ is a continuous function
and $y \in \R^N\setminus F(\partial \mathcal{O})$. The triplet is such
that $d(F,\mathcal{O},y)$ is meant to express the existence of
solutions of the problem $F(x)=y$, and possibly their locations in
$\mathcal{O}$ as well. Let us recall its relevant properties for our
problem,

\begin{definition}
Let $\mathcal{A} = \{ (F,\mathcal{O},y)\ $ such that $\ \mathcal{O} \subset \R^N\ $
open bounded, $\ F\in \mathcal{C}(\overline{\mathcal{O}}),\ y \notin F(\partial \mathcal{O}) \}$.
The Brouwer degree $d: \mathcal{A} \rightarrow \mathbb{Z}$ is defined to have
the following properties:
\begin{itemize}
\item $d((F,\mathcal{O},y)\neq 0$ implies that  $F^{-1}(y)\neq \emptyset$
\item $d((F(\lambda,\cdot),\mathcal{O},y(\lambda))$ is independent of
$\lambda$ if $F:[0,1]\times \overline{\mathcal{O}} \rightarrow \R^N$ and
$y:[0,1]\rightarrow \R^N$ are continuous, while $y(t) \notin F(\lambda,\partial \mathcal{O})$
for every $\lambda \in [0,1]$
\end{itemize}
\end{definition}

One can find the remaining properties in \cite{nonlinfunctanal}. The last property
refers to the homotopy invariance, and will allow us to demonstrate that there exists
at least one solution to (\ref{eq:scheme:ro}) by considering a simpler problem,
close to the unsteady Stokes equations.\\\\
The fixed-point theory may be seen as a particular case of the topological degree.
Uniqueness of the solution, if it is shown to exist by the topological degree,
isn't guaranteed.
\begin{theorem}[Existence of a solution]\label{theo:existence}
For a given $n\in\{1,\cdots,N-1\},$  let us assume that the density $\rho^{n}$
is such that $0 <\rho_{\min} \leq \rho_{K}^{n} \leq\rho_{\max}$ for all $K\in\mesh.$
Then the non-linear system \eqref{eq:scheme:ro} admits at least one solution
 $(\rho^{n+1},\bfu^{n+1},p^{n+1})$ in $L_{\mesh}\times\Hmeshzero\times L_{\mesh,0},$
 and any possible solution satisfies the estimates \eqref{romin},
\eqref{estiLdeux:ro} and \eqref{es:pres}.
\end{theorem}
\begin{proof}
 Let $N_{\mesh}=$ card$(\mesh)$ and $N_\edges=$ card$(\edgesint)$;
 we may identify $L_{\mesh}$ with $\R^{N_{\mesh}}$ and $\Hmeshzero$ with $\R^{N_{\edges}}$.
 Let $V=\R^{N_{\mesh}}\times\R^{N_{\edges}}\times\R^{N_{\mesh}}$.
 Let us introduce the function
 $F :  [0,1] \times V\rightarrow V$ defined by:
\[
 F(\lambda, \rho,\bfu,p)= \left|
 \begin{array}{ll}
  \displaystyle\frac{1}{\deltat} (\rho_{K}-\rho_{K}^{n})+\lambda \frac{1}{\vert K\vert}
  \sum_{\edge\in\edges(K)}F_{K,\edge}, & K\in\mesh \\[2ex]
  \displaystyle\frac{1}{\deltat}(\rho_{D_{\edge}} u_{\edge}-\rho_{D_{\edge}}^{n} u_{\edge}^{n})+
  \lambda \frac{1}{\vert D_{\edge}\vert} \ \sum_{\edged\in\tilde\edges(D_{\edge})} F_{\edge,\edged} u_{\edged} & \edge\in\edgesint\\
  \qquad \qquad \qquad -  \dive_\edges(\mu_\edgesd\mathbf{D}_\edgesd(\bfu)))_{D_\edge}+(\nabla p)_{\edge}- f_{\edge},\\[2ex]
    \displaystyle-\frac{1}{\vert K \vert}\sum_{\edge\in\edges(K)}\vert \edge\vert u_{K,\edge}
    +\frac{1}{\vert K\vert}\sum_{L\in\mesh}\vert L\vert\,  p_L,  &\ K\in\mesh.
  \end{array}\right.
\]
The source term is included in the expression of $F$, hence we have $y = 0_V$.
The function $F$ is continuous from $[0, 1] \times V $ to $V$ and the problem
$F (1,\rho, \bfu, p) = 0$  is equivalent to the system \eqref{eq:scheme:ro}.
Indeed, the first and second lines remain unchanged. Multiplying the third line
by $|K|$ then summing over the cells $K\in \mesh$, we end up with
\[
\displaystyle-\sum_{\edge\in\edgesext}\vert \edge\vert u_{K,\edge}
+ \text{card}(\mesh)\sum_{L\in\mesh}\vert L\vert\,  p_L = 0
\]
as the terms over $\edgesint$ have vanished by conservativity.  Using the fact that
$\bfu_{\edge}=0$ for $\edge \in \edgesext$ -- from the choice for $N_\edges$ and $V$
-- it follows that $\sum_{L\in\mesh}\vert L\vert\,  p_L=0$, enforcing the constraint
$p\in L_{\mesh,0}$. Hence we have $(\dive \bfu)_K = 0$, $\forall K\in \mesh$ giving \eqref{eq:scheme:ro}.\\\\
Also we note that the placement of $\lambda$ in the expression of $F$ takes all
the nonlinearities into account: the penalization of the convective term deals
with the product of $\rho^{n+1}$ and $\bfu^{n+1}$, and the penalization of the mass
balance equation allows us to go back to a constant density and viscosity.\\\\
In order to define the open set $\mathcal{O}$ we make use of the estimates
(\ref{romin}), (\ref{estiLdeux:ro}) and (\ref{es:pres}). Thanks to the properties
of convection fluxes, the problem $F(\lambda, \rho, \bfu, p) = 0$, satisfies
those estimates uniformly in $\lambda$.
Hence, we may introduce $C$ so that the bounded open set $\mathcal{O}$ can be defined as,
\[
 \mathcal{O}=\{ (\rho,\bfu,p)\in V \ \mbox{s.t.} \ \frac{\rho_{\min}}{2}< \rho< 2 \ \rho_{\max}, \
 \|\bfu\|_{1,\edges,0}< C \ \mbox{and}\ \| p\|_{L^{2}(\Omega)} < C\},
\]
with $y \notin F(1,\partial \mathcal{O})$.
In order to prove the existence of at least one solution to the scheme
\eqref{eq:scheme:ro}, it remains to show that $d((F(0,\cdot),\mathcal{O},0_V) \neq 0$.
For $\lambda = 0$ the density still remains variable regarding the space variable,
which prevents us from using directly the results from \cite{blanc05}, as done in
\cite{conv-mac-NS-18}.
Let us consider the function $G : (\rho, \bfu, p)\mapsto F (0,\rho, \bfu, p)$.
In order to show the existence of a solution to $G(\rho, \bfu, p) = 0_V$
we make use of the inverse function theorem. The function $G$ is differentiable
on $\mathcal{O}$, and its Jacobian matrix is given by:\\
\begin{equation*}
\mbox{Jac}\ G(\rho,\bfu,p)=
\left(\begin{array}{c|c}
\dfrac{1}{\deltat} Id_{\R^{N_{\mesh}\times N_{\mesh}}}
& 0 \\[1ex] \hline
A & {} \\
& S(\bfu,p) \\
0 \\
\end{array}\right)
\end{equation*}\\\\

where $A$ the matrix in $\R^{N_{\edges}\times N_{\mesh}}$ associated to the
differentiation of $(\rho_{D_\edge}u_\edge/\deltat)_{\edge \in \edgesint}$ and
of the viscous term (containing $\mu_\edgesd$) with respect to $(\rho_K)_{K\in\mesh}$.
The matrix $S(\bfu, p)\in \R^{(N_{\edges}+N_{\mesh}) \times(N_{\edges}+N_{\mesh})}$ is the
Jacobian matrix associated to the MAC discretization of the following Stokes
problem with a space-variable density and viscosity:\\\\
Find $(\bfu,p)$ such that $p\in L^2_0(\Omega)$ and
\begin{align*}
  \begin{array}{l l r }
 \rho(\bfx) \partial_t\bfu- \dive(\mu(\bfx)D(\bfu))+\nabla p=\bff & \mbox{ in} & \ \Omega\\
\dive \ \bfu=0&  \mbox{ in} &\ \Omega\\
\bfu=0&  \mbox{ on} &\ \partial\Omega
\end{array}
\end{align*}
where $\rho(\bfx), \mu(\bfx) \in L^\infty(\Omega)$ are approximated by
$\rho_{D_{\edge}}$ and $(\mu_\edged)_{\edged \in \edgesd(D\edge)}$ on each dual cell $D_{\edge}$.\\\\
Since $\lambda = 0$ we have $\rho(\bfx) = \rho^{n}(\bfx)$ and  $\mu(\bfx) = \mu^{n}(\bfx)$.
Hence we are dealing with given, and strictly positive, real quantities.
Therefore $S(\bfu, p)$ is invertible \cite{blanc05} and so is Jac $G(\rho, \bfu, p)$
by being a triangular block matrix.
This implies that the topological degree of $F (0,\rho, \bfu, p)$ is non-zero
and by the homotopy invariance, there exists at least one solution $(\rho, \bfu, p)$
to the equation $F (1,\rho, \bfu, p) = 0$, ie \ to the scheme \eqref{eq:scheme:ro}.
\end{proof}

\medskip
\section{Convergence of the scheme}

From now on let us consider a sequence of
time steps $(\deltat_{m})_\mnn$ and a sequence of
MAC grids $(\mathcal{D}_m)_\mnn= (\mathcal{M}_m,\mathcal{E}_m)_\mnn$ -- in the sense of
Definition \ref{def:MACgrid}. The sequence $(\deltat_m,\mathcal{D}_m)_\mnn$ will be
assumed to satisfy $(\deltat_m, h_{\mesh_m}) \rightarrow (0,0)$ for $m\rightarrow\infty$
in order to return to a continuous formulation of the domain $[0,T]\times\Omega$.
Therefore, for every $\mnn$, we consider the corresponding discrete scheme \eqref{eq:scheme:ro}
for $\deltat =\deltat_m$ and $\mathcal{D} = \mathcal{D}_m$. \\\\
So as to establish the convergence of the weak discrete scheme as
$(\deltat_m, h_{\mesh_m}) \rightarrow (0,0)$, we follow the steps from \cite{conv-mac-NS-18}
-- for the nonlinear convection term -- and \cite{latche-saleh-17}  -- for the mass balance equation --
by passing to the limit separately in the mass balance and momentum balance equations.
This final step requires the convergence of the discrete solutions $(\rho_m,\bfu_m)_\mnn$ as $m\rightarrow \infty$. \\\\
The strong convergence of the sequence of discrete velocities is obtained
using a compactness result from Annex \ref{annex:compac}. The latter may be applied
if an estimate on the time translates of the velocity holds, which we prove in the next subsection.

\subsection{Estimate on the time translates of the velocity}

\noindent In order to apply Theorem \ref{theo-compact-space} for the compactness
in $L^2$, we need to work around the term $\partial_t (\rho \bfu)$ from
\eqref{schemevitesse} to derive a bound on the time translates of the velocity.

\begin{lemma}
\label{lem:translates}
Let $(\rho,\bfu,p)$ be a discrete solution of \eqref{eq:scheme:ro}.
Then, for a given $\tau$ verifying $0 < \tau < T$ we have the following
estimate on the time translates of $\bfu$,
 \begin{equation}
   \int_0^{T-\tau} \int_\Omega  \vert \bfu(t,\bfx+ \tau) - \bfu(t,\bfx) \vert^2 \dx \dt \le C \sqrt{\tau + \deltat}
   \label{translates}
 \end{equation}
 where $C>0 $ only depends on $T$, $\bff$, $\rho_0, \mu_0$ and $\eta_\mesh$.
\end{lemma}

\begin{proof}

We proceed by analogy with the continuous case, e.g. \cite{BoyerFabrie-book}
to yield bounds on the time translates of the velocity.
Let $\tau > 0$ and $(\bfu, \rho ) \in X_{\edges,\deltat} \times Y_{\mesh,\deltat} $.
We note $\bfw(t,\bfx) =  \bfu (t + \tau,\bfx) -\bfu(t,\bfx)\in X_{\edges,\deltat} $.
In the continuous case, the estimate \eqref{translates} is obtained by bounding the term
\[
\int_0^{T-\tau} \int_\Omega \left(\rho(t,\bfx) \bfu(t,\bfx+ \tau) -\rho(t,\bfx)  \bfu(t,\bfx) \right) \cdot\bfw(t,\bfx) \dt.
\]
However in the context of the MAC scheme, the components of $\bfu$ are piecewise constant
on different meshes. Therefore we first deal with components separately.\\
Let $n_\tau$ be such that $T- \tau \in (t^{N-n_\tau -1}, t^{N-n_\tau}]$.
For any $t \in (0,T-\tau)$ we associate $(n,l)$ verifying $n =\lceil \frac{t}{\deltat}\rceil$
and $n+l =\lceil \frac{t + \tau}{\deltat}\rceil $. Thus we have
$t \in (t^{n-1}, t^{n}]$, $t+\tau \in (t^{n+l-1}, t^{n+l}]$ and we may use : $\rho(t) u(t) = \rho^{n} u^{n}$
and $\rho(t+\tau) u(t+\tau) = \rho^{n+l} u^{n+l}$ .
Writing $\rho(t,\bfx) u_i(t,\bfx+ \tau) -\rho(t,\bfx)  u_i(t,\bfx) =
\rho(t,\bfx+ \tau) u_i(t,\bfx+ \tau) -\rho(t,\bfx)  u_i(t,\bfx) - (\rho(t,\bfx+ \tau) - \rho(t,\bfx)) u_i(t,\bfx+ \tau)$,
we define $A^{(i)}_1$ and  $A^{(i)}_2$ as,\\
\begin{align*}
  & A^{(i)}_1(t) = \sum_{\edge \in \edgesinti} |D_\edge| \left( \rho^{n+l}_{D_\edge} u_{\edge}^{n+l} -\rho^{n  }_{D_\edge}  u_{\edge}^{n}\right)  w_{\edge}^{n}   \\
  & A^{(i)}_2(t) = \sum_{\edge \in \edgesinti} |D_\edge| \left( \rho^{n  }_{D_\edge} - \rho^{n+l}_{D_\edge} \right) u_{\edge}^{n+l}  w_{\edge}^{n}
\end{align*}
with $w_{\edge}^{n} = u_{\edge}^{n+l} - u_{\edge}^{n}$.\\
Therefore, we aim to bound the term,
\[ \sum_{i=1}^{d} \int_0^{T-\tau} A_1^{(i)}(t) + A_2^{(i)}(t) \dt \]
First, let us focus on $A^{(i)}_1(t)$. In order to invoke the momentum equation we need
to retrieve an expression with consecutive timesteps. Hence we reformulate $A^{(i)}_1(t)$ as,
\begin{align*}
  A^{(i)}_1(t) =\sum_{\edge \in \edgesinti} |D_\edge| w_{\edge}^{n} \sum_{k=n}^{n+l-1}
  \left( \rho^{k+1}_{D_\edge} u_{\edge}^{k+1} -\rho^{k  }_{D_\edge}  u_{\edge}^{k}\right)
\end{align*}
From the momentum equation \eqref{schemevitesse:dis}, we define
$(A^{(i)}_{1j})_j$, $j=1,\dots,4$ so that
$A^{(i)}_1(t) = A^{(i)}_{11}(t) + A^{(i)}_{12}(t) + A^{(i)}_{13}(t)+ A^{(i)}_{14}(t)$
with\\

\hspace{-30pt}\begin{minipage}{\textwidth}
\begin{flalign}\nonumber
\begin{array}{l l r l}
& A^{(i)}_{11}(t) = & & \displaystyle \sum_{k = n}^{n+l-1}\sum_{\edge \in \edgesinti} \deltat |D_\edge|f_{\edge}^{(k+1)} w_{\edge}^n  \\
& A^{(i)}_{12}(t) = &-& \displaystyle \sum_{k = n}^{n+l-1}\sum_{\edge \in \edgesinti} \deltat |D_\edge|(\partial_i p)^{k+1}_{\edge} w_{\edge}^n  \\
& A^{(i)}_{13}(t) = & & \displaystyle \sum_{k = n}^{n+l-1}\sum_{\edge \in \edgesinti} \deltat  |D_\edge| \dive_\edges
  (\mu\mathbf{D}(\bfu)))^{k+1}_{\edge} w_{\edge}^n  \\
& A^{(i)}_{14}(t) = &-& \displaystyle \sum_{k = n}^{n+l-1}\sum_{\edge \in \edgesinti} \deltat
  \sum_{\edged \in \edgesd(D_{\edge})} F^{k+1}_{\edge,\edged} u^{k+1}_{\edged} w_{\edge}^n
\end{array}
\end{flalign}
\end{minipage}\\~\\

From the definitions of $(n,l)$ we note that $\deltat l \leq \tau + \deltat$.
We deal with the term $\int_{0}^{T-\tau} A^{(i)}_{11}(t) \dt$ by using the Cauchy-Schwarz
inequality along with the inequality on $\deltat l$,
\begin{align}
  \int_{0}^{T-\tau} A^{(i)}_{11}(t) \dt & \le  \sum_{n=0}^{N-n_\tau} (\deltat)^2  \sum_{k = n}^{n+l-1}
  \Vert  f_i^{(k+1)}\Vert_{L^2(\Omega)}\Vert  w_i^n\Vert_{L^2(\Omega)}\nonumber  \\
  & \le  \sum_{n=0}^{N-n_\tau} \deltat \Vert  w_i^n\Vert_{L^2(\Omega)} \sum_{k = 0}^{N-1}
  \deltat \displaystyle \mathbb{1}_{[t^n,t^{n+l-1}]}(t^{k}) \Vert  f_i^{(k+1)}\Vert_{L^2(\Omega)}\nonumber\\
  & \le  \sum_{n=0}^{N-n_\tau} \deltat \Vert  w_i^n\Vert_{L^2(\Omega)} \sqrt{\deltat l}
  \Vert  f_i\Vert_{L^2(0,T;L^2(\Omega))} \le C(T) \sqrt{\tau + \deltat} \label{a11}
\end{align}
%
\noindent Owing to the fact that $u_i \in L^1(0,T; L^2(\Omega))$
(by Lemma \ref{lem:est-vit-ro}). The term $\sum_{i=1}^d A^{(i)}_{12}(t)$ vanishes thanks
to the discrete duality property \eqref{eq:duality} and to $\bfu$ being discrete-divergence-free
\eqref{schemediv:dis}. For the diffusive term we apply \eqref{eq:duality-dual}
followed by the Cauchy-Schwarz inequality. We then use \eqref{eq:symm-grad}
to retrieve an expression using $\nabla_\edgesd$ only,
\[
 \int_\Omega\dive_\edges(\mu\mathbf{D}(\bfu)))^{k+1} w^n \dx \le \mu_{\max}\Vert  u^{(k+1)} \Vert_{1,\edges, 0} \Vert \ \Vert w^n \Vert_{1,\edges, 0},
\]
Hence in order to deal with $\int_{0}^{T-\tau} A^{(i)}_{13}(t) \dt$ we proceed as in $A^{(i)}_{11}$ to obtain,
\begin{align}
  \sum_{i=1}^d\int_{0}^{T-\tau}  A^{(i)}_{13}(t) \dt
  & \le \mu_{\max} \sum_{n=0}^{N-n_\tau} \sum_{k=n}^{n+l-1}2 \deltat^2
  \Vert u^{(k+1)} \Vert_{1,\edges, 0}  \Vert w^n \Vert_{1,\edges, 0} \nonumber\\
  & \le \mu_{\max} \sum_{n=0}^{N-n_\tau}\deltat \Vert w^{n} \Vert_{1,\edges, 0}
  \sum_{k = 0}^{N-1} \deltat \displaystyle \mathbb{1}_{[t^n,t^{n+l-1}]}(t^{k}) \Vert u^{(k+1)} \Vert_{1,\edges, 0} \nonumber \\
  & \le C(T,\mu_0) \sqrt{\tau +\deltat} \Vert  \bfu \Vert_{L^2(0,T;\Hmeshzero)} \le C(T) \sqrt{\tau +\deltat} \label{a13}
\end{align}\\~
using $\bfu \in L^1(0,T; \Hmeshzero)$. The term $A_{14}(t)$ being associated to
$b_\edges((\rho\bfu)^{k+1}, \bfu^{k+1}, \bfw^n)$, we resort to Lemma \ref{lem-estimb} and \eqref{romin} to yield,
\[
A_{14}(t) = \sum_{i=1}^d A^{(i)}_{14}(t) \le \rho_{\max} \sum_{k=n}^{n+l-1} \deltat
\Vert \bfu^{(k+1)}\Vert_{1,\edges,0}^2 \Vert \bfw^n \Vert_{1,\edges,0}
\]
Therefore, by the Cauchy-Schwarz and  $\deltat l \leq \tau + \deltat$ inequalities
combined with $\mathbb{1}_{[t^n,t^{n+l-1}]}(t^{k+1}) = \mathbb{1}_{[t^{k-l+1},t^{k+1}]}(t^n)$,
we bound $\int_0^{T-\tau}  A_{14}(t) \dt$ as below,
\begin{align}
  \int_0^{T-\tau} A_{14}(t) \dt & \le \rho_{\max} \sum_{n=0}^{N-n_\tau}
  \sum_{k=n}^{n+l-1} \deltat^2 \Vert \bfu^{(k+1)}
    \Vert_{1,\edges,0}^2 \Vert \bfw^n \Vert_{1,\edges,0}\nonumber\\
 & \le \rho_{\max} \Big(\deltat  \sum_{k=0}^{N-1}   \Vert \bfu^{(k+1)}\nonumber
    \Vert_{1,\edges,0}^2\Big)\Big( \deltat  \sum_{n=0}^{N-n_\tau}
    \mathbb{1}_{[t^n,t^{n+l-1}]}(t^{k}) \Vert \bfw^n \Vert_{1,\edges,0}\Big)\nonumber\\
 & \label{a14} \le C(T,\rho_0) \sqrt{\tau + \deltat} 
\end{align}
\\

\noindent Similarly to $A_1^{(i)}$, we may write the term $A_2^{(i)}$ as,
\begin{align*}
  A^{(i)}_2(t) = \sum_{\edge \in \edgesinti} |D_\edge| u_{\edge}^{n+l}  w_{\edge}^{n}
  \sum_{k=n}^{n+l-1}\left( \rho^{n  }_{D_\edge} - \rho^{n+1}_{D_\edge} \right)
\end{align*}
Thanks to the discrete mass equation on the dual cells \eqref{mass-dual}
and to the discrete duality formula \eqref{eq:duality} we have:
\begin{align*}
  \begin{array}{r l r  l}
    A_2^{(i)}(t) & = & & \displaystyle\sum_{k=n}^{n+l-1} \deltat \sum_{\edge \in \edgesinti}
    \vert D_\edge \vert \dive_{D_\edge} ( \rho^{k+1}  \bfu^{k+1} ) u_{\edge}^{n+l}  w_{\edge}^n \dx\\
    & = & - & \displaystyle\sum_{k=n}^{n+l-1} \deltat \int_\Omega
    (\rho^{k+1}  \bfu^{k+1})_\edgesi \cdot \nabla_\edgesi  (u_i^{n+l}  w_i^n) \dx
  \end{array}
\end{align*}
 Now thanks to H\"older's inequality, we have :
\[
A_2^{(i)}(t) \le \rho_{\max} \deltat \sum_{k=n}^{n+l-1} \Vert \bfu^{(k+1)} \Vert_{L^6(\Omega)}
\Vert \nabla_\edgesi  (u_i^{n+l}  w_i^n) \Vert_{(L^{\frac 6 5}(\Omega))^d}
\]
We apply the Cauchy-Schwarz inequality then H\"older's inequality with $p=5$ for the gradient term.
This yields,
\begin{align*}
  A_2^{(i)}(t) & \le \displaystyle \rho_{\max} \Vert  \bfu  \Vert^{\frac 1 2}_{L^2(0,T;L^6(\Omega))}
  \sqrt{\deltat+\tau}   \Vert \nabla_\edgesi   (u_i^{n+l}  w_i^n)  \Vert_{(L^{\frac 6 5}(\Omega))^d}  \\
  & \le  \displaystyle \rho_{\max} | \Omega |^{\frac 1 6} \Vert  \bfu  \Vert^{\frac 1 2}_{L^2(0,T;L^6(\Omega))}
  \sqrt{\deltat+\tau}   \Vert \nabla_\edgesi   (u_i^{n+l}  w_i^n)  \Vert_{(L^{\frac 3 2}(\Omega))^d}
\end{align*}\\
Let us focus on the gradient term norm,\\
\[
\Vert \nabla_\edgesi (u_i^{n+l}  w_i^n)  \Vert_{(L^{\frac 3 2}(\Omega))^d}  \le \Vert  \nabla_\edgesi
( u_i^{n+l})  w_i^n \Vert_{(L^{\frac 3 2}(\Omega))^d} + \Vert  u_i^{n+l} \nabla_\edgesi  (w_i^n) \Vert_{(L^{\frac 3 2}(\Omega))^d}
\]\\
By H\"older's inequality with $p=4$ we have for both terms, 

\begin{align*}
  \Vert  \nabla_\edgesi  ( u_i^{n+l})  w_i^n \Vert_{(L^{\frac 3 2}(\Omega))^d}
   & \le  \Vert w_i^n \Vert_{L^{6}(\Omega)}  \Vert  \nabla_\edgesi u_i^{n+l} \Vert_{(L^{2}(\Omega))^d}\\
   & \le \Vert  \nabla_\edgesi u_i^{n+l} \Vert_{(L^{2}(\Omega))^d}^2 + \Vert w_i^n \Vert_{L^{6}(\Omega)}^2
\end{align*}

\begin{align*}
  \Vert  u_i^{n+l} \nabla_\edgesi  (w_i^n) \Vert_{(L^{\frac 3 2}(\Omega))^d}
   & \le  \Vert  u_i^{n+l} \Vert_{L^{6}(\Omega)} \Vert  \nabla_\edgesi  w_i^n \Vert_{(L^{2}(\Omega))^d}\\
   & \le \Vert  \nabla_\edgesi  w_i^n \Vert_{(L^{2}(\Omega))^d}^2 + \Vert  u_i^{n+l} \Vert_{L^{6}(\Omega)}^2
\end{align*}\\

\noindent Hence we were able to retrieve the norms in $\Hmeshizero$ of $u_i^{n+l}$
and $w^n$. Owing to the injection of $H^1_0(\Omega)$ into $L^6(\Omega)$ the estimates are sufficient.
We still require to perform the integration over $(0,T-\tau)$ on $A_2^{(i)}(t)$,\\
\begin{align*}
  \int_0^{T-\tau} \! \!A_2^{(i)}(t) \dt \le & C(\eta_\mesh, \rho_0)\sqrt{\deltat+\tau}
  \Big( \Vert u_i \Vert^2_{L^2(0,T;L^6(\Omega))}  +\Vert   u_i  \Vert_{L^2(0,T;  \Hmeshizero)}^2   \Big)
\end{align*}\\
Summing over $i= 1, \ldots, d$ yields the bound on $A_2(t)$:
\begin{align}
  \int_0^{T-\tau} A_2(t) \dt \le C(\eta_\mesh, \rho_0)\sqrt{\deltat+\tau} \label{a2}
\end{align}\\
%
We may now conclude with the bound on $\sum_{i=1}^{d} \int_0^{T-\tau} (A_1^{(i)}(t) + A_2^{(i)}(t)) \dt$ by summing
(\ref{a11}-i), (\ref{a13}-i) over $i= 1,\ldots, d$ with \eqref{a14} and \eqref{a2}. \end{proof}
%

\subsection{Convergence of the discrete scheme as \texorpdfstring{$\bs{\deltat, h_{\mathcal M} \rightarrow 0}$}{dt}.}

Before introducing the main theorem for the convergence of the
discrete scheme, we state a few results required for the passage to
the limit in the weak discrete equations : namely the consistency of
the discrete derivatives and the convergence of the
reconstructions. The former will assure the accuracy of the
approximation of the discrete derivatives, and using Taylor expansions
we will be able to recover the order of the truncation errors. The
latter will allow us to control the convergence of convex combinations
of the quantities $(\rho_m,\bfu_m)$ defined on $\mathcal{D}_m$.

\begin{lemma}[\textbf{Consistency of the discrete derivatives}] \label{lem-consistance}
Let $\disc=(\mesh, \edges)$ be an admissible MAC mesh and $\varphi \in C^\infty_c([0,T)\times \Omega)$.
For some $n \in \{0,N\}$ we note
$\varphi^n_\mesh = \Pi_\mesh \varphi(t^n,\cdot)$ and $\varphi^n_\edgesi = \Pi_\edgesi \varphi(t^n,\cdot)$
the projections of $\varphi$ given at $t^n$ on $\mesh$ and $\edgesi$ respectively.
Let
\[
\varphi_\mesh  = \sum_{n=0}^{N-1} \varphi^{n+1}_\mesh \characteristic_{(t^n,t^{n+1}]}
\quad \mbox{and} \quad \varphi_\edgesi = \sum_{n=0}^{N-1} \varphi^{n+1}_\edgesi \characteristic_{(t^n,t^{n+1}]}
\]
Then, for all $p\in [1,\infty]$,
\begin{flalign}\nonumber
\begin{array}{ l l  l l }
  & \|\eth_t \varphi_\mesh - \partial_t \varphi\|_{L^p([0,T]\times\Omega)}
  & + \quad \|\nabla_\edges \varphi_\mesh - \nabla \varphi\|_{L^p([0,T]\times\Omega)^d}
  & \leq C_1 (\deltat +  h_{\mesh})\vspace{10pt}\\
  & \|\eth_t \varphi_\edgesi - \partial_t \varphi\|_{L^p([0,T]\times\Omega)}
  & + \quad \|\nabla_\edgesd \varphi_\edgesi - \nabla \varphi\|_{L^p([0,T]\times\Omega)^d}
  & \leq C_1 (\deltat +  h_{\mesh}) \vspace{10pt}\\
  & \|\varphi^0_\mesh - \varphi(0,\cdot)\|_{L^p(\Omega)}
  & + \quad \|\varphi^0_\edges - \varphi(0,\cdot)\|_{L^p(\Omega)^d}
  & \leq C_2  h_{\mesh}
\end{array}
\end{flalign}
with $C_1$ and $C_2$ depending only on $\varphi,\Omega,T,p$.
\end{lemma}
\begin{proof}
Thanks to the regularity of $\varphi$, we may use Taylor's inequalities retrieve
 the desired bounds. Indeed,
\[
\|\eth_t \varphi_\mesh - \partial_t \varphi\|^p_{L^p} =
\sum_{n=0}^{N-1} \sum_{K\in\mesh} \int_{t^n}^{t^{n+1}}\int_K
\left|\frac{\varphi^{n+1}_K-\varphi^{n}_K}{\deltat} - \partial_t \varphi(t,\bfx)\right|^p \dx \dt
\]
Then there exists $C_1$ depending only on $\partial_{tt} \varphi$ such that,
\[
\|\eth_t \varphi_\mesh - \partial_t \varphi\|^p_{L^p} =
\sum_{n=0}^{N-1} \sum_{K\in\mesh} \int_{t^n}^{t^{n+1}}\int_K
\left| \partial_t \varphi(t^n,\bfx_K) - \partial_t \varphi(t,\bfx) + C_1\deltat \right|^p \dx \dt
\]
Using a Taylor inequality once again yields the constant $C_2$ depending only
on $\partial_{tt} \varphi$ and $\partial_{x}\partial_{t} \varphi$ and it follows
that,
\[
\|\eth_t \varphi_\mesh - \partial_t \varphi\|^p_{L^p} =
\sum_{n=0}^{N-1} \sum_{K\in\mesh} \int_{t^n}^{t^{n+1}}\int_K \left| C_2(\deltat + h_\mesh) + C_1\deltat \right|^p \dx \dt
\]
which concludes the bound on $\|\eth_t \varphi_\mesh - \partial_t \varphi\|_{L^p([0,T]\times\Omega)}$.
By a similar process and thanks to the regularity of $\varphi$, the derivation
of the remaining bounds is straightforward.
\end{proof}

\begin{lemma}[\textbf{Convergence of reconstructions in \texorpdfstring{$\bs{L^p, p\in [1,\infty)}$}{dt}}]
\label{lem-conv-recons}
Let $p\in [1,\infty)$ and a sequence $(q_m)_\mnn$ be such that for all $\mnn$, $q_m \in L_{\mesh_m}$.
If $q_m \longrightarrow \bar q$  as $m\rightarrow\infty$ in $L^p(\Omega)$ then,
\[
\mathcal R_{\mesh_m}^{\edges_m^{(i)}} q_m \rightarrow  \bar q\quad \mbox{as}\quad  m\rightarrow\infty \quad \mbox{in} \quad L^p(\Omega)
\]
Similarly, for a sequence of velocities $(v_m)_\mnn$ verifying $v_m \in H_{\edges_m^{(i)}}$ for all
$\mnn$ and $v_m$ strongly converging to $\bar v$ as $m\rightarrow\infty$ in $L^p(\Omega)$ then,
\[
\mathcal R^{\mesh_m}_{\edges_m^{(i)}} v_m \rightarrow  \bar v\quad \mbox{as}\quad  m\rightarrow\infty \quad \mbox{in} \quad L^p(\Omega)
\]
\[
\mathcal R^{(i,j)}_{\edgesd_m} v_m \rightarrow  \bar v\quad \mbox{as}\quad  m\rightarrow\infty \quad \mbox{in} \quad L^p(\Omega)
\]
\end{lemma}

\begin{proof}
The proofs being similar, let us show the first convergence result only.
Let $\varepsilon > 0$. Thanks to the density of $C^\infty_c(\Omega)$ in $L^p(\Omega)$,
let $\varphi \in C^\infty_c(\Omega)$ be the function verifying $\|\varphi  - \bar q\|_{L^p} < \varepsilon$. \\
\\
Thus using $\Pi_{\mesh_m}\bar q $ and $\Pi_{\mesh_m} \varphi$, respectively the
projection of $\bar q$ and $\varphi$ on $L_{\mesh_m}$ and noting that
$\mathcal R_{\mesh_m}^{\edges_m^{(i)}} q_m = (\mathcal R_{\mesh_m}^{\edges_m^{(i)}} \circ \Pi_{\mesh_m}) q_m$ yields,
\begin{flalign}\nonumber
 \begin{array}{r  l }
\| \mathcal R_{\mesh_m}^{\edges_m^{(i)}} q_m  - \bar q\|_{L^p}  \leq &
\|(\mathcal R_{\mesh_m}^{\edges_m^{(i)}} \circ \Pi_{\mesh_m}) q_m-(\mathcal R_{\mesh_m}^{\edges_m^{(i)}} \circ \Pi_{\mesh_m})\bar q \|_{L^p}\vspace{10pt}\\
& \quad + \|(\mathcal R_{\mesh_m}^{\edges_m^{(i)}} \circ \Pi_{\mesh_m}) \bar q -
(\mathcal R_{\mesh_m}^{\edges_m^{(i)}} \circ \Pi_{\mesh_m}) \varphi\|_{L^p} \vspace{10pt}\\
& \quad  \quad + \|(\mathcal R_{\mesh_m}^{\edges_m^{(i)}} \circ \Pi_{\mesh_m})\varphi - \varphi \|_{L^p} + \|\varphi - \bar q \|_{L^p}
 \end{array}
\end{flalign}
Using \eqref{fortin-prop1} and Lemma \ref{lem-stab-recons}, we conclude that the
 second and last terms in the right handside are bounded by $\varepsilon$. Moreover,
 from the strong convergence of $(q_m)_\mnn$ we may find $M_1$ such that for all
 $m\geq M_1$ we have $\|q_m  - \bar q \|_{L^p} < \varepsilon$.\\\\ From the
 regularity of $\varphi$ we resort to a Taylor's inequality to obtain a constant
 $C$ depending only on $p,\varphi, \Omega$ and $T$ such that
\[
\|(\mathcal R_{\mesh_m}^{\edges_m^{(i)}} \circ \Pi_{\mesh_m})\varphi - \varphi \|_{L^p} \leq C h_{\mesh_m}.
\]
Then there exists $M_2$ verifying $\|(\mathcal R_{\mesh_m}^{\edges_m^{(i)}} \circ \Pi_{\mesh_m})\varphi - \varphi \|_{L^p} < \varepsilon$
for all  $m\geq M_2$. Taking $M = \max{(M_1,M_2)}$, we can then conclude with
the strong convergence of $(R_{\mesh_m}^{\edges_m^{(i)}} q_m)_\mnn$ towards
$\bar q$ in $L^p([0,T]\times\Omega)$.
\end{proof}

\begin{remark}
Therefore, given the strong convergence of a discrete solution to a limit in
$L^p([0,T]\times\Omega)$, with $1\leq p < \infty$, we can conclude with the
strong convergence of its reconstruction towards this same limit.
\end{remark}

\begin{theorem}[Convergence of the discrete scheme]
\label{conv:ins}
Let $(\deltat_{m})_{m\in \xN}$ be a sequence of
time steps and $(\mathcal{D}_m)_\mnn= (\mathcal{M}_m,\mathcal{E}_m)_\mnn$ a sequence of
MAC grids (in the sense of Definition \ref{def:MACgrid}) such that  $\deltat_{m}\rightarrow 0$
and $h_{\mesh_m} \to 0$  as $m \to +\infty$~. We assume that there exists $\eta >0$
controlling the regularity of every MAC mesh in the sequence :  $\eta_{\mesh_m} \le
\eta$ for any $\mnn$ -- with $\eta_{\mesh_m}$ defined by \eqref{regmesh}.
Let $(\rho_m,\bfu_m)$ be a solution to  \eqref{eq:scheme:ro} for $\deltat=\deltat_{m}$ and $\mathcal{D}=\mathcal{D}_m$.
Then there exists $\bar\rho$ with $\rho_{\min}\leq\bar \rho\leq\rho_{\max}$
and $\bar \bfu \in L^{2}(0,T; \boldsymbol{E}(\Omega))$  verifying up to a subsequence:
\begin{list}{-}{\itemsep=0.ex \topsep=0.5ex \leftmargin=1.cm \labelwidth=0.7cm \labelsep=0.3cm \itemindent=0.cm}
\item the sequence $(\bfu_m)_\mnn$ converges to $\bar \bfu$ in $L^{2}(0,T; L^{2}(\Omega)^{d})$,
\item the sequence $(\rho_m)_\mnn$ converges to $\bar \rho$ in $ \in L^{2}(0,T;L^2(\Omega))$,
\item $(\bar \rho,\bar \bfu)$ is a solution to the weak formulation \eqref{massweak} and \eqref{qdmweak}.
\end{list}
\end{theorem}
\medskip
\paragraph{Step 1- Weak convergence of \texorpdfstring{$\bs{(\rho_m)_\mnn\ \text{in}\ L^{\infty} ( (0,T) \times \Omega )}$}{dt} and maximum principle.}
Let us consider the sequence $(\rho_{m} )_{\mnn}$ verifying \eqref{eq:scheme:ro} for $\deltat=\deltat_{m}$ and $\mathcal{D}=\mathcal{D}_m$.
\noindent Thanks to the bound \eqref{romin}, there exists a subsequence of
$(\rho_{m} )_{\mnn}$, denoted once again $(\rho_{m} )_{\mnn}$,
which converges $\star$-weakly to some function $\bar \rho$ in $L^{\infty} ((0, T )\times\Omega)$, i.e.:
\begin{equation}
 \lim_{m\rightarrow\infty}\int_{0}^{T}\int_{\Omega}\rho_{m}(t,\bfx)\varphi(t,\bfx)\dx\dt=\int_{0}^{T}\int_{\Omega}
 \bar \rho(t,\bfx)\varphi(t,\bfx)\dx\dt,\quad \forall\varphi\in L^{1}( (0,T) \times \Omega )
 \label{convergro}
 \end{equation}
Let $B$ be a borelian set of $(0, T )\times\Omega$ so that we can take
$\varphi(t,\bfx) = \characteristic_{B}(t,\bfx)$ in \eqref{convergro}. From the
bounds on $\rho_{m}(t,\bfx)$ given by \eqref{romin}, we have the non-negativity
of the integrals:
\begin{flalign}\nonumber
\begin{array}{l r l }
\ds \int_{0}^{T}\int_{\Omega}(\rho_m(t,\bfx)-\rho_{\min}) \characteristic_{B}(t,\bfx)\dx\dt \vspace{5pt}  \\ \vspace{2pt}
\quad \quad \quad \mbox{and} \quad \ds \int_{0}^{T}\int_{\Omega}(\rho_{\max}- \rho_m(t,\bfx)) \characteristic_{B}(t,\bfx)\dx\dt,  \quad \forall \mnn
\end{array}
\end{flalign}
\noindent Passing to the limit as $m \rightarrow \infty$ yields,
\begin{equation*}
 \int_{0}^{T}\int_{\Omega}(\bar \rho(t,\bfx)-\rho_{\min}) \characteristic_{B}(t,\bfx)\dx\dt  \ge 0\quad \mbox{ and } \quad
 \int_{0}^{T}\int_{\Omega}(\rho_{\max}-\bar \rho(t,\bfx)) \characteristic_{B}(t,\bfx)\dx\dt \ge 0
\end{equation*}
which is equivalent to $\rho_{\min}\leq\bar \rho(t,\bfx)
\leq\rho_{\max}$ a.e. in $ (0, T )\times \Omega$.\\
\paragraph{Step 2- Compactness of \texorpdfstring{$\bs{(\bfu^m)_\mnn}$ in $\bs{L^2(0,T;(L^2(\Omega))^d)}$}{dt}}
Let us apply the time compactness Theorem \ref{theo-compact-space} for $p = 2$,
with the Banach space $B = (L^{2}(\Omega))^{d}$ and the sequence of Banach spaces
$(X_m)_\mnn $ where $ X_m = \Hmeshzerom$, $m \in \xN $ is endowed with the norm
defined in \eqref{norm-rv}. By Lemma \ref{compac:astuce} and Theorem \ref{compac:space}
we have that any bounded sequence in $\Hmeshzeromi$ is relatively compact in
$L^{2}(\Omega)$ for all $i=1,..,d$. Hence $\Hmeshzeromi$ endowed with the $i$--th
component norm \eqref{norm-rv} is compactly embedded in $L^2(\Omega)$ in the
sense of Definition \ref{compact-embedded-sequence}. Finally  $(X_m)_\mnn $ is
compactly embedded in $(L^{2}(\Omega))^{d}$.\\

 Let us consider the sequence $(\bfu_m)_\mnn$ verifying \eqref{eq:scheme:ro} for
 $\deltat=\deltat_{m}$ and $\mathcal{D}=\mathcal{D}_m$. Thanks to the first bound
 from Lemma \ref{lem:est-vit-ro} we have that $(\|\bfu_m\|_{L^{1}(0,T;\Hmeshzerom)})_\mnn$ is bounded.
 Furthermore, using the discrete Poincaré inequality on the first bound yields
 that $(\bfu_m)_\mnn$ is bounded in $L^{2}(0,T;(L^{2}(\Omega)^{d}))$. Let us show
 that the remaining condition is satisfied: that is, demonstrating the existence
 of $\zeta$ with $\zeta(\tau)\rightarrow 0$ as $\tau \rightarrow 0$ for all $\mnn$
 such that,

 \begin{equation}\nonumber
    \int_0^{T-\tau} \| \bfu_m(t+\tau,\cdot) -\bfu_m(t,\cdot) \|^2_{L^2} \dt \le \zeta(\tau) \qquad \forall \tau \in (0,T) \quad \forall \mnn
 \end{equation}
First, we use the following upper bound for the expression in \eqref{translates}: $C\sqrt{\tau + \deltat} \leq C(\tau^{\frac{1}{2}} + \deltat^{\frac{1}{2}})$,
with $C$ independent of $m$ thanks to the assumption $\eta_{\mesh_m} \le
\eta $.\\Let $\varepsilon >0$.  Then we may find $\tau_1$ and $M \in \xN$ so
that we have $\tau_1^{\frac{1}{2}} \leq \varepsilon$ and $\deltat_m^{\frac{1}{2}} \leq\varepsilon$ for all $ m \geq M$. Thus, by Lemma \ref{lem:translates}, we have for all $m \geq M$,
\begin{equation}\nonumber
   \int_0^{T-\tau} \| \bfu_m(t+\tau,\cdot) -\bfu_m(t,\cdot) \|^2_{\bs L^2} \dt \le C_1\varepsilon \qquad \forall \tau \in (0,\tau_1)
\end{equation}
For $m < M$, we introduce ${\varphi}^{m} \in C([0,T]\times\Omega)$ verifying $\|\bfu_m - \bs\varphi^{m}\|_{L^2((0,T),(L^{2}(\Omega))^{d})} < \varepsilon$. Therefore
%
%
\begin{flalign}\nonumber
\begin{array}{l r l }
  & \ds \scaleobj{.9}{\int_0^{T-\tau}} \| \bfu_m(t+\tau,\cdot) -\bfu_m(t,\cdot) \|^2_{\bs L^2} \dt \le   3\Big[
    & \ds \scaleobj{.9}{\int_0^{T-\tau}}\| \bfu_m(t+\tau,\cdot) - \bs\varphi^{m}(t+\tau,\cdot) \|^2_{\bs L^2} \dt \vspace{2pt}  \\ \vspace{2pt}
    & + & \ds \scaleobj{.9}{\int_0^{T-\tau}} \| \bs\varphi^{m}(t+\tau,\cdot)  - \bs\varphi^{m}(t,\cdot)  \|^2_{\bs L^2} \dt \\
    & + & \ds \scaleobj{.9}{\int_0^{T-\tau}} \| \bs\varphi^{m}(t,\cdot)  - \bfu_m(t,\cdot)  \|^2_{\bs L^2} \dt \Big]
\end{array}
\end{flalign} which yields\\
\begin{equation}\nonumber
  \ds \int_0^{T-\tau} \| \bfu_m(t+\tau,\cdot) -\bfu_m(t,\cdot) \|^2_{L^2} \dt \le 6 \varepsilon
  + \ds  3 \int_0^{T-\tau} \| \bs\varphi^{m}(t+\tau,\cdot)  - \bs\varphi^{m}(t,\cdot)  \|^2_{L^2}
\end{equation}\\
with $ \int_0^{T-\tau} \| \bs\varphi^{m}(t+\tau,\cdot)  - \bs\varphi^{m}(t,\cdot)  \|^2_{L^2} \rightarrow 0$ as $\tau \rightarrow 0$.
Thus, we  may find $\tau^{m}$ such that $\int_0^{T-\tau} \| \bfu_m(t+\tau,\cdot) -\bfu_m(t,\cdot) \|^2_{L^2} <  9 \varepsilon$ for $\tau \in (0,\tau^m)$.\\\\
Let us define $\overline{\tau} = \min(\tau_1, \min_{m\leq M} \tau^m)$ and $C = \max(C_1, 9)$. It follows that $\forall \tau \in (0,\overline{\tau})$,
\begin{equation}\nonumber
   \ds \int_0^{T-\tau} \| \bfu_m(t+\tau,\cdot) -\bfu_m(t,\cdot) \|^2_{L^2} \dt \le C\varepsilon \quad \mbox{ for all  } \quad \mnn
\end{equation}\\
Hence, the third condition of Theorem \ref{theo-compact-space} is satisfied and
we can conclude with the existence of $\bar \bfu\in L^{2}(0,T;L^{2}(\Omega)^{d})$
such that, up to a subsequence,
\[
\bfu_m\rightarrow \bar \bfu \text{ in }  L^{2}\left(0,T;L^{2}(\Omega)^{d}\right)\mbox{ as } m \to +\infty.
\]
%
%
\paragraph{Step 3- Passing to the limit as \texorpdfstring{$\bs{\deltat,h_\mesh\rightarrow 0}$}{dt} in the mass balance equation\\ }
For this step, we prove a weak Lax-Wendroff consistency for the discrete mass
balance equation : by passing to the limit as  $\deltat,h_\mesh\rightarrow 0$ in
\eqref{mass:weak}, we aim to show that the limit $(\bar \rho,\bar \bfu)$
obtained in the previous steps satisfies the weak mass balance equation
\eqref{massweak}.\\\\
Let  $\psi \in C_c^\infty( [0,T) \times \Omega )$ so that
we take $\psi_{m}^{n}= \Pi_{\mesh_m}  \psi(t_n,\cdot) \in L_{\mesh}$ as a test
function in \eqref{mass:weak}. Multiplying by $\deltat_m$ and summing over
$n= \{0, \ldots, N_m-1\}$ (with $N_m \deltat_m = T$) yields:
\[
 \sum_{n=0}^{N_m-1} \deltat_m \int_{\Omega} \eth_{t} \rho_{\mesh_m}^{n+1} \ \psi_m^{n} \dx
 + \sum_{n=0}^{N_m-1} \deltat_m \int_{\Omega}\dive_{\mesh_m}(\rho \bfu)_m^{n+1}\ \psi_m^{n}\dx
 =T_{1,m}+T_{2,m} = 0
\]
by defining the first and second terms as $T_{1,m}$ and $T_{2,m}$ respectively.
Let us use the definition of each operator and drop the subscript $m$ to reduce
the clutter of notations. It comes down to studying the convergence of each of
the following terms:
\begin{flalign}\nonumber
\begin{array}{l l }
 T_{1,m}= & \ds \sum_{n=0}^{N-1}\sum_{K\in\mesh}\vert K\vert (\rho_{K}^{n+1}-\rho_{K}^{n})
\psi_{K}^{n} \vspace{10pt} \\
T_{2,m}= & \ds \sum_{n=0}^{N-1}\deltat\sum_{K\in\mesh}\psi_{K}^{n}\sum_{\edge\in\edges(K)}
F_{K,\edge}^{n+1}
\end{array}
\end{flalign}\\
We deal with $T_{1,m}$ by rearranging the sum so as to perform a discrete
integration by parts for the time variable:
\[
 T_{1,m}=-\sum_{n=0}^{N-1}\sum_{K\in\mesh}\vert K\vert \rho_{K}^{n+1}(\psi_{K}^{n+1}-
 \psi_{K}^{n})-\sum_{K\in\mesh}\vert K\vert \rho_{K}^{(0)} \psi_{K}^{(0)}
\]
since $\psi_{K}^{N}=0$ thanks to the compact support of $\psi$. Let us note
$\psi_m = \sum_{n=0}^{N-1} \psi^{n+1}_m \characteristic_{]t^n,t^{n+1}]}$. This
is equivalent to,
\[
 T_{1,m}=-\int_{0}^{T}\int_{\Omega} \rho_{m}(t,\bfx) \eth_{t} \psi_{m}(t,\bfx)\dx\dt-\int_{\Omega}
 \rho_{m}^{(0)}(\bfx) \psi_m(0,\bfx)\dx,
\]
The weak-$\star$ convergence of $(\rho_{m} )_{\mnn}$ in $L^\infty((0,T)\times\Omega)$
and the strong convergence of $(\eth_{t} \psi_{m})_\mnn$ towards $\partial_{t} \psi$
in $L^1((0,T)\times\Omega)$ yields the convergence of the first term in $T_{1,m}$.
The second term is dealt with in the same way : from the initialization of the
scheme and the assumed regularity of the initial data  $\rho_{ 0} \in L^{\infty} (\Omega)$
and the strong convergence of $(\psi_m(0,\cdot))_\mnn$ in $L^1(\Omega)$. Hence we
have,

\[
\lim_{m\rightarrow\infty} T_{1,m}=-\int_{0}^{T}\int_{\Omega} \bar \rho(t,\bfx) \partial_{t} \psi(t,\bfx)
\dx\dt-\int_{\Omega}
\rho_{0}(\bfx) \psi(0,\bfx)\dx.
\]
Let us now focus on $T_{2,m}$. Using the expression of the mass flux $F_{K,\edge}$
and the boundary conditions on the velocity, we reorder the sum in $T_{2,m}$ so as
to perform a discrete integration by parts for the space variable:
\[
T_{2,m}=-\sum_{i=1}^{d}\sum_{n=0}^{N-1}\deltat \sum_{\substack{\edge \in \edgesint^{(i)} \\ \edge=K|L}}\vert D_{\edge}\vert
\rho_{\edge}^{n+1} u_{K,\edge}^{n+1}
 \frac{\vert \edge\vert}{\vert D_{\edge}\vert} (\psi_{L}^{n}-\psi_{K}^{n}),
\]
The terms $\rho^{n+1}_\edge$ being computed using the upwind scheme, we decompose
$T_{2,m}$ into two sums : $T_{2,m}=\mathcal{T}_{2,m}+\mathcal{R}_{2,m}$ with

\begin{flalign}\nonumber
\begin{array}{r l }
  \mathcal{T}_{2,m} = & - \ds \sum_{i=1}^{d}\sum_{n=0}^{N-1}\deltat \sum_{\substack{\edge \in \edgesint^{(i)} \\ \edge=K|L}}
  \vert D_{\edge}\vert \rho^{n+1}_{D_\edge} u_\edge^{n+1}\frac{\vert \edge\vert}{\vert D_{\edge}\vert}
   (\psi_{L}^{n}-\psi_{K}^{n})(\bfn_{K,\edge}\cdot \bfe_i) \vspace{10pt} \\
\mathcal{R}_{2,m} = &- \ds \sum_{i=1}^{d}\sum_{n=0}^{N-1}\deltat \sum_{\substack{\edge \in \edgesint^{(i)} \\ \edge=K|L}}
\vert D_{\edge}\vert (\rho_{\edge}^{n+1}-  \rho^{n+1}_{D_\edge} ) u_{K,\edge}^{n+1}
\frac{\vert \edge\vert}{\vert D_{\edge}\vert} (\psi_{L}^{n}-\psi_{K}^{n})
\end{array}
\end{flalign}\\

by recalling that $\vert D_{\edge}\vert \rho^{n+1}_{D_\edge}= \vert D_{K,\edge}\vert \rho_{K}^{n+1}+
\vert D_{L,\edge}\vert \rho_{L}^{n+1}$. In this form however, the term
$\mathcal{T}_{2,m}$ reveals the convex combinations of the density
$\mathcal R_{\mesh}^{\edgesi}\rho_m$ for $i=1,..,d$ and we have,

\[
\mathcal{T}_{2,m} =  - \ds \sum_{i=1}^d\ds \int_{0}^{T}\int_{\Omega}
\mathcal R_{\mesh}^{\edgesi}(\rho_m(t,\bfx))u_{i,m}(t,\bfx)\eth_i \psi_m(t,\bfx)\dx\dt.
\]
Since the sequence $(\rho_m)_m$ is only $\star$--weakly convergent, we can
neither conclude about the convergence of $(R_{\mesh}^{\edgesi}\rho_m)_m$ yet
nor pass to the limit in $\mathcal{T}_{2,m}$. Hence, we must reformulate the sum
 so as to iterate over elements of $\mesh$. For every $i=1,..,d$, let us note
 $K = [\edge \edge']$ for $(\edge,\edge ')\in (\edges(K)\cap\edgesi)^2$, with
 $\edge,\edge'$ thus being the two faces of the cell $K$ normal to $\bfe_i$.
 Notice that we may include the elements of $\edgesexti$ owing to the homogeneous
 boundary conditions for the velocity. Therefore, we have,
\begin{flalign}\nonumber
\begin{array}{r l }
  T_{2,m} & =- \ds \sum_{i=1}^{d}\sum_{n=0}^{N-1}\deltat \sum_{\substack{\edge \in \edgesint^{(i)} \\ \edge=K|L}}
  \vert D_{\edge}\vert \rho_{D_\edge} u_\edge^{n+1}(\eth \psi^n)_\edge \vspace{5pt}\\
  & = - \ds \sum_{i=1}^{d}\sum_{n=0}^{N-1}\deltat \sum_{\substack{K \in \mesh \\ K = [\edge \edge']}}
  \Big(\vert D_{K,\edge}\vert \rho_{K} u_\edge^{n+1}(\eth \psi^n)_\edge
  + \vert D_{K,\edge'}\vert \rho_{K} u_{\edge'}^{n+1}(\eth \psi^n)_{\edge'}\Big)\vspace{5pt} \\
  & = - \ds \sum_{i=1}^{d}\sum_{n=0}^{N-1}\deltat \sum_{\substack{K \in \mesh \\ K = [\edge \edge']}}
  | K | \rho_{K} \Big( \frac{\vert D_{K,\edge}\vert}{|K|} u_\edge^{n+1}(\eth \psi^n)_\edge
  +  \frac{\vert D_{K,\edge'}\vert}{|K|} u_{\edge'}^{n+1}(\eth \psi^n)_{\edge'}\Big) \vspace{5pt} \\
  & = - \ds \sum_{i=1}^{d}\sum_{n=0}^{N-1}\deltat \sum_{\substack{K \in \mesh \\ K = [\edge \edge']}}
  | K | \rho_{K} \frac{1}{2}\Big(  (u_\edge^{n+1}(\eth \psi^n)_\edge + u_{\edge'}^{n+1}(\eth \psi^n)_{\edge'}\Big) \vspace{5pt} \\
  & = - \ds \sum_{i=1}^{d}\sum_{n=0}^{N-1}\deltat \sum_{\substack{K \in \mesh \\ K = [\edge \edge']}}
  | K | \rho_{K} \mathcal  R^{\mesh}_{\edgesi}(u^{n+1}_i \eth_i \psi^n)_\edge.
\end{array}
\end{flalign}\\
Finally we have,
\begin{flalign}\nonumber
\begin{array}{r l }
   T_{2,m} & = - \ds \sum_{i=1}^d\ds \int_{0}^{T}\int_{\Omega} \rho_m(t,\bfx)\mathcal R^{\mesh}_{\edgesi} (u_{i,m}(t,\bfx)\eth_i \psi_m(t,\bfx))\dx\dt
\end{array}
\end{flalign}
The strong convergence of $(u_{i,m})_\mnn$ and $(\eth_i \psi_m)_\mnn$ in $L^2((0,T)\times\Omega)$
yields the strong convergence of $(u_{i,m} \eth_i \psi_m)_\mnn$ in $L^2((0,T)\times\Omega)$.
By Lemma \ref{lem-conv-recons}, we have:\\
\[
R^{\mesh}_{\edgesi} \big(u_{i,m}\eth_i \psi_m\big) \ds \longrightarrow
\bar \bfu\eth_i \psi  \quad \mbox{in} \quad  L^2((0,T)\times\Omega) \quad \mbox{as} \quad m \rightarrow \infty
\]~\\
Hence, combined with the weak convergence of $\rho_m \rightharpoonup \bar \rho \mbox{ in } L^2((0,T)\times\Omega)$
we may pass to the limit in $\mathcal{T}_{2,m}$ to get,
\[
  \lim_{m\rightarrow \infty} \mathcal{T}_{2,m}=-\int_{0}^{T}\int_{\Omega}\bar \rho(t,\bfx)\bar \bfu(t,\bfx)\cdot \nabla
  \psi(t,\bfx)\dx\dt
\]
\bop
We are left to show that the remainder $\mathcal{R}_{2,m}$ tends to $0$ as
$m\rightarrow\infty$. Owing to the upwind choice \eqref{eq:divflux} for $\rho_\edge$
we have: \\
\begin{equation}\nonumber
  \vert D_{\edge}\vert (\rho_{\edge}^{n+1}-  \rho^{n+1}_{D_\edge} ) \leq  \max{(\vert D_{K\edge}\vert, \vert D_{L,\edge}\vert)} | \rho_K^{n+1}-  \rho_L^{n+1} |
\end{equation}\\
Therefore, $\mathcal{R}_{2,m}$ may be bounded with,
\[
\vert\mathcal{R}_{2,m} \vert \le
  \sum_{n=0}^{N-1}\deltat \sum_{i=1}^{d}\sum_{\substack{\edge \in \edgesint^{(i)} \\ \edge=K|L}} |\edge|
  \vert\rho_{L}^{n+1}-\rho_{K}^{n+1}\vert \vert u_{\edge}^{n+1} \vert  \vert\psi_{L}^{n}-\psi_{K}^{n}\vert
  \]
We then use the Cauchy-Schwarz inequality to highlight the left handside of the
weak $BV$ estimate \eqref{eq:bvweak},
\begin{flalign}\nonumber
\begin{array}{l l }
  \vert \mathcal{R}_{2,m} \vert
    & \leq  \ \bigl(\ds \sum_{n=0}^{N-1}\deltat \ds \sum_{i=1}^{d}\sum_{\substack{\edge \in \edgesint^{(i)} \\ \edge=K|L}}
  |\edge|\vert\rho_{L}^{n+1}-\rho_{K}^{n+1}\vert^2 \vert u_{\edge}^{n+1} \vert \bigr)^{\frac 1 2}
  \bigl(\ds \sum_{n=0}^{N-1}\deltat \ds \sum_{i=1}^{d}\sum_{\substack{\edge \in \edgesint^{(i)} \\ \edge=K|L}} |\edge|
    \vert\psi_{L}^{n}-\psi_{K}^{n}\vert^2 \vert u_{\edge}^{n+1} \vert\bigr)^{\frac 1 2} \\
    & \leq  \sqrt{C} \bigl(\ds \sum_{n=0}^{N-1}\deltat \ds \sum_{i=1}^{d}\sum_{\substack{\edge \in \edgesint^{(i)} \\ \edge=K|L}}
    |\edge| \vert\psi_{L}^{n}-\psi_{K}^{n}\vert^2 \vert u_{\edge}^{n+1} \vert\bigr)^{\frac 1 2}
\end{array}
\end{flalign}\\
By Lemma \ref{lem-consistance}, there exists $C_\psi$ depending only $\partial_i \psi$,
$\Omega$ and $T$ such that $\vert\psi_{L}^{n}-\psi_{K}^{n}\vert \leq C_\psi d(\bfx_K,\bfx_L)$.
As $|D_\edge| = |[\bfx_K,\bfx_L] \times \edge|$ it yields that,
\begin{align*}
  \vert\mathcal{R}_{2,m} \vert
    &\le  \sqrt C C_\psi  \sqrt{h_m} \bigl( \sum_{n=0}^{N-1}\deltat\sum_{i=1}^{d}\sum_{\substack{\edge \in
    \edgesint^{(i)} \\ \edge=K|L}} |D_\edge|\vert u_{\edge}^{n+1} \vert\bigr)^{\frac 1 2} \\
    &\le C(\psi,T,\Omega,\Vert \bfu_m \Vert_{L^2(L^2)}) \sqrt{h_m}.
\end{align*}
Thanks to the bound \eqref{estiLdeux:ro} on $\bfu_m$ in $L^2( (0,T) \times \Omega )$,
we have $\vert\mathcal{R}_{2,m} \vert \longrightarrow 0$ as $h_m \to 0$ which
allows us to conclude that $(\bar \rho, \bar \bfu)$ satisfies the weak mass
balance equation.

\paragraph{Step 4- Regularity of the limit $ \boldsymbol{\bar{u}}$ }
In order to demonstrate the strong convergence of $(\rho_m)_m$ in the next section,
we require $\boldsymbol{\bar{u}}$ to be divergence free. First, let us show that
$\boldsymbol{\bar{u}}\in L^2((0,T),\bfH^1_0(\Omega))$. A characterization of
$\boldsymbol{\bar{u}}\in L^2((0,T),\bfH^1(\Omega))$ is given by the existence of
a constant $C$ such that,
\begin{equation}\label{eq:h10}
  \left| \int_0^T\int_\Omega \bar u_i \partial_j \varphi \dx \right| \leq
  C \|\varphi\|_{L^2((0,T)\times\Omega)} \quad \forall \varphi \in C^\infty_c((0,T)\times\Omega)\quad \forall i,j=1,..,d
\end{equation}

We adapt the proof of Theorem \ref{compac:space} from \cite{FV-book} to our case.
Let $\varphi \in C^\infty_c((0,T)\times\Omega)$ and $(i,j)\in \llbracket1,d\rrbracket$.
From now on we drop the subscript $i$ in the velocity notation. Let us define the
continuation of $(u_m)_\mnn$ on $(\R \times \R^d)\backslash((0,T)\times\Omega)$,
denoted $(\mathbb P u_m)_\mnn$, by
\[
\mathbb P {u}_m= u_m \quad \mbox{a.e. on} \quad (0,T)\times\Omega \quad \mbox{and}
\quad \mathbb P {u}_m= 0 \quad \mbox{a.e. on} \quad (\R \times \R^d)\backslash((0,T)\times\Omega).
\]
Then $(\mathbb P {u}_m)_\mnn$ converges strongly to $\mathbb P \bar u$ in $L^2(\R\times\R^d)$ with $\mathbb P \bar u$ being the continuation of $\bar u$ on $(\R \times \R^d)\backslash((0,T)\times\Omega)$.\\\\
Let $\eta \in \R^d, \eta \neq 0$ and $\mnn$. Using Cauchy-Schwarz, we have:
\[
\int_\R\int_{\R^d}  \frac{\mathbb P u_m(t,\bfx+\eta) - \mathbb P u_m(t,\bfx)}{|\eta|}\varphi \dx \dt
\leq \frac{1}{|\eta|} \left[ \int_\R \|\mathbb P u_m(t,\cdot+\eta) -\mathbb P u_m(t,\cdot) \|_{L^2}^2 \dt \right]^{\frac{1}{2}} \|\varphi\|_{L^2}
\]~\\
Again, using the space compactness estimate from Lemma \ref{compac:astuce} and
the bound on the $\Hmeshizero$ norms of $(u_m)_\mnn$ \eqref{estiLdeux:ro} yields
the existence of $C$ depending only on initial data $\rho_0, \bfu_0$ and $\bff$
such that,
\[
\int_\R\int_{\R^d}  \frac{\mathbb P u_m(t,\bfx+\eta) - \mathbb P u_m(t,\bfx)}{|\eta|} \varphi \dx \dt
\leq \frac{1}{|\eta|} \left( C |\eta|(|\eta| + C h_{\mesh_m}) \right)^{\frac{1}{2}} \|\varphi\|_{L^2(L^2)}
\]
Passing to the limit as $m \rightarrow \infty$ yields
\[
\int_\R\int_{\R^d}  \frac{\mathbb P \bar u(t,\bfx+\eta) -\mathbb P \bar u(t,\bfx)}{|\eta|} \varphi(t,\bfx) \dx \dt
\leq \sqrt{C} \|\varphi\|_{L^2(\R\times\R^d)}
\]
We perform a change of variable to obtain,
\[
\int_\R\int_{\R^d}  \frac{\varphi(t,\bfx) -\varphi(t,\bfx-\eta)}{|\eta|} \mathbb P \bar u(t,\bfx) \dx \dt
\leq \sqrt{C} \|\varphi\|_{L^2(\R\times\R^d)}
\]
We may retrieve an estimate with $\partial_j \varphi$ by setting $\eta = h \bfe_j$
with $h>0$. Using a Taylor's inequality and passing to the limit as $h \rightarrow 0$
yields \eqref{eq:h10} on $\R \times \R^d$. Hence $P \bar u \in L^2(\R, H^1(\R^d))$.
Finally, since $\bar u = P \bar u$ a.e. on $(0,T)\times\Omega$ and $P \bar u = 0$
outside $(0,T)\times\Omega$, we conclude that $\bar u \in L^2((0,T), H_0^1(\Omega))$.\\\\
Let us show that $\bar \bfu$ is divergence free. For any $\varphi \in C^\infty_c((0,T)\times\Omega)$
we define the sequence $(\varphi_m)_\mnn$ as,
\[
\varphi_{m}  = \sum_{n=0}^{N-1} \varphi^{n+1}_{m} \characteristic_{(t^n,t^{n+1}]} \quad
\mbox{with} \quad \varphi^n_{_m} = \Pi_{\mesh_m} \varphi(t^n,\cdot) \quad \forall n\in\{1,N\}
\]
The sequence $(u_m)_\mnn$ being discrete divergence free a.e. in $(0,T)\times\Omega$
 combined with the discrete integration by parts \ref{eq:duality} yields,
\begin{flalign}\nonumber
\begin{array}{r l }
  0 = &  \ds \int_0^T \int_\Omega  \varphi_m  \dive_{\mesh_m} \bfu_m \dx \dt \vspace{10pt}\\
  = &  - \ds \int_0^T \int_\Omega  \nabla_{\edges_m}\varphi_m  \cdot \bfu_m \dx \dt
\end{array}
\end{flalign}\\
Using the strong convergences of $\nabla_{\edges_m}\varphi_m$ (by Lemma \ref{lem-consistance})
and $(\bfu_m)_m$ towards $\nabla \varphi$ and $\bar \bfu$ respectively in $L^2((0,T),L^2(\Omega)^d)$,
we conclude with:
\begin{flalign}\nonumber
\begin{array}{r l }
  0 =& \ds \lim_{m\rightarrow \infty} - \ds \int_0^T \int_\Omega  \nabla_{\edges_m}\varphi_m  \cdot \bfu_m \dx \dt
  =  - \ds \int_0^T \int_\Omega  \nabla\varphi \cdot \bar \bfu \dx \dt = \ds \int_0^T \int_\Omega  \varphi  \dive \bar \bfu \dx \dt
\end{array}
\end{flalign}\\
Finally, the limit $\bar \bfu$ belongs to $L^2((0,T),\bfH^1_0(\Omega))$ and
$\dive \bar \bfu = 0$ a.e. on $(0,T)\times\Omega$.

\paragraph{Step 5- Weak convergence of the discrete derivatives in $\boldsymbol{L^{2}(\Omega)}$}
First, owing to the discrete "$L^2(\bfH^1_0)$" bound $\eqref{estiLdeux:ro}$ we
have that for all $i,j = 1,..,d$ the sequence of derivatives $(\eth_j u_{i,m})_\mnn$
is bounded in ${L^{2}((0,T) \times \Omega)}$, thus is weakly convergent to some
limit in $L^{2}((0,T) \times \Omega)$.
For $i\in\{1,..,d\}$ and $n\in \{1,..,N\}$ -- although we drop the superscript
$n$ -- we show that $\eth_i u_{i,m} \rightharpoonup \partial_i \bar u_{i}$ as
$m\rightarrow \infty$. Let $\varphi \in C^\infty_c(\Omega)$ so that we may take
$\varphi_m = \Pi_\mesh \varphi$. Then,
\[
\int_\Omega  \eth_i u_{i,m} \varphi \dx = \int_\Omega  \eth_i u_{i,m} \varphi_m \dx
+ \int_\Omega  \eth_i u_{i,m} (\varphi - \varphi_m) \dx
\]
Thanks to the regularity of $\varphi$ and the bound on $\eth_i u_{i,m}$, the
second term vanishes as $m\rightarrow \infty$ and we may work on the first term
only. An integration by part $\eqref{eq:duality}$ yields,
\[
\int_\Omega  \eth_i u_{i,m} \varphi_m \dx  = - \int_\Omega u_{i,m} \eth_i \varphi_m \dx
\]
Using the strong convergence of $u_{i,m}$ with the consistency of the derivatives
(Lemma \ref{lem-consistance}) we pass to the limit to get,
\[
\lim_{m\rightarrow \infty}\int_\Omega  \eth_i u_{i,m} \varphi \dx  =
- \int_\Omega \bar u_{i} \partial_i \varphi \dx
=  \int_\Omega \partial_i \bar u_{i}  \varphi \dx
\]
and we can conclude by density. The result with $\eth_j u_{i,m}$ follows, by
considering the reconstructions $\mathcal  R^{\mesh}_{\edgesi}u_{i,m}$ and
$\mathcal  R_{\mesh}^{\edgesj}\varphi_m$ and their convergence properties.

\paragraph{Step 6- Strong convergence of \texorpdfstring{$\bs{(\rho_m)_\mnn}$ in $\bs{L^2( (0,T) \times \Omega )}$}{dt} }
We follow the proof in \cite{latche-saleh-17} to obtain the strong convergence
of $(\rho_m)_m$ in ${L^2( (0,T) \times \Omega )}$. Obviously, we cannot resort
to the previous space compactness results from the lack of estimates on the space
translates of the discrete density. Instead, using results from \cite{DiPL} and
\cite{BoyerFabrie-book} and recalled in Annex \ref{annex:transport}, we turn to
the regularity of $\bar \rho$ to conclude with the result. \\\\ The boundness of
 $(\rho_{m} )_{\mnn}$ in $L^{\infty}((0,T) \times \Omega)$ yields that $(\rho_{m} )_{\mnn}$
 converges weakly towards $\bar \rho$ in $L^2( (0,T) \times \Omega )$. Hence,
 thanks to the convexity and the continuity of the $L^2$ norm for the weak topology,
 we use the following classical result \cite[Corollary II.2.8]{BoyerFabrie-book}:
\[
 \|\bar\rho\|_{L^{2} ( (0,T) \times \Omega )}\leq\ds \liminf_{ m \to \infty} \|\rho_{m}\|_{ L^{ 2} ( (0,T) \times \Omega )}.
\]
Let us show that $\limsup_{ m \to \infty} \|\rho_{m}\|_{L^2} \leq \|\bar\rho\|_{L^{2}}$
which in turn, will yield the strong convergence of $(\rho_{m} )_{\mnn}$ in $L^2( (0,T) \times \Omega )$.
\cite[Proposition II.2.11]{BoyerFabrie-book}.\\\\
By \eqref{eq:rhobv}, we have for all $n$ in $\{0,\cdots, N-1 \}$:
\begin{equation*}
 \sum_{K\in\mesh_m}\vert K\vert (\rho_{K}^{n+1}  )^{2}\leq\sum_{K\in\mesh_m}\vert K\vert (\rho_{K}^{(0)})^{2}\leq
 \|\rho_{0}\|_{L^{2}(\Omega)}^{2},
\end{equation*}
which yields $\|\rho_m(., t)\|_{L^{2}(\Omega)}^{2}\le \|\rho_0(., t)\|_{L^{2}(\Omega)}^{2}$
for all $t\in(0,T)$ and $\mnn$.\\\\
Moreover, since $\bar \rho$ is a solution to the weak transport equation \eqref{massweak},
with $\bar\bfu \in L^2([0,T],\bfH^1_0(\Omega))$ being divergence free, we apply
Theorem \ref{unicite-transport} to conclude with the unicity and time continuity
of the limit $\bar \rho$ and we have $\bar \rho \in C^{0}(0,T;L^{2}(\Omega))$.\\
Taking $\beta(\zeta) = \zeta^2$ in Theorem \ref{renormalisation} along with the
test function $\varphi = 1$, we can establish that for any
$t \in [0,T)$ then $\|\bar\rho(.,t)\|_{L^{2}(\Omega)}= \|\rho_{0}\|_{L^{2}(\Omega)}$.\\\\
Therefore, we have
$\|\rho_{m}(.,t)\|_{L^{2}(\Omega)}^{2}\leq\|\bar \rho(.,t)\|_{L^{2}(\Omega)}^{2}$
for all $t\in[0,T)$ and $\mnn$. Integrating this last inequality for $t \in [0, T )$,
we obtain $\|\rho_{m}\|_{L^{2}( (0,T) \times \Omega )}^{2}\leq\|\bar\rho\|_{L^{2}( (0,T) \times \Omega )}^{2}$
for all $\mnn$, and passing to the limit as $m$ goes to infinity yields:
\begin{equation*}
 \limsup_{n\rightarrow\infty} \|\rho_{m}\|_{L^{2}( (0,T) \times \Omega )}\leq\|\bar\rho\|_{L^{2}( (0,T) \times \Omega )}.
\end{equation*}
This proves that $\lim_{m\rightarrow\infty}\|\rho_{m}\|_{L^{2}( (0,T) \times \Omega )}=
\|\bar\rho\|_{L^{2}( (0,T) \times \Omega )}$. Combined with the weak convergence
of $(\rho_{m})_\mnn$ towards $\bar \rho$ in $L^{2}((0,T) \times \Omega)$ we conclude
that,
\[\rho_{m} \longrightarrow \bar \rho \quad \mbox{as}\quad  m \to \infty \quad \mbox{in}\quad L^{2}((0,T) \times \Omega)\]

\paragraph{Step 7- Strong convergence of the viscosity tensor in $\boldsymbol{L^{2}((0,T) \times \Omega)}$}
Thanks to the continuity of $\mu : \R \rightarrow \R_+ $, the strong convergence
 of $(\mu_m)_\mnn$ is rather straightforward. Indeed, owing to the strong convergence
 of $(\rho_{m})_\mnn$ to $\bar \rho$  in $L^{2}((0,T) \times \Omega)$ we may use
 the reciprocal of the dominated convergence theorem. Hence there exists a subsequence
 still denoted  $(\rho_{m})_\mnn$ such that,
\[\rho_{m}(t,\bfx) \rightarrow \bar \rho(t,\bfx) \quad \mbox{a.e. in} \quad [0,T]\times\Omega\]
Therefore, $\mu_{m}(t,\bfx) = \mu(\rho_{m}(t,\bfx))$ converges pointwise to
$\bar \mu(t,\bfx) = \mu(\bar \rho(t,\bfx)) $ a.e. in  $[0,T]\times\Omega$. Using \eqref{mumin} --
the $L^\infty([0,T]\times\Omega)$ bound on $(\mu_m)_\mnn$ -- and the dominated convergence theorem yields the strong convergence
of $(\mu_m)_\mnn$ to $\bar \mu$ in $L^{2}((0,T) \times \Omega)$.\\\\
For any $(i,j) \in \llbracket1,d\rrbracket^2$ we show that the sequence of viscosities
associated to $\edgesd_m^{(i,j)}$ converges strongly to $\bar \mu$ in $L^{2}((0,T)  \times \Omega)$ as $m\rightarrow\infty$.
Nothing is to be done for $i=j$ as we have $\{D_\edged, \edged \in \edgesd_m^{(i,i)}\} = \{K \in \mesh_m\}$.
Let $i \neq j$ and $ n \in \{0,..,N\}$. Then $\mu^{n}_{{ij}_m}$
can be defined as a reconstruction of $\mu^n_m$. Using the same arguments as the
proof in Lemma \ref{lem-stab-recons} yields that $\|\mu_{{ij}_m}\|_{L^p((0,T)\times\Omega)} \leq \|\mu_m\|_{L^p((0,T)\times\Omega)}$.
Hence, by Lemma \ref{lem-conv-recons} we are able to conclude that
\[
\mu_{{ij}_m} \rightarrow \bar \mu \quad \mbox{as} \quad m
\rightarrow \infty \quad \mbox{in} \quad L^{2}((0,T) \times \Omega)   \quad \forall i,j = 1,..,d\]

\paragraph{Step 8- Passing to the limit as \texorpdfstring{$\bs{\deltat,h_\mesh\rightarrow 0}$}{dt} in the momentum balance equation }

Similarly to Step 3, we prove a weak Lax-Wendroff consistency for the
discrete momentum balance equation : by passing to the limit as
$\deltat,h_\mesh\rightarrow 0$ in \eqref{qdm:weak}, we aim to show
that the limit $(\bar \rho,\bar \bfu)$ obtained in the previous steps
satisfies the weak momentum equation \eqref{qdmweak}.\\\\
Let $\bfvarphi \in C_c^\infty( [0,T) \times \Omega )^{d}$, such that
$\dive \bfvarphi =0$, so that we take $ \bfvarphi_{m}^{n}=\widetilde{\Pi}_{\edges_m} \bfvarphi(\cdot,t_n)$.
The preservation of the divergence by the Fortin operator \eqref{fortin-prop3} yields
that $\dive_\mesh \bfvarphi_{m}^{n} = 0$.
Therefore $\bfvarphi_{m}^{n}$ belongs to $\boldsymbol{E}_{\edges}$ and may be
used as a test function in \eqref{qdm:weak}.\\\\
Multiplying by $\deltat_m$ and summing over $n= \{0, \ldots, N_m-1\}$ we get,

\begin{multline*}
\sum_{n=0}^{N_m-1} \deltat_m \left[ \int_{\Omega} \eth_{t} (\rho\bfu)^{n+1}_{m} \cdot \bfvarphi_m^{n} \dx
+ b_\edges((\rho\bfu)_{m}^{n+1},\bfu_{m}^{n+1},
\bfvarphi_m^{n}) \right.
\\ \left. +\int_\Omega \mu^{n+1}_\edgesd \bs{D_\edgesd}(\bfu_{m}^{n+1}) : \bs{D_\edgesd}(\bfvarphi_m^{n}) \dx -
\int_\Omega \bff_{\edges_m}^{n+1} \cdot \bfvarphi_m^{n} \dx\right] =0,
\end{multline*}~\\
which we decompose according to,
\begin{align*}
 &
 T_{1,i}^{(m)} =\sum_{n=0}^{N-1}\sum_{\edge\in\edgesi} \int_{\Omega} \eth_{t} (\rho u_i)^{n+1}_{m}
 \varphi_{m,i}^{n}  \dx, &
 &
\quad T_{2,i}^{(m)} = \sum_{n=0}^{N-1}\deltat \ b_\edges^{(i)}((\rho\bfu)_{m}^{n+1},u_{i,m}^{n+1},  \varphi_{i,m}^{n}),\\
 &
 T_{3}^{(m)}=\sum_{n=0}^{N-1}\deltat \int_\Omega \mu^{n+1}_\edgesd \bs{D_\edgesd}(\bfu_{m}^{n+1}) : \bs{D_\edgesd}(\bfvarphi_m^{n}) \dx,
  & &
\quad T_{4,i}^{(m)}=\sum_{n=0}^{N-1} \deltat\int_{\Omega} f_{i,m}^{n+1}\varphi_{i,m}^{n}\dx.
\end{align*}
to give $\sum_{i=1}^{d}\bigl[ T_{1,i}^{(m)}+T_{2,i}^{(m)}+T_{4,i}^{(m)}\bigr] + T_{3}^{(m)} =0$. \\\\
Let us treat the convergence of each term individually. We deal with $T_{1,i}^{(m)}$
by rearranging the sum so as to perform a discrete integration by parts for the
time variable
\begin{align*}
  T_{1,i}^{(m)} &=\sum_{n=0}^{N-1}\deltat \sum_{\edge\in\edgesi} \vert D_{\edge} \vert \;
  \frac{\rho_{D_{\edge}}^{n+1} u^{n+1}_{\edge}-\rho_{D_{\edge}}^{n} u^{n}_{\edge}} {\deltat}
  \; \varphi_{\edge}^{n}\\
  & =-\sum_{n=0}^{N-1} \deltat \sum_{\edge\in\edgesi} \vert D_{\edge} \vert \rho^{n+1}_{D_{\edge}}
  u_{\edge}^{n+1} \frac{(\varphi_{\edge}^{n+1} -\varphi_{\edge}^{n})}{\deltat}
  -\sum_{\edge\in\edgesi} \vert D_{\edge} \vert \rho^{(0)}_{D_{\edge}}  u_{\edge}^{(0)} \ \varphi_{\edge}^{(0)}.
\end{align*}
since $\varphi_\edge^{(T)}=0$ thanks to the compact support of $\varphi$.
We recall that
$\vert D_{\edge}\vert \rho^{n+1}_{D_\edge}= \vert D_{K,\edge}\vert \rho_{K}^{n+1}+ \vert D_{L,\edge}\vert \rho_{L}^{n+1}$:
as in Step 3 we are dealing with a reconstruction of $\rho_m$ and we have,
\begin{align*}
 T_{1,i}^{(m)}&  =-\int_{0}^{T} \int_{\Omega} \mathcal R_{\mesh}^{\edgesi}(\rho_m(t,\bfx)) u_{i,m}(t,\bfx) \eth_{t}
 \varphi_{i,m}(t,\bfx) \dx \dt \\
 & \quad  \quad  \quad  \quad  \quad  \quad  \quad -\int_{\Omega}\mathcal R_{\mesh}^{\edgesi}(\rho^{(0)}_{m}(\bfx)) u_{i,m}^{(0)}(\bfx)
\varphi_{i,m}^{(0)}(\bfx)\dx.
\end{align*}

Therefore passing to the limit in $T_{1,i}^{(m)}$ is straightforward :  Lemma \ref{lem-conv-recons} yields
that $(\mathcal R_{\mesh}^{\edgesi}\rho_m)_m$ converges strongly to $\bar \rho$ in $L^{2} ( (0,T) \times \Omega )$
as $m \to +\infty$ . Combined with the strong convergences of  $(u_{i,m})_\mnn$
and $\eth_t \varphi_{i,m}$ (by Lemma \ref{lem-consistance}) towards $\bar u_{i}$
 and $\partial_{t} \varphi_{i}$ respectively in $L^{2} ( (0,T) \times \Omega )$
 we can conclude for the limit in the first term.\\\\ From the initialization
 of the scheme and the assumed regularity of the initial data -- i.e. $\rho_{ 0} \in L^{\infty} (\Omega)$
 and $u_{0}\in L^{2} (\Omega)$ -- we get that $\mathcal R_{\mesh}^{\edgesi}(\rho_{m}^{(0)})$ and $u_{i,0}$ converge
to $\rho_{0}$  and  $u_{i,0}$ respectively in $L^{2} (\Omega)$. Finally,
Lemma \ref{lem-consistance} yields the convergence of $\varphi_{i,m}^{(0)}$ to
$\varphi_i(.,0)$ in $L^{2} (\Omega)$ and we have,

\[
\lim_{m\rightarrow\infty}\sum_{i=1}^{d} T_{1,i}^{(m)} =  - \int_{0}^{T}\int_{\Omega} \bar\rho(t,\bfx) \bar \bfu(t,\bfx) \
\cdot \partial_{t}  \bfvarphi(t,\bfx)\dx\dt -\int_{\Omega}\rho_{0}(\bfx) \bfu_{0}(\bfx) \cdot \bfvarphi(0,\bfx) \dx
\]

\noindent
Thanks to the reformulation of $b^{(i)}_\edges$ to \eqref{eq:new-b} we have for
the convection term $T_{2,i}^{(m)}$ ,

\begin{align*}
T_{2,i}^{(m)} & = \sum_{n=0}^{N-1}\deltat \ b_\edges^{(i)}((\rho\bfu)_{m}^{n+1},u_{i,m}^{n+1},  \varphi_{i,m}^{n}) \vspace{5pt}\\
& = - \ds \sum_{n=0}^{N-1}\deltat \ds \sum_{j=1}^d
\int_\Omega (\mathcal R_\edgesd^{(j,i)})^\rho \rho^{n+1}_{\edges^{(j)},m} (\mathcal R_\edgesd^{(j,i)})^u u_{j,m}^{n+1}
(\mathcal R_\edgesd^{(i,j)})^v u_{i,m}^{n+1} \eth_j \varphi_{i,m}^{n} \dx
 \end{align*}

Moreover, we recall that $\rho^{n+1}_{\edges^{(j)},m}$ is defined
using an upwind scheme (see Lemma \ref{lem-new-b} and \eqref{eq:divflux}).
Therefore it can be seen as a reconstruction of
$\rho^{n+1}_{m}$ in the sense of definition \ref{def:R-mesh-to-edgesi}
where we take $\alpha_\edge = 0$ or $1$ according to the upwind scheme
and a bound similar to that of Lemma \ref{lem-stab-reconsij}
holds. Then, given the strong convergence of $(\rho_m)_\mnn$ in
$L^2((0,T)\times \Omega)$ we have that $((\mathcal
R_\edgesd^{(j,i)})^\rho \rho_{\edges^{(j)},m})_\mnn \rightarrow \bar
\rho $ in $L^2((0,T)\times \Omega)$ as $m\rightarrow \infty$. Finally,
we decompose $T_{2,i}^{(m)}$ as,
\begin{align*}
  T_{2,i}^{(m)} & =  - \ds \sum_{n=0}^{N-1}\deltat \ds \sum_{j=1}^d
  \int_\Omega (\mathcal R_\edgesd^{(j,i)})^\rho \rho^{n+1}_{\edges^{(j)},m}
  (\mathcal R_\edgesd^{(j,i)})^u u_{j,m}^{n+1} (\mathcal R_\edgesd^{(i,j)})^v u_{i,m}^{n+1} \eth_j \varphi_{i,m}^{n+1} \dx \\
  & \quad \quad  \quad +  \ds \sum_{n=0}^{N-1}\deltat \ds \sum_{j=1}^d
  \int_\Omega (\mathcal R_\edgesd^{(j,i)})^\rho \rho^{n+1}_{\edges^{(j)},m}
  (\mathcal R_\edgesd^{(j,i)})^u u_{j,m}^{n+1} (\mathcal R_\edgesd^{(i,j)})^v u_{i,m}^{n+1} \eth_j (\varphi_{i,m}^{n+1} - \varphi_{i,m}^{n+1}) \dx
 \end{align*}

using the estimate on the trilinear form \eqref{estimb-rho} combined
with the bounds on the discrete solutions \eqref{romin},
\eqref{estiLdeux:ro} and the regularity of $\bfvarphi$, we conclude
that the second term tends to $0$ as $m\rightarrow \infty$.  From the
consistence properties of the reconstruction operators and the strong
convergences of $(\rho_m)_m$ and $(\bfu_m)_m$ in $L^2((0,T)\times \Omega)$.
Thus, up to a subsequence,
$(\mathcal R_\edgesd^{(j,i)})^\rho \rho^{n+1}_{\edges^{(j)},m} (\mathcal R_\edgesd^{(j,i)})^u u_{j,m}^{n+1} (\mathcal R_\edgesd^{(i,j)})^v
u_{i,m}^{n+1} \eth_j \varphi_{i,m}^{n+1}$ tends to $\bar \rho \bar u_j \bar u_i \partial_j \varphi_i $ in $L^1((0,T)\times \Omega)$ as $m\to
+\infty$, and we get,
\[
\lim_{m\to +\infty} T_{2,i}^{(m)} =  \int_0^T \int_\Omega \bar \rho \bar u_j  \bar u_i \partial_j \varphi_i \dx \dt,
\]

\smallskip

\noindent Thanks to the symmetry property \eqref{eq:symm-grad}, the diffusion term $T_{3}^{(m)}$ reads,
\[
T_{3}^{(m)}= \frac{1}{2}\sum_{n=0}^{N-1}\deltat \sum_{i,j=1}^d
\left[\quad  \int_\Omega \mu_{ij,m}^{n+1} \eth_j u^{n+1}_{i,m} \eth_j \varphi^n_{i,m}\dx
  + \int_\Omega \mu_{ij,m}^{n+1} \eth_j u^{n+1}_{i,m} \eth_i \varphi^n_{j,m}\dx \right]
\]
Let us pass to the limit in each term. First, we note that:
\[
\int_\Omega \mu_{ij,m}^{n+1} \eth_j u^{n+1}_{i,m} \eth_j \varphi^n_{i,m}\dx =
\int_\Omega \mu_{ij,m}^{n+1} \eth_j u^{n+1}_{i,m} \eth_j \varphi^{n+1}_{i,m}\dx
+ \int_\Omega \mu_{ij,m}^{n+1} \eth_j u^{n+1}_{i,m} \eth_j (\varphi^{n}_{i,m}-\varphi^{n+1}_{i,m})\dx
\]

Thanks to the regularity of $\varphi$ and the bounds on
$\mu_{ij,m}^{n+1}$ and $\eth_j u^{n+1}_{i,m}$ the second term vanishes
as $m\rightarrow \infty$ . The desired convergence result follows
using the weak convergence of the discrete derivatives (Step 5)
combined with the strong convergences as $m \rightarrow \infty$ of
$\mu_{ij,m}$, $\eth_j \varphi_{i,m}$ and $\eth_i\varphi_{j,m}$ to
$\bar \mu$, $\partial_j \varphi_{i}$ and $\partial_i \varphi_{j}$
respectively in $L^2((0,T)\times\Omega)$ (up to a subsequence, by the
dominated convergence theorem). Therefore,
\[
\lim_{m\rightarrow \infty}T_{3}^{(m)} = \int_{0}^{T} \int_\Omega \bar \mu \bs{D}(\bar \bfu) : \bs{D}(\bs{\varphi})\dx \dt
\]
\noindent
Finally, we have for the last term $T_{4,i}^{(m)}$:
\[
T_{4,i}^{(m)}=\sum_{n=0}^{N-1} \deltat\int_{\Omega} f_{i,m}^{n+1}\varphi_{i,m}^{n+1}\dx
+ \sum_{n=0}^{N-1} \deltat\int_{\Omega} f_{i,m}^{n+1}(\varphi_{i,m}^{n}-\varphi_{i,m}^{n+1})\dx
\]
Thanks to the continuity of the Fortin operator \eqref{fortin-prop1} we get the
strong convergence of $(f_{i,m})_m$ towards $f_{i}$ in $L^p((0,T)\times\Omega)$
for $p\in[1,\infty)$. Moreover, by a Taylor inequality (Lemma \ref{lem-consistance})
the second  term tends to $0$ as $m \rightarrow \infty$. Then,
\[
\lim_{m\rightarrow\infty} \sum_{i=1}^d T_{4,i}^{(m)} = \int_{0}^{T} \int_\Omega   \bff   \cdot \bfvarphi  \dx \dt
\]
which concludes the convergence proof.

\bibliographystyle{plain}
\bibliography{conv2}

\clearpage

\appendix
\thispagestyle{empty}

\section{Transport equation theory}\label{annex:transport}
Let $\rho_0 \in L^\infty(\Omega)$ and $\bfu \in L^2((0,T),\bfH^1_0(\Omega))$ with $\dive \bfu = 0$.
We assume that $\rho \in L^\infty(]0,T[\times \Omega)$ is a solution to the following problem,
\begin{equation}\label{solu-faible-transport}
  \int_0^T\int_\Omega \rho \left(\partial_t \varphi + \bfu\cdot \nabla \varphi \right)\dx\dt
  + \int_\Omega \rho_0\varphi(0,.) = 0  \qquad \forall \varphi \in C^1_c(]0,T[\times \Omega)
\end{equation}

\begin{theorem}[Renormalisation property \cite{BoyerFabrie-book}]\label{renormalisation}
  Let $\Omega$ be an open bounded domain $\subset \R^d$ with Lipschitz border,
  $\rho_0 \in L^\infty(\Omega)$ and $\bfu \in L^2((0,T),\bfH^1_0(\Omega))$ with $\dive \bfu = 0$.
  For any $\beta \in C^1(\R)$ and any solution of \eqref{solu-faible-transport}
  $\rho  \in L^\infty(]0,T[\times\Omega)$ for the initial data $\rho_0$ we have:
\begin{equation}
  \partial_t \beta(\rho) + \dive(\beta(\rho)\bfu) = 0
\end{equation}
in a weak sense with $\beta(\rho)_{|_{t=0}} = \beta(\rho_0)$, and as a consequence
$\beta(\rho)$ is solution of \eqref{solu-faible-transport} for the initial data $\beta(\rho_0)$.
\end{theorem}

\begin{theorem}[Unicity and regularity of the solution]\label{unicite-transport}
  Let $\Omega$ be an open bounded domain $\subset \R^d$ with Lipschitz border,
  $\rho_0 \in L^\infty(\Omega)$ and $\bfu \in L^2((0,T),\bfH^1_0(\Omega))$ with
  $\dive \bfu = 0$. If $\rho  \in L^\infty(]0,T[\times\Omega)$ is solution to
  \eqref{solu-faible-transport}  for the initial data $\rho_0$, then such a solution is unique and:
  \[
  \rho \in C^0([0,T],L^q(\Omega)) \quad \forall q < + \infty.
  \]
\end{theorem}

\section[Discrete functional analysis]{Discrete functional analysis}\label{annex:compac}


\begin{lemma}[{\cite[Lemme 3.3]{FV-book}}]\label{compac:astuce}
  Let $\Omega$ be an open bounded domain of $R^d$, $d=2,3$. Let $\mathcal{D} = (\mesh, \edges)$
  be an admissible MAC grid according to Definition \ref{def:MACgrid} and an element
  $u \in \Hmeshizero$ with $i \in \llbracket1,d\rrbracket$. Thus we define $\tilde{u}$ in the
  following manner : $\tilde{u} = u \quad \mbox{a.e. in } \Omega  $
  and $\tilde{u} = 0 \quad \mbox{a.e. in } \R^d\setminus\Omega $. Then, there exists $C>0$
  dependant only  on $\Omega$ such that,
  \[
  \| \tilde{u}(\cdot + \eta) - \tilde{u}\|^2_{L^2(\R^d)}
  \le \| u \|^2_{1,\edgesi,0} |\eta|(|\eta| + C h_\mesh), \quad \forall \eta \in \R^d.
  \]
\end{lemma}
\bop
\medskip
\begin{theorem}[{\cite[Theorem 3.10]{FV-book}}]\label{compac:space}
  Let $\Omega$ be an open bounded domain of $\R^N$ with Lipschitz boundary, $N \geq 1$,
  and $\{u_n, n \in \mathbb{N}\}$ a bounded sequence of $L^2(\Omega)$.
  Let us define the sequence $\{\tilde{u}_n, n \in \mathbb{N}\}$ as
  $\tilde{u}_n = u_n  \quad \mbox{a.e. on } \Omega$ and $\tilde{u}_n = 0  \quad \mbox{a.e. on }\mathbb{R}^N\setminus \Omega$.
  Assuming the existence of a constant $C \in \mathbb{R}$ et $\{h_n, n \in \mathbb{N}\}\subset \mathbb{R_+}$
  with $h_n \rightarrow 0$ and $n \rightarrow \infty$ with,
\begin{equation}\label{theo-compact:translates}
  \| \tilde{u}_n(\cdot + \eta) - \tilde{u}_n\|^2_{L^2(\mathbb{R}^N)}\leq
  C |\eta |(|\eta| + h_n), \quad \forall n \in \mathbb{N}, \forall \eta \in \mathbb{R}^N.
\end{equation}
Then, $\{u_n, n \in \mathbb{N}\}$ is relatively compact in $L^2(\Omega)$.
\end{theorem}
\medskip

\begin{definition}[Compactly embedded sequence]\label{compact-embedded-sequence}
  Let $B$ be a Banach space and $(X_n)_{n\in\xN}$ be a sequence of Banach spaces included in $B$.
  The sequence $(X_n)_{n\in\xN}$ is said to be compactly embedded in $B$ if any sequence satisfying:
\begin{itemize}
  \item $u_n\in X_n$ for all $n\in\xN$
  \item the sequence $(\|u_n\|_{X_n})_{n\in\xN}$ is bounded
\end{itemize}
is relatively compact in B.
\end{definition}

\begin{theorem}[Time compactness with a sequence of subspaces, \cite{book-gallouet-14} Proposition 4.48]\label{theo-compact-space}
  \qquad .\\
  We set $1 \le p < \infty$ and $T>0$. Let $B$ be a Banach space and $(X_n)_{n\in \xN}$
  be a sequence of Banach spaces compactly embedded in $B$. Let $(f_n)_{n\in \xN}$ be a sequence
  of $L^p((0,T);B)$ satisfying the following conditions:
\begin{itemize}
\item The sequence $(f_n)_{n\in \xN}$ in bounded in $L^p((0,T);B)$.
\item The sequence $(\|f_n\|_{L^1(0,T;X_n)})_{n\in \xN}$ is bounded.
\item There exists a non-decreasing function $\eta$ from $(0,T)$ to $\R_{+}$ such that $\lim_{h\rightarrow0} \eta(h) = 0$ and, $\forall h \in (0,T)$ and $n\in \xN$, we have:
\begin{equation}
  \int_0^{T-h} \| f_n(t+h) -f_n(t) \|^p_B \dt \le \eta(h)
\end{equation}
\end{itemize}
Then, $(f_n)_{n\in \xN}$ is relatively compact in $L^p((0,T);B)$.
\end{theorem}

\end{document}